\documentclass[12pt,leqno]{article}

\usepackage{amssymb,amsmath,amsthm}
\usepackage[all]{xy}

\usepackage{enumerate}
\usepackage{mathrsfs}
\usepackage{graphics}
\usepackage[usenames,dvipsnames]{xcolor}
\usepackage[colorlinks=true, pdfstartview=FitV, linkcolor=blue,%
citecolor=blue, urlcolor=blue]{hyperref}
\usepackage{fancyheadings}
\newcommand{\nc}{\newcommand}
\nc{\on}{\operatorname}

\newcommand{\spa}{\vspace{0.3ex}\noindent}

\newlength{\my}
\nc{\noi}{\noindent}

\renewcommand{\Re}{\operatorname{Re}}

\newtheorem{theorem}{Theorem}[subsection]
\newtheorem{proposition}[theorem]{Proposition}
\newtheorem{lemma}[theorem]{Lemma}
\newtheorem{corollary}[theorem]{Corollary}

\theoremstyle{definition}

\newtheorem{definition}[theorem]{Definition}
\newtheorem{notation}[theorem]{Notation}
\newtheorem{example}[theorem]{Example}

\newtheorem{remark}[theorem]{Remark}

\newtheorem{conjecture}[theorem]{Conjecture}

\nc{\Rem}{\begin{remark}}
\nc{\enrem}{\end{remark}}
\nc{\Conj}{\begin{conjecture}}
\nc{\enconj}{\end{conjecture}}
\nc{\Th}{\begin{theorem}}
\nc{\enth}{\end{theorem}}
\nc{\Lemma}{\begin{lemma}}
\nc{\enlemma}{\end{lemma}}
\nc{\Cor}{\begin{corollary}}
\nc{\encor}{\end{corollary}}
\nc{\Def}{\begin{definition}}
\nc{\edf}{\end{definition}}

\nc{\RR}{\mathrm{R}}
\nc{\LL}{\mathrm{L}}

\newcommand{\BBA}{{\mathbb{A}}}
\newcommand{\C}{{\mathbb{C}}}

\newcommand{\R}{{\mathbb{R}}}

\newcommand{\Z}{{\mathbb{Z}}}

\newcommand{\BBP}{{\mathbb P}}

\newcommand{\BBD}{{\mathbb D}}

\newcommand{\BBH}{{\mathbb H}}
\newcommand{\BBHd}{{\mathbb H}^*}

\newcommand{\PR}{{\BBP^1(\R)}}

\newcommand{\cor}{{\mathbf{k}}}

\newcommand{\icor}{\mathrm{I}\mspace{1mu}\cor}

\newcommand{\JA}{\mathrm{I}{\A}}
\newcommand{\JAop}{{\mathrm{I}{\A^\op}}}
\newcommand{\JD}{{\mathrm{I}}\mspace{1mu}{\D}}

\newcommand{\JfD}{\mathrm{I}{\opb{f}\D}}

\newcommand{\JR}{{\mathrm{I}}{\shr}}
\newcommand{\JfR}{{\mathrm{I}}{\opb{f}\shr}}
\newcommand{\JgR}{\mathrm{I}{\opb{g}\shr}}

\newcommand{\iC}{\mathrm{I}\mspace{.5mu}{\C}}
\newcommand{\iD}{\mathrm{I}{\D}}

\nc{\tA}{\mathcal{A}}
\nc{\rA}{\mathscr{A}}

\newcommand{\iDrm}{\mathrm{I}{\Drm}}

\def\D{\mathscr{D}}
\def\DA{\D^\tA}

\def\sha{\mathscr{A}}
\def\shb{\mathscr{B}}
\def\shc{\mathscr{C}}
\def\shd{\mathscr{D}}
\def\she{\mathscr{E}}
\def\shf{\mathscr{F}}

\def\shi{\mathscr{I}}

\def\shk{\mathscr{K}}
\def\shl{\mathscr{L}}
\def\shm{\mathscr{M}}
\def\shn{\mathscr{N}}
\def\sho{\mathscr{O}}

\def\shr{\mathscr{R}}
\def\shs{\mathcal{S}}
\def\sht{\mathscr{T}}

\def\shv{\mathscr{V}}
\def\shw{\mathscr{W}}

\newcommand{\rmpt}{\mathrm{pt}}

\newcommand{\into}{\hookrightarrow}

\renewcommand{\to}[1][]{\xrightarrow[]{#1}}
\newcommand{\from}[1][]{\xleftarrow[]{#1}}
\newcommand{\isoto}[1][]{\xrightarrow[#1]%
{{\raisebox{-.6ex}[0ex][-.6ex]{$\mspace{1mu}\sim\mspace{2mu}$}}}}
\newcommand{\isofrom}[1][]{\xleftarrow[#1]%
{{\raisebox{-.6ex}[0ex][-.6ex]{$\mspace{1mu}\sim\mspace{2mu}$}}}}

\newcommand{\To}[1][\rule{1ex}{0pt}]{\xrightarrow{\hs{.6ex}#1\hs{.6ex}}}

\newcommand{\muHom}[1][]{\mathrm{Hom}^\mu_{\raise1.5ex\hbox to.1em{}#1}}
\newcommand{\Hom}[1][]{\mathrm{Hom}_{\raise1.5ex\hbox to.1em{}#1}}
\newcommand{\RHom}[1][]{\RR\mathrm{Hom}_{\raise1.5ex\hbox to.1em{}#1}}
\newcommand{\Ext}[2][]{\mathrm{Ext}_{\raise1.5ex\hbox to.1em{}#1}^{#2}}
\renewcommand{\hom}[1][]{{\mathscr{H}\mspace{-4mu}om}_{\raise1.5ex\hbox to.1em{}#1}}
\newcommand{\rhom}[1][]{{\RR\mathscr{H}\mspace{-3mu}om}_{\raise1.5ex\hbox to.1em{}#1}}
\newcommand{\rhomc}[1][]
{{\mathscr{H}\mspace{-3mu}om}^*_{\raise1.5ex\hbox to.1em{}#1}}

\newcommand{\cihom}[1][]
{{\mathscr{I}\mspace{-3mu}}{hom}^+_{\raise1.5ex\hbox to.1em{}#1}}

\nc{\ihom}[1][]{{\shi\mspace{-3mu}hom}_{\raise1.5ex\hbox to.1em{}#1}}
\nc{\rihom}[1][]{{\mspace{2mu}\mathrm{R}\shi\mspace{-3mu}hom}_{\raise1.5ex\hbox to.1em{}#1}}
\nc{\fihom}[1][]{{\shi\mspace{-3mu}hom}^{\mathrm{E}}_{\raise1.5ex\hbox to.1em{}#1}}
\nc{\FHom}[1][]{{\mathrm{RHom}^{\mathrm{E}}_{\raise1.5ex\hbox to.1em{}#1}}}
\nc{\fhom}[1][]{{\mathscr{H}%
\mspace{-3mu}om}^{\mathrm{E}}_{\raise1.5ex\hbox to.1em{}#1}}
\nc{\Endom}[1][]{{\she\mspace{-3mu}nd}_{\raise1.5ex\hbox to.1em{}#1}}

\newcommand{\eeim}[1]{{#1}_{!!}}
\newcommand{\ssubset}{\subset\mspace{-3mu}\subset}
\newcommand{\A}{\mathcal{A}}
\newcommand{\B}{\mathcal{B}}

\nc{\Tam}{{\mathrm{E}}}

\newcommand{\ext}[2][]{{\mathscr{E}xt}_{\raise1.5ex\hbox to.1em{}#1}^{#2}}
\newcommand{\Tor}[2][]{\mathrm{Tor}^{\raise1.5ex\hbox to.1em{}#1}_{#2}}
\newcommand{\tens}[1][]{\mathbin{\otimes_{\raise1.5ex\hbox to-.1em{}{#1}}}}
\newcommand{\ltens}[1][]{\mathbin{\overset{\mathrm{L}}\tens}_{#1}}

\newcommand{\etens}{\mathbin{\boxtimes}}

\newcommand{\Endo}[1][]{\mathrm{End}_{\raise1.5ex\hbox to.1em{}#1}}

\newcommand{\Aut}[1][]{\mathrm{Aut}_{\raise1.5ex\hbox to.1em{}#1}}
\newcommand{\sect}{\Gamma}
\newcommand{\rsect}{\mathrm{R}\Gamma}

\newcommand{\conv}[1][]{\mathop{\circ}\limits_{#1}}
\newcommand{\econv}[1][]{\mathop{\circ}\limits^{\mathrm{E}}\limits_{#1}}

\newcommand{\ctens}[1][]{\mathbin{\overset{+}\tens}_{#1}}

\newcommand{\cconv}[1][]{\mathop{\circ}\limits^{#1}}
\newcommand{\Dconv}{\cconv[{\mathrm D}]}

\newcommand{\cetens}[1][]{\mathop{\etens}\limits^{+}\limits_{#1}}

\newcommand{\VV}{{\mathsf{V}}}
\newcommand{\VVd}{{\mathsf{V}^*}}

\newcommand{\WW}{{\mathbb{V}}}
\newcommand{\WWd}{{\mathbb{V}^*}}

\newcommand{\PP}{\mathsf{P}}
\newcommand{\Drm}{\mathsf{D}}

\newcommand{\oim}[1]{{#1}_*}
\newcommand{\eim}[1]{{#1}_!}
\newcommand{\roim}[1]{\RR{#1}_*}
\newcommand{\reim}[1]{\RR{#1}_!}

\newcommand{\reeim}[1]{\RR{#1}_{\mspace{1mu}!!}}
\newcommand{\opb}[1]{#1^{-1}}

\newcommand{\epb}[1]{#1^{\,!}\,}

\newcommand{\Dtens}[1][]{\overset{\mathrm{D}}\otimes_{\raise1.5ex\hbox to-.1em{}#1}}
\newcommand{\Detens}[1][]{\overset{\mathrm{D}}\etens_{\raise1.5ex\hbox to-.1em{}#1}}
\newcommand{\Ddual}{{\BBD}}

\newcommand{\Deim}[1]{{\mathrm{D}}{\mspace{-1mu}#1_{\mspace{1mu}!}}}
\newcommand{\Dopb}[1]{{\mathrm{D}}{#1}^{*}}
\newcommand{\Doim}[1]{{\mathrm{D}}{#1}_{*}}

\newcommand{\good}{\mathrm{good}}
\newcommand{\qgood}{\mathrm{q\text-good}}
\newcommand{\ghol}{\mathrm{g\text-hol}}

\nc{\rE}{\mathrm{E}}
\nc{\enh}{\mathsf{E}}
\newcommand{\dual}{\mathrm{D}}

\newcommand{\Edual}{\dual^\rE}

\newcommand{\Toim}[1]{\enh{#1}_*}
\newcommand{\Teeim}[1]{\enh{#1}_{\mspace{1mu}!!}}
\newcommand{\Topb}[1]{\enh\mspace{2mu}#1^{-1}}
\newcommand{\Tepb}[1]{\enh\mspace{2mu}#1^{\,!}}

\nc{\EF}[1][]{{}^{\enh}\mspace{-3mu}\shf_{#1}}
\nc{\EFa}[1][]{{}^{\enh}{\mspace{-3mu}\shf^a_{#1}}}
\nc{\FS}[1][]{{}^{\mathrm{S}}{\mspace{-3mu}\shf_{#1}}}
\nc{\FSa}[1][]{{}^{\mathrm{S}}{\mspace{-3mu}\shf^a_{#1}}}
\nc{\Leg}[1][]{{\mathrm{Conv}}{(#1)}}
\nc{\dom}{\mathrm{dom}}
\nc{\domo}{\dom^\circ}

\nc{\hol}{\mathrm{hol}}

\nc{\Mod}{\mathrm{Mod}}
\nc{\rh}{\mathrm{rh}}
\newcommand{\md[1]}{\Mod({#1})}

\newcommand{\mdf[1]}{\Mod^{f}({#1})}

\newcommand{\mdc[1]}{{\Mod_{\coh}}({#1})}

\newcommand{\mdhol[1]}{\Mod_{\hol}({#1})}
\newcommand{\mdrh[1]}{\Mod_{\rh}({#1})}
\newcommand{\mdrc[1]}{\Mod_{\Rc}({#1})}

\newcommand{\Perv}{{\mathrm{Perv}}}

\nc{\sHH}{\mathscr{H}\mspace{-4mu}\mathscr{H}}

\nc{\sMH}{\mathscr{M}\mspace{-4mu}\mathscr{H}}

\newcommand{\eqdot}{\mathbin{:=}}
\newcommand{\seteq}{\mathbin{:=}}

\newcommand{\cl}{\colon}
\newcommand{\scbul}{{\,\raise.4ex\hbox{$\scriptscriptstyle\bullet$}\,}}

\newcommand{\tw}[1]{\widetilde{#1}}

\newcommand{\twX}{{\widetilde{X}}}

\newcommand{\ol}{\overline}
\newcommand{\bl}{\bigl(}
\newcommand{\br}{\bigr)}
\newcommand{\lp}{{\rm(}}
\newcommand{\rp}{{\rm)}}

\newcommand{\Cc}{{\C\text{-c}}}
\newcommand{\Rc}{{\R\text{-c}}}
\newcommand{\Irc}{{\mathrm{I}\mspace{1mu}\R\text{-c}}}

\newcommand{\Sol}{\mspace{1mu}{\shs\mspace{-2.5mu}\mathit{ol}}}

\newcommand{\ba}{\begin{array}}
\newcommand{\ea}{\end{array}}

\nc{\be}{\begin{enumerate}}
\nc{\ee}{\end{enumerate}}
\newcommand{\bnum}{\begin{enumerate}[{\rm(i)}]}
\newcommand{\enum}{\end{enumerate}}
\newcommand{\banum}{\begin{enumerate}[{\rm(a)}]}
\newcommand{\eanum}{\end{enumerate}}

\newcommand{\bna}{\be[{\rm(a)}]}

\newcommand{\eq}{\begin{eqnarray}}
\newcommand{\eneq}{\end{eqnarray}}
\newcommand{\eqn}{\begin{eqnarray*}}
\newcommand{\eneqn}{\end{eqnarray*}}

\newcommand{\set}[2]{\left\{#1 \mathbin{;} #2 \right\}}

\nc{\Proof}{\begin{proof}}
\nc{\QED}{\end{proof}}
\nc{\Prop}{\begin{proposition}}
\nc{\enprop}{\end{proposition}}

\nc{\rop}{{\mathrm{op}}}
\nc{\op}{\rop}
\nc{\tot}{\mathrm{tot}}
\nc{\Op}{{\mathrm{Op}}}
\nc{\dist}{{\mathrm{dist}}}
\nc{\LocSyst}{{\mathrm{LocSyst}}}
\nc{\eu}{\mathrm{eu}}
\nc{\hh}{\mathrm{hh}}
\nc{\mueu}{{\mu\eu}}

\DeclareMathOperator{\id}{id}

\DeclareMathOperator{\ori}{{or\mspace{2mu}}}
\DeclareMathOperator{\chv}{char}

\newcommand{\Supp}{\on{Supp}}
\newcommand{\Der}[1][]{\mathsf{D}^{#1}}
\newcommand{\Derb}{\Der[\mathrm{b}]}

\newcommand{\SSupp}{\on{SingSupp}}

\newcommand{\RD}{\mathrm{D}}
\newcommand{\Derp}{\Der[+]}
\newcommand{\Derm}{\Der[-]}

\newcommand{\Int}{{\on{Int}}}
\newcommand{\coh}{\mathrm{coh}}

\newcommand{\gr}{{\mathrm{gr}}}

\newcommand{\dT}{{\dot{T}}}

\newcommand{\II[1]}{\mathrm{I}({#1})}
\newcommand{\IIf[1]}{\mathrm{I}^{f}({#1})}

\newcommand{\SI[1]}{{\mathcal I}({#1})}

\newcommand{\IIrc[1]}{\mathrm{I}_{\Rc}({#1})}

\newcommand{\indcc}{\mathrm{Ind}(\shc)}
\newcommand{\indc}{\mathrm{IC}}

\newcommand{\mdcp[1]}{\Mod^{\,\mathrm{c}}({#1})}
\newcommand{\BDC}{\Derb}
\newcommand{\TDC}{{\mathrm{E}}^{\mathrm{b}}}
\newcommand{\TDCpm}{{\mathrm{E^b_\pm}}}
\newcommand{\TDCmp}{{\mathrm{E^b_\mp}}}
\newcommand{\TDCp}{{\mathrm{E^b_+}}}
\newcommand{\TDCm}{{\mathrm{E^b_-}}}
\newcommand{\TDCC}{\mathrm{E}}

\newcommand{\mop}{\mathrm{r}}

\newcommand{\Tmp}{\mathsf{T}}
\newcommand{\OEn}{\sho^{\mspace{2mu}\enh}}
\newcommand{\DbT}{\Db^\Tmp}

\newcommand{\OvE}{\Omega^\enh}
\newcommand{\drE}{\mathcal{DR}^\enh}
\newcommand{\solE}{\mspace{1mu}\mathcal{S}ol^{\mspace{1mu}\enh}}

\nc{\wc}[1]{\overset{\mbox{$\scriptscriptstyle\vee$}}{#1}}
\nc{\field}{\cor}
\nc{\bM}{\widehat M}
\nc{\bN}{\widehat N}
\nc{\bX}{{\widehat X}}
\nc{\bS}{\widehat S}
\nc{\bY}{\widehat Y}
\nc{\bL}{\widehat L}
\nc{\bR}{{\ol\R}}
\nc{\bV}{{\ol \VV}}
\nc{\bW}{{\ol \WW}}

\nc{\bVd}{{\ol \VVd}}
\nc{\bWd}{{{\ol \WW}^*}}

\newcommand{\fM}{{M_\infty}}

\newcommand{\fR}{{\R_\infty}}
\newcommand{\fV}{{\VV_\infty}}
\newcommand{\fVd}{{\VV^*_\infty}}
\newcommand{\fW}{{\WW_\infty}}
\newcommand{\fWd}{{\WW^*_\infty}}

\nc{\oM}{{\ol M}}
\nc{\oN}{\ol N}
\nc{\oX}{{\ol X}}
\nc{\oS}{\ol S}
\nc{\oY}{\ol Y}
\nc{\oL}{\ol L}
\nc{\oR}{{\ol\R}}
\nc{\Tl}{\mathrm{L^E}}
\nc{\Tr}{\mathrm{R^E}}

\nc{\sa}{\mathrm{sa}}

\newcommand{\Msa}{{M_{\sa}}}

\newcommand{\Db}{{\mathcal D} b}
\newcommand{\Dbt}{{\mathcal D} b^{\mspace{2mu}\mathrm t}}
\newcommand{\Dbtv}{{\mathcal D} b^{\mathrm t\vee}}
\newcommand{\Cinft}[1][M]{\mathcal{C}^{\infty,\mathrm t}_{#1}}
\newcommand{\Cinfw}[1][M]{\mathcal{C}^{\infty,\mathrm w}_{#1}}
\newcommand{\Cinf}[1][M]{\mathcal{C}^{\infty}_{#1}}
\newcommand{\Cinfom}[1][M]{\mathcal{C}^{\mspace{1mu}\omega}_{#1}}

\newcommand{\thom}{Thom}
\newcommand{\wtens}{\mathbin{\overset{\mathrm{w}}\tens}}
\newcommand{\Ot}[1][X]{\sho^{\mspace{2.5mu}{\mathrm t}}_{#1}}
\newcommand{\OO}[1][X]{\sho_{#1}}
\newcommand{\Ow}[1][X]{\sho^{\mspace{2mu}\mathrm w}_{#1}}
\newcommand{\Oww}[1][X]{\sho^{\mspace{2mu}\omega}_{#1}}
\newcommand{\rc[1]}{{\Brc}({#1})}
\newcommand{\rcc[1]}{{\Brc}^{\,\mathrm{c}}({#1})}
\nc{\BRC}{{\R}\hbox{-}{\mathrm{Cons}}}
\nc{\Brc}{{\R}\hbox{-}{\mathrm{C}}}
\newcommand{\Ovt}{\Omega^{\mspace{1.5mu}{\mathrm t}}}

\newcommand{\dr}{\mathcal{DR}}
\newcommand{\drt}{\mathcal{DR}^{\mathrm t}}

\newcommand{\sol}{\mathcal Sol}

\newcommand{\solt}{\mathcal Sol^{\mspace{2.5mu}\mathrm t}}
\newcommand{\reghol}{\rh}

\newcommand{\Sp}{\mathrm{Sp}}

\newcommand{\indlim}[1][]{\mathop{\varinjlim}\limits_{#1}}
\newcommand{\sindlim}[1][]{\smash{\mathop{\varinjlim}\limits_{#1}}\,}

\newcommand{\Ind}{\mathrm{Ind}}
\newcommand{\prolim}[1][]{\mathop{\varprojlim}\limits_{#1}}
\newcommand{\sprolim}[1][]{\smash{\mathop{\varprojlim}\limits_{#1}}\,}

\newcommand{\inddlim}[1][]{\mathop{\text{\rm``{$\varinjlim$}''}}\limits_{#1}}
\newcommand{\sinddlim}[1][]{\smash{\mathop{\text{\rm``{$\varinjlim$}''}}\limits_{#1}}\,}

\newcommand{\dsum}[1][]{\mathbin{\oplus_{#1}}}
\nc{\eps}{\varepsilon}
\nc{\hs}{\hspace*}
\nc{\nn}{\nonumber}
\nc{\tM}{\widetilde{M}}
\nc{\h}{\mathbf{h}}
\nc{\tf}{\tilde{f}}
\nc{\trf}{{{}^{\mathrm{t}}\mspace{-3mu}f}}
\nc{\codim}{\on{codim}}
\nc{\lh}{\mathscr{H}}
\nc{\bwr}{\scalebox{1.1}{$\wr$}}
\nc{\dTi}{\dT^{*,\mathrm{in}}}
\nc{\Cd}{\mathrm{C}}
\nc{\tK}{\widetilde{K}}
\nc{\aMM}{a_{M\times M}}
\nc{\e}{\mspace{1mu}\mathrm{e}\mspace{1mu}}
\nc{\lan}{\langle}
\nc{\ran}{\rangle}
\nc{\la}{\lambda}
\newcommand{\cross}[1][]{\mathbin{\times_{#1}}}
\newcommand{\Yes}{$\circ$}
\newcommand{\No}{$\cross$}
\newcommand{\Union}{\bigcup\limits}

\nc{\vphi}{\varphi}
\nc{\vep}{\varepsilon}

\nc{\At}{\tA_\twX}
\nc{\bu}{\boldsymbol{u}}
\nc{\bA}{\boldsymbol{A}}
\nc{\hbu}{\widehat{\boldsymbol{u}}}
\nc{\ex}{\mathrm{e}}
\nc{\vpi}{\varpi}
\newcommand{\soplus}{\mathop{\mbox{\small $\bigoplus$}}\limits}
\nc{\one}{\mathbf{1}}
\nc{\setp}[1]{\{#1\}}
\nc{\GL}{\mathrm{GL}}
\nc{\cI}{\mathrm{I}}
\newcommand{\TDCrc}{{\mathrm{E^b_{\Rc}}}}
\nc{\vs}{\vspace*}

\nc{\wb}[1]{\mbox{$\rule[-1.1ex]{0ex}{2ex}#1$}}
\nc{\wwb}[1]{\mbox{$\rule[-1.8ex]{0ex}{3ex}#1$}}
\nc{\bpi}{\ol{\pi}}
\nc{\Tsupp}{\Supp^{\mathrm E}}

\nc{\abu}{\Vec{\bu}}
\nc{\av}{\Vec{v}}
\nc{\au}{\Vec{u}}
\newcommand{\OOO}{\mathsf{O}}

\nc{\FN}{\mathrm{FN}}
\nc{\DFN}{\mathrm{DFN}}

\nc{\ake}{\hs{.25ex}}

\nc{\va}{\Vec{a}}
\nc{\bal}{\begin{align}}
\nc{\aal}{\end{align}}
\nc{\baln}{\begin{align*}}
\nc{\ealn}{\end{align*}}

\usepackage{graphics}

\usepackage{makeidx}

\numberwithin{equation}{subsection}

\begin{document}

\title{Regular and  irregular holonomic D-modules}
\author{Masaki Kashiwara and Pierre Schapira}

\date{July 1, 2015}

\maketitle
\begin{abstract}
This is a survey paper. In a first part,  we recall  the main results on the tempered holomorphic solutions of D-modules in the language of indsheaves and, as an application, the Riemann-Hilbert correspondence for regular holonomic modules. In a second part, 
we present the enhanced version of the first part, treating along the same lines the irregular holonomic case. 
\end{abstract}

\tableofcontents
\section*{Introduction}
\addcontentsline{toc}{section}{\numberline{}Introduction}
These Notes are an expanded version of a series of lectures 
given at the IHES in February/March 2015 (see~\cite{KS15}),
based on \cite{DK13} and \cite{KS14}. 

Here, we assume the reader familiar with the language of sheaves and D-modules, in the derived sense. 

Let $X$ be a complex manifold. Denote by $\md[\shd_X]$ the abelian category of left $\shd_X$-modules, by 
$\mdhol[\shd_X]$ the full subcategory of holonomic $\shd_X$-modules  and by $\Perv(\C_X)$ the abelian category of perverse sheaves with coefficients in $\C$. Consider the functor constructed in~\cite{Ka75}
\eqn
&&\Sol\cl \mdhol[\shd_X]^\rop\to\Perv(\C_X),\\
&&\shm\mapsto \rhom[\shd](\shm,\OO).
\eneqn
(Note that at this time the notion of perverse sheaves was not explicit, but in his paper the author proved that 
$\rhom[\shd](\shm,\OO)$ is $\C$-constructible and satisfies the properties which are now called perversity.)

It is well-known that this functor is not faithful. For example, if $X=\BBA^1(\C)$, the complex line with coordinate $t$, $P=t^2\partial_t-1$ and $Q=t^2\partial_t+t$, then the two $\shd_X$-modules 
$\shd_X/\shd_X P$ and $\shd_X/\shd_X Q$  have 
the same sheaves of solutions.

A natural idea to overcome this difficulty is to replace the sheaf  $\OO$ with presheaves of holomorphic functions with various growths such as 
for example the presheaf $\Ot$ of holomorphic functions with tempered growth. This presheaf is not  a sheaf for the usual topology, but it becomes  a sheaf for a suitable Grothendieck topology, the subanalytic topology,  and here we shall  embed the category of subanalytic sheaves in that of indsheaves. 

As we shall see, the indsheaf $\Ot$ is not sufficient  to obtain a Riemann-Hilbert correspondence, but it is a first step to this direction. To obtain a final result, it is necessary to add an extra variable and to work with an ``enhanced'' version of $\Ot$ in order to describe ``various growths'' in a rigorous way.

In a first part, we shall recall the main results of the theory of indsheaves and subanalytic sheaves and we shall 
explain with some details the operations on D-modules and their tempered holomorphic solutions. As an application, we obtain the Riemann-Hilbert correspondence for regular holonomic D-modules as well as the fact that the de Rham functor commutes with integral transforms.

In a second part, we do the same for the sheaf of enhanced tempered solutions of (no more necessarily regular) holonomic D-modules. For that purpose, we first recall the main results of the theory of indsheaves on bordered spaces and its enhanced version,
a generalization to indsheaves of a construction of Tamarkin~\cite{Ta08}. 

Let us describe with some details the contents of these Notes.

\medskip\noi
{\bf Section~\ref{section:review}} is a brief review on the theory of sheaves and D-modules. Its aim is essentially to fix the notations and to recall the main formulas of constant use. 

\medskip\noi
In {\bf Section~\ref{section:indshv}}, extracted from~\cite{KS96,KS01}, 
we briefly describe the category of indsheaves on a locally compact space and the six operations on indsheaves. A method for constructing indsheaves on a subanalytic space is the use of  the subanalytic 
Grothendieck topology, a topology  for which the open sets are the open relatively compact subanalytic subsets and the coverings are the finite coverings. On a real analytic manifold $M$,  this allows us to construct the indsheaves of Whitney functions, tempered 
$\mathrm{C}^\infty$-functions and  tempered distributions. On a complex manifold $X$, by taking the Dolbeault complexes with such coefficients, we obtain the indsheaf (in the derived sense) $\Ow$ of Whitney holomorphic functions and the indsheaf $\Ot$  of tempered holomorphic functions.

\medskip\noi
Then, in {\bf Section~\ref{section:tempered}}, also extracted from~\cite{KS96,KS01}, we study the tempered de Rham and Sol (Sol for solutions) functors, that is, we study these functors with values in the sheaf of tempered holomorphic functions. We prove two main results  which will be the main tools to treat the regular Riemann-Hilbert correspondence later. The first one is Theorem~\ref{thm:ifunct0} which calculates the inverse image of the tempered de Rham complex. It is a reformulation of a theorem of~\cite{Ka84}, a vast generalization of the famous Grothendieck theorem on the de Rham cohomology of algebraic varieties. The second result, Theorem~\ref{th:tGrauert}, is a tempered version of  the  Grauert direct image theorem. 

\medskip\noi
In {\bf Section~\ref{section:reghol}} we give a  proof  of the main theorem of~\cite{Ka80,Ka84} on the  Riemann-Hilbert correspondence for regular holonomic D-modules (see Corollary~\ref{cor:RegRH3}). Our proof is based on Lemma~\ref{le:normaDR} which essentially claims that to prove that regular holonomic D-modules have a certain property, it is enough to check that this property 
is stable by  projective direct images and is satisfied by modules of ``regular normal forms'', that is, 
modules associated with equations of the type $z_i\partial_{z_i}-\la_i$ or $\partial_{z_j}$. The 
Riemann-Hilbert correspondence as formulated in loc.\ cit.\ is not enough to treat integral transform and we have to prove a ``tempered'' version of it (Theorem~\ref{thm:ifunct4}). 
We then  collect all results on the  tempered solutions of D-modules in a single formula which, roughly speaking, asserts that the tempered de Rham functor commutes with integral transforms whose kernel is regular holonomic (Theorem~\ref{th:7412}). We end this section with a detailed study of the irregular holonomic D-module 
$\shd_X\exp(1/z)$ on $\BBA^1(\C)$, following~\cite{KS03}. This case shows that the solution functor with values in the indsheaf $\Ot$ gives many informations on the holonomic D-modules, but not enough: it is not fully faithful.
As seen in the next sections, in order to treat irregular case, we need
the enhanced version of the setting discussed in this section.

\medskip\noi
{\bf Section~\ref{section:bordered}}, extracted from~\cite{DK13}, treats indsheaves on bordered spaces. A bordered space is a pair  $(M,\bM)$ of good topological spaces with $M\subset \bM$ an open embedding. The derived category of indsheaves on $(M,\bM)$ is the quotient of the category of indsheaves on 
$\bM$ by that of indsheaves on $\bM\setminus M$. Indeed, contrarily to the case of usual sheaves, this quotient is not equivalent to the derived category of indsheaves on $M$. 

The main idea to treat the irregular Riemann-Hilbert correspondence is to replace the indsheaf $\Ot$ with an enhanced version, the object $\OEn_X$. Roughly speaking, this object (which is no more an indsheaf) is obtained as the image 
of  the complex of solutions of the operator $\partial_t-1$ acting on $\Ot[X\times\C]$, in a suitable category, namely that of enhanced indsheaves.

 \medskip\noi
{\bf Section~\ref{section:enhanced}}, also extracted from~\cite{DK13}, defines and studies the triangulated category 
$\TDC(\icor_M)$  of enhanced indsheaves on $M$, adapting to indsheaves a construction of Tamarkin~\cite{Ta08}. Denoting by
$\R_\infty$ the bordered space  $(\R,\overline\R)$ in which $\ol\R$ is the two points compactification of $\R$, 
the category $\TDC(\icor_M)$ is the quotient of the category of indsheaves on $M\times\R_\infty$ by the subcategory 
of indsheaves which are 
isomorphic to the inverse image of indsheaves on $M$. 

\medskip\noi
{\bf Section~\ref{section:hol}}, mainly extracted from~\cite{DK13}, treats the irregular Riemann-Hilbert correspondence. Similarly as in the regular case, an essential tool is  Lemma~\ref{lem:redux} which asserts that to prove that  holonomic D-modules have a certain property, it is enough to check that this property 
is stable by  projective direct images and is satisfied by modules of ``normal forms'', that is, D-modules of the type 
$\shd_X\exp\vphi$ where $\vphi$ is a meromorphic function. This lemma follows directly from the  
fundamental results of Mochizuki~\cite{Mo09,Mo11} (in the algebraic setting) and later Kedlaya~\cite{Ke10,Ke11} 
in  the analytic case, after  preliminary results  by Sabbah~\cite{Sa00}.
The proof of the irregular Riemann-Hilbert correspondence is rather intricate and uses enhanced constructible sheaves 
and a duality result between
the enhanced solution functor and the enhanced de Rham functor. 
However, this theorem formulated in~\cite{DK13} (Corollary~\ref{cor:irregRH2}) is not 
enough to treat irregular  integral transform and we have to prove an ``enhanced'' version of it (Theorem~\ref{th:irrRH1}, extracted from~\cite{KS14}). 

\medskip\noi
In {\bf Section~\ref{section:IT}}, extracted from~\cite{KS14}, we apply the preceding results. 
The main formula~\eqref{eq:7413} asserts, roughly speaking, that the enhanced de Rham functor commutes with integral transforms with irregular kernels. 
In a previous paper~\cite{KS97} we had already proved (without the machinery of enhanced indsheaves) that given a complex vector space $\WW$, the Laplace transform induces an isomorphism of the Fourier-Sato transform of the conic sheaf associated with $\Ot[\WW]$ with the similar sheaf on $\WW^*$ (up to a shift). 
We obtain here a similar result in a non-conic setting, replacing $\Ot[\WW]$ 
with its enhanced version  $\OEn_\WW$. For that purpose, we extend first the Tamarkin non conic Fourier-Sato transform to the enhanced setting.

\medskip\noindent
{\bf Bibliographical and historical comments.}
A first important step in a modern  treatment of the Riemann-Hilbert correspondence is the 
book of Deligne~\cite{De70}. 
A second important step is the constructibility theorem~\cite{Ka75}
and a precise formulation of this correspondence in 1977 by the same author (see \cite[p.~287]{Ra78}).
Then a detailed sketch of proof of the theorem establishing this correspondence (in the regular case) appeared  in~\cite{Ka80} where the functor $\thom$ of tempered cohomology was introduced, and a detailed proof appeared in~\cite{Ka84}.  
 A different proof to this correspondence appeared in~\cite{Me84}. 
The functorial operations on the functor $\thom$, 
as well as its dual  notion, the Whitney tensor product $\wtens$,  are systematically studied in~\cite{KS96}. 
These two functors are in fact better understood
by the language of $\Ot$ and $\Ow$, the  indsheaves  
of tempered holomorphic functions and Whitney holomorphic functions
introduced in~\cite{KS01}.

In the early 2000, it became clear that the indsheaf $\Ot[X]$ of tempered holomorphic functions is an essential tool for  the study of irregular holonomic modules and a toy model was studied  in~\cite{KS03}. However, on 
$X=\BBA^1(\C)$, the two  holonomic $\shd_X$-modules 
$\shd_X\exp(1/t)$ and  $\shd_X\exp(2/t)$ have the same tempered holomorphic solutions,  which shows that $\Ot[X]$ is not precise enough to treat irregular holonomic D-modules. 
This difficulty is overcome in~\cite{DK13} by adding an extra variable 
in order to capture the growth at singular points.   This is done, first by adapting to indsheaves a construction of Tamarkin~\cite{Ta08}, leading to the notion of ``enhanced indsheaves'', then by defining the 
``enhanced indsheaf of tempered holomorphic functions''. 
Using fundamental results of Mochizuki~\cite{Mo09,Mo11} (see also Sabbah~\cite{Sa00} for preliminary results  and see
Kedlaya~\cite{Ke10,Ke11} for the analytic case),
this leads to the solution of  the Riemann-Hilbert correspondence for 
 (not necessarily regular)  holonomic D-modules.

As already mentioned, most of the results  discussed  here are already known. 
 We sometimes don't give proofs, or only give a sketch of the proof. 
However, 
Theorems~\ref{th:constdualDmod}, \ref{th:newconstruct} and Corollaries~\ref{cor:KSduality}, \ref{cor:newconstruct} are new. 

\section{A brief review on sheaves and D-modules}\label{section:review}

As already mentioned in the introduction, we assume the 
 reader familiar with the language of sheaves and D-modules, in the derived sense.
Hence, the aim of this section is mainly to fix some notations.

\subsection{Sheaves}\label{subsec:shv}
We refer to~\cite{KS90} for all notions of sheaf theory used here.
For simplicity, we denote by $\cor$ a field, although most of the results would remain true when $\cor$ is a
commutative ring of finite global dimension.  

A topological space is {\em good} 
\index{good!topological space}%
if it is Hausdorff, locally compact, countable at infinity and has finite flabby dimension. 
Let $M$ be such a space. For a subset $A\subset M$, we denote by $\ol A$ its closure and $\Int(A)$ its interior.

One denotes by $\md[\cor_M]$  
\glossary{$\md[\cor_M]$}%
the abelian category of sheaves of $\cor$-modules on $M$ and by 
$\Derb(\cor_M)$ 
\glossary{$\Derb(\cor_M)$}%
its bounded derived category.
Note that $\md[\cor_M]$ has a finite homological dimension. 

For a locally closed subset $A$ of $M$, one denotes by $\cor_A$ the constant sheaf on $A$ with stalk $\cor$ extended by $0$ on $X\setminus A$. For $F\in\Derb(\cor_M)$, one sets $F_A\eqdot F\tens\cor_A$. 
One denotes by $\Supp(F)$ the support of $F$. 
\glossary{$\Supp(F)$}%

We shall make use of the dualizing complex 
\index{dualizing complex}%
on $M$, denoted by $\omega_M\glossary{$\omega_M$}$,
and the duality functors
\eq\label{eq:dualfcts}
&&\RD'_M\eqdot\rhom(\scbul,\cor_M),\quad \RD_M\eqdot\rhom(\scbul,\omega_M).
\eneq
\glossary{$\RD'_M$}%
\glossary{$\RD_M$}%
Recall that, when $M$ is a real manifold, $\omega_M$ is isomorphic to the orientation sheaf shifted by the dimension.

We have the two internal operations of internal hom and tensor product:
\glossary{$\rhom$}\glossary{$\tens$}%
\eqn
\rhom(\scbul,\scbul)&\cl& \Derb(\cor_M)^\rop\times \Derb(\cor_M)\to \Derb(\cor_M),\\
\scbul\ltens\scbul&\cl&\Derb(\cor_M)\times \Derb(\cor_M)\to \Derb(\cor_M).
\eneqn
Hence,  $\Derb(\cor_M)$ has a structure of commutative tensor category with $\cor_M$ as unit object
and $\rhom$ is the inner hom of this tensor category. 

Now let $f\cl M\to N$ be a morphism of good topological spaces. One has the functors
\glossary{$\opb{f}$}\glossary{$\epb{f}$}\glossary{$\roim{f}$}\glossary{$\reim{f}$}%
\eqn
&&\opb{f}\cl\Derb(\cor_N)\to\Derb(\cor_M)\mbox{ inverse image},\\
&&\epb{f}\cl\Derb(\cor_N)\to\Derb(\cor_M)\mbox{ extraordinary  inverse image},\\
&&\roim{f}\cl\Derb(\cor_M)\to\Derb(\cor_N)\mbox{ direct image},\\
&&\reim{f}\cl\Derb(\cor_M)\to\Derb(\cor_N)\mbox{ proper direct image}.
\eneqn
We get the pairs of adjoint functors $(\opb{f},\roim{f})$ and $(\reim{f},\epb{f})$.

The operations associated with the functors $\tens,\rhom, \opb{f},\epb{f},\roim{f},\reim{f}$ are called  
\index{Grothendieck's six operations!for sheaves}%
Grothendieck's six operations. 

For two topological spaces $M$ and $N$, one defines  the functor of 
\index{external tensor product!for sheaves}%
external tensor product 
\glossary{$\etens$}%
\eqn
&&\scbul\etens\scbul\cl \Derb(\cor_M)\times\Derb(\cor_N)\to\Derb(\cor_{M\times N})
\eneqn
 by setting $F\etens G\eqdot\opb{q_1}F\tens\opb{q_2}G$, where $q_1$ and $q_2$ are the projections from $M\times N$ to $M$ and $N$, respectively.

Denote by $\rmpt$ the topological space with a single element  and by $a_M\cl M\to\rmpt$
\glossary{$a_M\cl M\to\rmpt$}%
the unique morphism.
One has the isomorphism 
\eqn
&&\cor_M\simeq\opb{a_M}\cor_\rmpt,\quad \omega_M\simeq\epb{a_M}\cor_\rmpt.
\eneqn
There are many important formulas relying the six operations. In particular we have the formulas 
below in which $F,F_1,F_2\in\Derb(\cor_M)$, $G,G_1,G_2\in \Derb(\cor_N)$:
\eqn
&&\rhom(F\tens F_1,F_2)\simeq\rhom\bl F,\rhom(F_1,F_2)\br,\\
&&\roim{f}\rhom(\opb{f}G,F)\simeq\rhom(G, \roim{f}F),\\
&&\reim{f}(F\tens\opb{f}G)\simeq F\tens\reim{f}G\quad\text{
(projection formula),\index{projection formula!for sheaves}}\\
&&\epb{f}\rhom(G_1,G_2)\simeq\rhom(\opb{f}G_1,\epb{f}G_2),
\eneqn
and for a Cartesian square of good topological spaces, 
\eq\label{eq:CartSq}
\ba{c}\xymatrix@C=8ex{
M' \ar[r]^{f'} \ar[d]^{g'} & N' \ar[d]^{g} \\
M \ar[r]^{f}\ar@{}[ur]|-\square & N
}\ea
\eneq
we have the {\em base change formula!for sheaves}
\index{base change formula!for sheaves}%
\eqn
&& \opb{g}\reim{f}\simeq \reim f'\opb{g'}.
\eneqn

In these Notes, we shall also encounter 
{\em $\R$-constructible} sheaves. References are made to~\cite[Ch.~VIII]{KS90}.
Let $M$ be a real analytic manifold. On $M$ there is the family of subanalytic sets due to Hironaka and Gabrielov
(see~\cite{BM88,VD98} for an exposition). 
\index{subanalytic!subset}%
This family is 
stable by all usual operations (finite intersection and  locally finite union, complement, closure, interior) and contains the family of semi-analytic sets (those locally defined by analytic inequalities). If $f\cl  M\to N$ is a morphism of real analytic manifolds, then the inverse image of a subanalytic set is subanalytic. If $Z$ is subanalytic in $M$ and $f$ is proper on the closure of $Z$, then $f(Z)$ is subanalytic in $N$. 

A sheaf $F$ is $\R$-\emph{constructible}
\index{R@$\R$-constructible!sheaves}%
if there exists a subanalytic stratification 
$M=\bigsqcup_{j\in J} M_j$ such that for each $j\in J$, the sheaf $F\vert_{M_j}$ is locally constant of finite rank. 
One defines the category 
\glossary{$\Derb_\Rc(\cor_M)$}%
$\Derb_\Rc(\cor_M)$ as the full subcategory of $\Derb(\cor_M)$ consisting of objects $F$ such that $H^i(F)$ is $\R$-constructible for all  $i\in\Z$  and one proves that this category is triangulated.

The category $\Derb_\Rc(\cor_M)$ is stable by the usual internal operations (tensor product, internal hom) and the duality functors in~\eqref{eq:dualfcts} induce anti-equivalences on this category. 

If $f\cl M\to N$ is a morphism of real analytic manifolds, then $\opb{f}$ and $\epb{f}$ send $\R$-constructible objects to $\R$-constructible objects. If $F\in\Derb_\Rc(\cor_M)$ and $f$ is proper on $\Supp(F)$, then 
$\reim{f}F\in\Derb_\Rc(\cor_N)$. 

\subsection{D-modules}\label{subsec:Dmod}
References for 
\index{D-module}%
D-module theory are made to~\cite{Ka03}. See also \cite{Ka70,Bj93,HTT08}. 

Here, we shall briefly recall some basic constructions in  the theory of D-modules that we shall use. Note that 
many classical functors that shall appear in this section will be extended to indsheaves in 
Section~\ref{section:tempered} and the subsequent sections.

In this subsection, the base field is the complex number field $\C$.

Let $(X,\OO)$ be a {\em complex} manifold. 
We denote as usual by 
\begin{itemize}
\item
$d_X$ the complex dimension of $X$,
\glossary{$d_X$}%
\item
$\Omega_X$ the invertible sheaf of differential forms of top degree,
\glossary{$\Omega_X$}%
\item
$\Omega_{X/Y}$ the invertible $\OO$-module $\Omega_X\tens[{\opb{f}\OO[Y]}]\glossary{$\Omega_{X/Y}$}%
\opb{f}(\Omega_Y^{\otimes-1})$ for a morphism $f\cl X\to Y$ of complex manifolds,
\item
$\Theta_X$ the sheaf of holomorphic vector fields,
\glossary{$\Theta_X$}%
\item 
 $\D_X$ the sheaf of algebras  of finite-order differential operators.
\glossary{$\D_X$}%
\end{itemize}

Denote by $\md[\D_X]$ 
\glossary{$\md[\D_X]$}%
\glossary{$\md[\D^\rop_X]$}%
the abelian category of left $\D_X$-modules and by  $\md[\D^\rop_X]$ that of  right $\D_X$-modules. There is an equivalence 
 \eq
 &&\mop\colon \md[\D_X] \isoto \md[\D_X^\op],\quad \shm\mapsto \shm^\mop\eqdot\Omega_X\tens[\OO]\shm. 
\eneq
By this equivalence, it is enough to study left $\D_X$-modules. 

The ring $\D_X$ is coherent and one denotes by $\mdc[\D_X]$ 
\glossary{$\mdc[\D_X]$}%
the thick abelian subcategory of $\md[\D_X]$ consisting of coherent modules. 
 
To a coherent $\D_X$-module $\shm$ one associates its characteristic variety $\chv(\shm)$, 
\glossary{$\chv(\shm)$}%
a closed $\C^\times$-conic 
{\em co-isotropic} (one also says {\em involutive}) $\C$-analytic subset of the cotangent bundle $T^*X$. 
The involutivity property is 
a central theorem of the theory and is due to~\cite{SKK73}. A purely algebraic proof was obtained later in~\cite{Ga81}.

If $\chv(\shm)$ is Lagrangian, $\shm$ is called {\em  holonomic}.
\index{D-module!holonomic}%
\index{holonomic D-module}%
 It is immediately checked that the full subcategory $\mdhol[\D_X]$ of $\mdc[\D_X]$ consisting of holonomic $\D$-modules is a thick abelian subcategory.
 
A $\D_X$-module $\shm$ is \emph{quasi-good} 
\index{D-module!quasi-good}%
\index{quasi-good D-module}%
if, for any relatively compact open subset 
$U\subset X$, $\shm\vert_U$ is a sum of coherent $(\OO\vert_U)$-submodules. A $\D_X$-module $\shm$ is  \emph{good} 
\index{D-module!good}%
\index{good!D-module}%
if it is quasi-good and coherent. The subcategories of $\md[\D_X]$ consisting of quasi-good (resp.\ good) $\D_X$-modules are abelian and thick. Therefore, one has  the triangulated categories
\begin{itemize}
\item
$\BDC_\coh(\D_X)=\set{\shm\in\Derb(\D_X)}%
{\text{$H^j(\shm)$ is coherent for all $j\in\Z$}}$,
\glossary{$\BDC_\coh(\D_X)$}%
\item
$\BDC_\hol(\D_X)=\set{\shm\in\Derb(\D_X)}%
{\text{$H^j(\shm)$ is holonomic for all $j\in\Z$}}$,
\glossary{$\BDC_\hol(\D_X)$}%
\item 
$\BDC_\qgood(\D_X)=\set{\shm\in\Derb(\D_X)}{\text{$H^j(\shm)$ is quasi-good  for all $j\in\Z$}}$,
\glossary{$\BDC_\qgood(\D_X)$}%
\item
$\BDC_\good(\D_X)=\set{\shm\in\Derb(\D_X)}{\text{$H^j(\shm)$ is good  for all $j\in\Z$}}$.
\glossary{$\BDC_\good(\D_X)$}%
\end{itemize}
 One may also consider the unbounded derived categories $\RD(\D_X)$, 
$\RD^-(\D_X)$   and $\RD^+(\D_X)$ and the full triangulated subcategories consisting of coherent, holonomic, quasi-good and good modules. 

We have the functors
\eqn
\rhom[\shd_X](\scbul,\scbul)&\cl& \Derb(\D_X)^\rop\times \Derb(\D_X)\to \Derp(\C_X),\\
\scbul\ltens[\D_X]\scbul&\cl&\Derb(\D_X^\rop)\times \Derb(\D_X)\to \Derm(\C_X).
\eneqn
We also have the functors
\eqn
\scbul\Dtens\scbul&\cl&\Derm(\D_X)\times\Derm(\D_X) \to\Derm(\D_X),\\
\scbul\Dtens\scbul&\cl&\Derm(\D^\rop_X)\times\Derm(\D_X) \to\Derm(\D^\rop_X),
\eneqn
\glossary{$\Dtens$}%
constructed as follows. The 
$(\D_X, \D_X\tens \D_X)$-bimodule structure on $\D_X\tens[{\OO}]\D_X$
gives 
\eqn
&&\shm\tens[{\OO}]\shn\simeq (\D_X\tens[{\OO}]\D_X)\tens[\D_X\tens\D_X](\shm\tens\shn)
\eneqn
the structure of a $\D_X$-module for $\shm$ and $\shn$ two $\D_X$-modules, 
and similarly for $\shn$ a right $\D_X$-module.

There are  similar constructions  with right $\D_X$-modules.

One defines the 
\index{duality!for D-modules}%
duality functor for D-modules by setting
\glossary{$\Ddual$}%
\eqn
&&\hs{-2ex}\Ddual_X\shm= \rhom[\D_X](\shm,\D_X\tens[\OO]\Omega_X^{\otimes-1})[d_X]
\in\Derb(\D_X)\quad \text{for $\shm\in\Derb(\D_X)$,}\\
&&\hs{-2ex}\Ddual_X\shn= \rhom[\D_X^\rop](\shn,\Omega_X\tens[\OO]\D_X)[d_X]
\in \Derb(\D^\rop_X)\quad\mbox{for $\shn\in\Derb(\D^\rop_X)$.}
\eneqn

Let $X$ and $Y$ be two complex manifolds. 
One defines the functor of 
\index{external tensor product!for D-modules}%
external tensor product for D-modules
\glossary{$\Detens$}%
\eqn
&&\scbul\Detens\scbul\cl \Derb(\D_X)\times\Derb(\D_Y)\to\Derb(\D_{X\times Y})
\eneqn
by setting $\shm\Detens\shn= \D_{X\times Y}\tens[\shd_X\etens\shd_Y](\shm\etens\shn)$.

\vspace{1ex}

Now, let $f\cl X\to Y$ be a morphism of complex manifolds. 
The {\em  transfer  bimodule} 
\index{transfer bimodule}%
\glossary{$\D_{X\to Y}$}%
$\D_{X\to Y}$ is a $(\D_X,\opb{f}\D_Y)$ bimodule defined as follows. As an 
 $(\sho_X,\opb{f}\D_Y)$-bimodule, 
$\D_{X\to Y}=\sho_X\tens[\opb{f}\sho_Y] \opb{f}\D_Y$. 
The left $\D_X$-module structure of 
 $\D_{X\to Y}$ is deduced from the action of $\Theta_X$. For
 $v\in\Theta_X$, denoting by $\sum_ia_i\tens w_i$ its image in 
$\sho_X\tens[\opb{f}\sho_Y]\opb{f}\Theta_Y$, the action  of $v$ on $\D_{X\to Y}$ is given by 
\eqn
&&v(a\tens P)=v(a)\tens P+\sum_i aa_i\tens w_iP.   
\eneqn
\glossary{$\D_{X\from Y}$}%
One also uses the opposite transfer  bimodule $\D_{Y\from X}=\opb{f}\D_Y\tens[{\opb{f}\OO[Y]}]\Omega_{X/Y}$, an $(\opb{f}\D_Y,\D_X)$-bimodule. 

Note that for another morphism of complex manifolds $g\cl Y\to Z$, one has the natural isomorphisms
\eqn
&&\D_{X\to Y}\ltens[\opb{f}\D_Y]\opb{f}\D_{Y\to Z}\simeq\D_{X\to Z},\\
&&\opb{f}\D_{Z\from Y}\ltens[\opb{f}\D_Y]\D_{Y\from X}\simeq\D_{Z\from X}.
\eneqn
One can now define the external operations on  D-modules by setting:
\glossary{$\Dopb{f}$}%
\glossary{$\Deim{f}$}%
\glossary{$\Doim{f}$}%
\eqn
&&\Dopb f\shn\eqdot \D_{X\to Y}\ltens[\opb{f}\D_Y]\opb{f}\shn,\mbox{ for }\shn\in\Derb(\D_Y),\\
&&\Deim f\shm\eqdot\reim{f}(\shm\ltens[\D_X]\D_{X\to Y})\mbox{ for }\shm\in\Derb(\D_X^\rop),
\eneqn
and one defines $\Doim{f}\shm$ by replacing $\reim{f}$ with $\roim{f}$ 
in the  above formula. By using the opposite transfer bimodule $\D_{Y\from X}$ one defines similarly the inverse image of a right $\D_Y$-module or the direct image of a left $\D_X$-module. 

One calls respectively $\Dopb{f}$, $\Doim{f}$ and $\Deim{f}$ the inverse image,  direct image and proper direct image functors in the category of D-modules. 

Note that
\eqn
&&\Dopb f\OO[Y]\simeq\OO,\quad \Dopb f\Omega_Y\simeq\Omega_X.
\eneqn
Also note that the properties of being quasi-good are stable by inverse image and tensor product, as well as by direct image by maps proper on the support of the module. The property of being good is stable by duality. 

\vspace{1ex}

Let $f\cl X\to Y$ be a morphism of complex manifolds. One associates the maps
\eqn\label{diag:tgmor}
\xymatrix@C=8ex{
T^*X\ar[rd]_-{\pi_X}&X\times_Y\ar[d]
\ar[l]_-{f_d}\ar[r]^-{f_\pi}T^*Y
                                      & T^*Y\ar[d]^-{\pi_Y}\\
&X\ar[r]^-f&Y.
}\eneqn
One says that $f$ is 
\index{non-characteristic}%
non-characteristic for $\shn\in\Derb_\coh(\D_Y)$ if the map $f_d$  is proper (hence, finite) on $\opb{f_\pi}\bl\chv(\shn)\br$.

The classical de Rham and 
\index{de Rham functor!}%
\index{solution functor!}%
solution functors are defined by
\glossary{$\dr_X$}%
\glossary{$\sol_X$}%
\begin{align*}
\dr_X &\cl \Derb(\D_X) \to \Derb(\C_X), &\shm &\mapsto \Omega_X \ltens[\D_X] \shm, \\
\sol_X &\colon \Derb(\shd_X)^\op \to \Derb(\C_X), &\shm &\mapsto \rhom[\D_X] (\shm,\OO).
\end{align*}
For $\shm\in\Derb_\coh(\D_X)$, one has
\eq\label{eq:dualdrsolnt}
\sol_X(\shm) \simeq \dr_X(\Ddual_X\shm)[-d_X].
\eneq

\begin{theorem}[{\rm Projection formulas~\cite[Theorems~4.2.8,~4.40]{Ka03}}]\label{th:Dprojform}
\index{projection formula!for D-modules}%
Let  $f\cl X\to Y$ be a morphism of complex manifolds. Let $\shm\in\Derb(\D_X)$ and 
$\shl\in\Derb(\D_Y^\rop)$.
There are natural isomorphisms{\rm:}
\eq
\Deim{f}(\Dopb{f}\shl\Dtens\shm)&\simeq&\shl\Dtens\Deim{f}\shm,\label{eq:DDprojform}\\
\reim{f}(\Dopb{f}\shl\ltens[\shd_X]\shm)&\simeq&\shl\ltens[\shd_Y]\Deim{f}\shm.\label{eq:Dprojform}
\eneq
In particular, there is an isomorphism \lp commutation of the de Rham functor and direct images\rp
\eq\label{eq:oimdr}
\reim{f}(\dr_X(\shm))\simeq\dr_Y(\Deim{f}\shm).
\eneq
\end{theorem}

\begin{theorem}[{\rm Commutativity with duality~\cite{Ka03,Sc86}}]\label{th:oimopbDdual}
Let  $f\cl X\to Y$ be a morphism of complex manifolds.
\bnum
\item
Let  $\shm\in\Derb_\good(\D_X)$ and 
assume that $f$ is proper on $\Supp(\shm)$. Then $\Deim{f}\shm\in\Derb_\good(\D_Y)$ and 
$\Ddual_Y(\Deim{f}\shm)\simeq\Deim{f}\Ddual_X\shm$.
\item
Let $\shn\in\Derb_\qgood(\D_Y)$. Then 
$\Dopb{f}\shn\in\Derb_\qgood(\D_X)$. Moreover,  if  $\shn\in\Derb_\coh(\D_Y)$ and $f$ is non-characteristic for 
$\shn$, then 
 $\Dopb{f}\shn\in\Derb_\coh(\D_X)$ and
$\Ddual_X(\Dopb{f}\shn)\simeq\Dopb{f}\Ddual_Y\shn$.
\ee
\end{theorem}
\begin{corollary}\label{cor:Dadj2}
Let  $f\cl X\to Y$ be a morphism of complex manifolds. 
\bnum
\item
Let  $\shm\in\Derb_\good(\D_X)$ and  assume that $f$ is proper on $\Supp(\shm)$. Then
we have the isomorphism for $\shn\in \RD(\D_Y)$:
\eq\label{eq:Doim}
&&\hs{2ex}\roim{f}\rhom[\shd_X](\shm,\Dopb{f}\shn)\,[d_X]\simeq\rhom[\shd_Y](\Doim{f}\shm,\shn)\,[d_Y].
\eneq
In particular, with the same hypotheses, we have the  isomorphism \lp commutation of the Sol functor and direct images\rp
\eq\label{eq:Doimsol}
&&\roim{f}\rhom[\shd_X](\shm,\OO)\,[d_X]\simeq\rhom[\shd_Y](\Doim{f}\shm,\OO[Y])\,[d_Y].
\eneq
\item
 Let $\shn\in \Derb_\coh(\D_Y)$ and assume that $f$ is non-characteristic for $\shn$. Then
we have the isomorphism for $\shm\in \RD(\D_X)${\rm:}
\eq\label{eq:Dopbsol}
&&\hs{2ex}\roim f\rhom[\D_X](\Dopb f\shn,\shm)[d_X] \simeq
\rhom[\D_Y](\shn,\Doim f\shm)[d_Y].
\eneq
\ee
\end{corollary}

A 
\index{transversal Cartesian diagram}%
\emph{transversal Cartesian diagram} is a commutative diagram
\begin{equation}\label{eq:transCart}
\ba{c}\xymatrix@C=8ex{
X' \ar[r]^{f'} \ar[d]^{g'} & Y' \ar[d]^{g} \\
X \ar[r]^{f}\ar@{}[ur]|-\square & Y
}\ea
\end{equation}
with $X'\simeq X\times_Y Y'$ and such that the map of tangent spaces
\[
T_{g'(x)}X \dsum T_{f'(x)}Y' \to T_{f(g'(x))}Y
\]
is surjective for any $x\in X'$.

\begin{proposition}[{Base change formula}] \label{pro:transCart}
\index{base change formula!for D-modules}%
Consider the transversal Cartesian diagram \eqref{eq:transCart}.
Then, for any $\shm\in\BDC_\good(\D_X)$ such that $\Supp(\shm)$ is proper over $Y$,
we have the isomorphism
\[
\Dopb g\, \Doim f \shm \simeq \Doim{f'}\,\Dopb{g'}\shm.
\]
\end{proposition}

\section{Indsheaves}\label{section:indshv}

\subsection{Ind-objects}
References are made to~\cite{SGA4} or to~\cite{KS06} for an exposition
on ind-objects. 

Let $\shc$ be a category (in a given universe). One denotes by $\shc^{\wedge}$ the big category of
functors  from $\shc^\rop$ to ${\bf Set}$. 
By the fully faithful functor $h^\wedge\cl \shc \to \shc^\wedge$,
we regard $\shc$ as a full subcategory of $\shc^{\wedge}$. 

\glossary{$\shc^\wedge$}%
An ind-object in $\shc$ is an object 
\index{ind-object}%
$A\in \shc^\wedge$ which is isomorphic
\glossary{$\sinddlim$}%
to $\inddlim[i\in I]X_i$ where $X_i\in\shc$ and   $I$
filtrant and small. Here, $\sinddlim$ is the inductive limit in $\shc^\wedge$. One denotes by $\indcc$
\glossary{$\indcc$}%
 the full   subcategory of 
$\shc^{\wedge}$ consisting of ind-objects. 

\begin{theorem}
Let $\shc$ be an abelian category.
\bnum
\item 
The category $\indcc$ is abelian.
\item 
The natural functors  $\iota\cl \shc\to \indcc$ 
and $\indcc\to \shc^{\wedge}$ are fully faithful. 
\item 
The category $\indcc$ admits exact small filtrant inductive limits, also denoted by $\sinddlim$ 
and the functor $\indcc\to \shc^{\wedge}$ commutes with such limits.
\item 
Assume that $\shc$ admits small projective limits. 
Then the category $\indcc$ admits small 
projective limits, and the functor $\shc\to \indcc$ commutes with such limits.
\item 
Assume that $\shc$ admits small inductive limits, denoted by $\sindlim$. Then the functor $\iota$ admits a left adjoint 
$\alpha$. For $X=\inddlim[i]X_i$ with $X_i\in\shc$ and $I$ small and filtrant,  $\alpha(X)\simeq\indlim[i]X_i$. 
\enum
\end{theorem}

Note that for $X=\sinddlim[i]X_i$ and $Y=\sinddlim[j]Y_j\in\indcc$ 
with $X_i,Y_j\in\shc$, one has
\eqn
&&\Hom[\indcc](X,Y)\simeq\prolim[i]\indlim[j]\Hom[\shc](X_i,Y_j).
\eneqn

\begin{example}
Let $\cor$ be a field.  Denote by $\md[\cor]$ the category of $\cor$-vector spaces and by $\mdf[\cor]$
 its  full subcategory consisting of finite-dimensional vector spaces. 
Denote for short by  $\II[\cor]$
\glossary{$\md[\cor]$}%
\glossary{$\mdf[\cor]$}%
\glossary{$\II[\cor]$}%
the category of ind-objects in $\md[\cor]$.
The functor $\alpha\cl \II[\cor]\to\md[\cor]$ admits a left adjoint $\beta\colon\md[\cor]\to\II[\cor]$ defined as follows.
For $V\in\md[\cor]$, set $\beta(V)=\sinddlim W$, where $W$ ranges over the
family of finite-dimensional vector subspaces of $V$. 
In other words, $\beta(V)$ is the functor 
\eqn
&& \md[\cor]^\rop \to \md[\Z],\\
&&M\mapsto \indlim[W\subset V]\Hom[k](M,W),\quad\mbox{$W$ finite-dimensional}.
\eneqn
Note that  $\beta(V)(M)\simeq \Hom[\cor](M,\cor)\tens V$.

If $V$ is infinite-dimensional, $\beta(V)$ is not representable in $\md[\cor]$.  Moreover, 
$\Hom[{\II[\cor]}](\cor,V/\beta(V))\simeq 0$.

Now, denote by  $\IIf[\cor]$ the
 category of ind-objects  in  $\mdf[\cor]$. There is an equivalence of categories
\eqn
&&\alpha\cl \IIf[\cor]\isoto\md[\cor],\quad \sinddlim[i]V_i\mapsto\sindlim[i]V_i.
\eneqn
We get the {\em non} commutative diagram of categories
\eq\label{diag:NC1}
&&\hs{-10ex}\ba{l}\xymatrix{
\ar@{}[rd]|(.66){NC}&\IIf[\cor]\ar[d]^-{\tw\iota}\ar[ld]_\sim\\
\md[\cor]\ar[r]_-{\iota}&\II[\cor].
}
\ea
\eneq
Moreover, the functor $\tw\iota$ commutes with small inductive limits but the functor $\iota$ does not.
\end{example}
It is proved in~\cite[Prop.~15.1.2]{KS06} that the category $\II[\cor]$ does not have enough injectives.

\begin{definition}
An object $A\in\indcc$ is quasi-injective if the functor $\Hom[\indcc](\scbul,A)$ is exact on the category $\shc$. 
\end{definition}
It is proved in loc.\ cit.\  that if $\shc$ has enough injectives, then $\indcc$ has enough quasi-injectives.

\subsection{Indsheaves}
Let $M$ be a good topological space and let $\cor$ be a field
as in subsection~\ref{subsec:shv}.

One denotes  by $\mdcp[\cor_M]$ 
\glossary{$\mdcp[\cor_M]$}%
the full subcategory of $\md[\cor_M]$ consisting of sheaves with compact support. 
We set for short:
\eqn
&&\II[\cor_M]:= {\rm Ind}(\mdcp[\cor_M])
\eneqn
\glossary{$\II[\cor_M]$}%
and calls an object of this category 
\index{indsheaves}%
an  {\em indsheaf}  on $M$. 

When there is no risk of confusion, we shall simply write $\icor_M$ instead of $\II[\cor_M]$. 
\begin{theorem}\label{th:stI}
The prestack 
\glossary{$\SI[\cor_M]$}%
$\SI[\cor_M]\cl U\mapsto\II[\cor_U]$, $U$ open in $M$, is a stack.
\end{theorem}
For $F=\inddlim[i]F_i\in{\II[\cor_M]}$ and $G=\inddlim[j]G_j\in{\II[k_M]}$ 
with $F_i$, $G_j\in\mdcp[\cor_M]$,  we set:
 \eqn
F\tens G  &=&\inddlim[i,j](F_i\tens G_j),\\
\ihom(F,G)&=&\prolim[i]\inddlim[j]\hom(F_i,G_j).
\eneqn
\glossary{$\ihom$}%
\glossary{$\tens$}%
Note that for $F\in\md[\cor_M]$ and $\{G_j\}_{j\in J}$ a small filtrant inductive system in $\II[k_M]$, we have
\eqn
&&\ihom(F,\inddlim[j]G_j)\simeq\inddlim[j]\ihom(F,G_j). 
\eneqn

\Lemma \label{lem:indtens}
The category $\II[\cor_M]$ is a tensor category
with $\tens$ as a tensor product and $\cor_M$ as a unit object.
\enlemma
Note that $\ihom$ is the inner hom of the tensor category $\II[\cor_M]$,
i.e., we have
\eqn
&&\Hom[{\II[\cor_M]}](K_1\tens K_2,K_3)
\simeq\Hom[{\II[\cor_M]}]\bl K_1, \ihom(K_2,K_3)\br.
\eneqn
\glossary{$\alpha_M$}%
\glossary{$\beta_M$}%
We have two pairs $(\alpha_M,\iota_M)$ and $(\beta_M,\alpha_M)$  of adjoint functors 
\eqn
&& \xymatrix@C=6ex{
{\md[\cor_M]\;}\ar@<-1.5ex>[rr]_-{\beta_M}\ar@<1.5ex>[rr]^-{\iota_M}&&{\;\II[\cor_M]}\ar[ll]|-{\;\alpha_M\;}.
}
\eneqn
The functor $\iota_M$ 
\glossary{$\iota_M$}%
is given by 
\eqn
&&\iota_MF=\inddlim[U\ssubset M]F_U,\mbox{ $U$ open relatively compact in $M$.}
\eneqn
The functor $\alpha_M$
is defined by
\eqn
&&\alpha_M\cl \inddlim[{ i\in I}] F_i\mapsto \indlim[{i\in I }] F_i \quad\mbox{($I$ small and filtrant).}
\eneqn
For $F\in\md[\cor_M]$, $\beta_M(F)$ is the functor
\eqn
&&\beta_M(F)\cl G\mapsto\sect(M;H^0(\RD'_M G)\tens F),\quad (G\in\mdcp[\cor_M]).
\eneqn
(This last formula is no more true if $\cor$ is not a field.)
\begin{itemize}
\item
$\iota_M$ is exact, fully faithful, and commutes with $\sprolim$, 
\item $\alpha_M$ is exact and commutes with $\sprolim$ and $\sindlim$,
\item $\beta_M$ is exact, fully faithful and commutes with $\sindlim$,
\item $\alpha_M$ is left  adjoint to $\iota_M$,
\item $\alpha_M$ is right adjoint to $\beta_M$,
\item $\alpha_M\circ\iota_M\simeq \id_{\md[\cor_M]}$ 
and $\alpha_M\circ\beta_M\simeq \id_{\md[\cor_M]}$.
\end{itemize}
Denote as usual by 
\eqn
&&\hom[\icor_M]\cl \II[\cor_M]^\rop\times\II[\cor_M]\to\md[\cor_M]
\eneqn 
the $hom$ functor  of the stack $\SI[\cor_M]$. Then 
\eqn
&&\hom[\icor_M]\simeq\alpha_M\circ\ihom,
\eneqn
and
$$\Hom[{\II[\cor_M]}](K_1,K_2)\simeq\sect\bl M;\hom[\icor_M](K_1, K_2)\br\quad
\text{for $K_1,K_2\in\II[\cor_M]$.}
$$

\begin{notation}
As far as there is no risk of confusion, we shall not write the functor $\iota_M$. Hence, we identify a sheaf $F$ on $M$ and its image by $\iota_M$.
\end{notation}

\begin{example}
Let $U\subset M$ be an open subset, $S\subset M$ a closed subset. Then
\eqn
&&\beta_M(\cor_U)\simeq\inddlim[V]\cor_V,\, V\mbox{ open }, V\ssubset U,\\
&&\beta_M(\cor_S)\simeq\inddlim[V]\cor_{\ol V},\, V\mbox{ open }, S\subset V.
\eneqn
Let $a\in M$ and consider the skyscraper sheaf $\cor_{\{a\}}$. Then $\beta_M(\cor_{\{a\}})\to\cor_{\{a\}}$ is an epimorphism in $\II[\cor_M]$ and defining $N_a$ by the exact sequence:
\eqn
0\to N_a\to \beta_M(\cor_{\{a\}})\to\cor_{\{a\}}\to0,
\eneqn
we get that $\Hom[{\icor_M}](\cor_U,N_a)\simeq 0$ for all open neighborhood $U$ of $a$. 
\end{example}

 Let $f\cl M\to N$ be a continuous map.
\index{Grothendieck's six operations!for indsheaves}%
Let $G=\inddlim[i]G_i\in{\II[k_N]}$ with $G_i\in\mdcp[\cor_N]$. One defines $\opb{f}G\in{\II[\cor_M]}$ by the formula
\eqn
&&\opb{f}G=\inddlim[i]\opb{f}G_i.
\eneqn

Let $F=\inddlim[i]F_i\in{\II[\cor_M]}$
with $F_i\in\mdcp[k_M]$. One defines $\oim{f}F\in{\II[\cor_N]}$ by
the formula:
\eqn
\oim{f}(\inddlim[i]F_i)&=&\prolim[K]\inddlim[i]\oim{f}(F_{iK})\mbox{ ($K$ compact in $M$)}.
\eneqn
The two functors $\oim{f}$ and $\opb{f}$ commute with both the functors 
$\iota$ and $\alpha$  and that is the reason why 
we keep the same notations as for usual sheaves.
Recall that for a usual sheaf $F$, its proper direct image is defined by 
\eqn
&&\eim{f}F=\indlim[U\ssubset M]\oim{f}F_U.
\eneqn
 Hence, one defines the proper direct image of $F=\sinddlim[i]F_i\in{\II[\cor_M]}$ with $F_i\in\mdcp[k_M]$ by
\eqn
\eeim{f}(\inddlim[i]F_i)&=&\inddlim[i]\oim{f}(F_{i}).
\eneqn
However, 
$\eeim{f}\circ\iota_M\neq\iota_N\circ\eim{f}$ in general. That is why we have used a different notation.

The category $\II[\cor_M]$ does not have enough injectives,
even for $M=\rmpt$ as already mentioned.
In particular, it is not a Grothendieck category. One can however 
construct the derived functors and  the six operations  for indsheaves. 
The functor $\opb{f}$  has a right adjoint
 $\roim{f}$. The functor $\reeim{f}$ admits a right adjoint, denoted by $\epb{f}$.  
\glossary{$\roim{f}$}%
\glossary{$\reeim{f}$}%
\glossary{$\epb{f}$}%
\glossary{$\opb{f}$}%
\glossary{$\rihom$}%
\glossary{$\tens$}%
Hence we have functors
\eqn
\iota_M&\cl&\Derb(\cor_M)\to\Derb(\icor_M),\\
\alpha_M&\cl&\Derb(\icor_M)\to\Derb(\cor_M),\\
\beta_M&\cl&\Derb(\cor_M)\to\Derb(\icor_M),\\
\tens&\cl& \Derb(\icor_M)\times\Derb(\icor_M)\to\Derb(\icor_M),\\
\rihom&\cl& \Derb(\icor_M)^\rop\times\Derb(\icor_M)\to\RD^+(\icor_M),\\
\rhom[\icor_M]&\cl& \Derb(\icor_M)^\rop\times\Derb(\icor_M)\to\RD^+(\cor_M),\\
\roim{f}&\cl& \Derb(\icor_M)\to\Derb(\icor_N),\\
\opb{f}&\cl& \Derb(\icor_N)\to\Derb(\icor_M),\\
\reeim{f}&\cl& \Derb(\icor_M)\to\Derb(\icor_N),\\
\epb{f}&\cl& \Derb(\icor_N)\to\Derb(\icor_M).
\eneqn

We may summarize the commutativity of the 
various functors we have introduced in the table below. Here, ``$\circ$''
means that the functors commute, and ``$\cross$'' they do not. Moreover, $\sindlim$ are taken over small filtrant categories. 
\eq\label{tabular}
\begin{tabular}{c|c|c|c|c|c|c|c|c}
& $\tens$ & $\opb{f}$ & $\oim{f}$ & $\eeim{f}$ & $\epb{f}$&$\indlim$&$\prolim$& \\
\cline{1-8}
$\iota$  & \Yes & \Yes & \Yes & \No & \Yes&\No&\Yes& \\
\cline{1-8}
$\alpha$ & \Yes & \Yes & \Yes & \Yes & \No&\Yes&\Yes& \\
\cline{1-8}
$\beta$ & \Yes& \Yes & \No  & \No  & \No&\Yes&\No& \\
\cline{1-8}
\end{tabular}
\eneq
 Note that the pairs $(\opb{f},\roim{f})$ and $(\reeim{f},\epb{f})$ are pairs of adjoint functors.
Finally, note that the functor $\epb{f}$  
commutes with filtrant inductive limits (after taking the cohomology).

\subsection{Ring action}\label{section:Iring}

We do not recall here the notion of a ring object $\B$ or a $\B$-module in 
a tensor category $\shs$ (see \cite[\S\,5.4]{KS01}). 
(In the sequel,  we shall consider the tensor category
$\II[\cor_M]$, see Lemma~\ref{lem:indtens}.)
For such a ring object $\B$ in $\shs$, 
we denote by $\md[\B]$ the abelian category of $\B$-modules in $\shs$ 
and by $\Derb(\B)$ its derived category.

We shall encounter the following situation. Let  $\A$ be a sheaf of 
$\cor$-algebras on $M$. 
Consider an object 
$\shm$ of $\II[\cor_M]$ together with a morphism of sheaves  of $\cor$-algebras
\eqn
&&\A\to\Endom[{\SI[\cor_M]}](\shm).
\eneqn
In this case one says that $\shm$ is an $\A$-module in $\II[\cor_M]$. One denotes by 
\begin{itemize}
\item
$\II[\A]$ the abelian category of $\A$-modules in $\II[\cor_M]$,
\item
$\Derb(\JA)\eqdot \Derb(\II[\A])$ its bounded derived category. We use similar notations with $\Derb$ replaced with   $\Der[+]$, $\Der[-]$ and $\Der$. 
\end{itemize}
\glossary{$\II[\A]$}%
One shall not confuse the category $\II[\A]$ with the category $\Ind(\mdcp[\A])$ of ind-objects  in the category of 
sheaves of $\A$-modules with compact support, and  we shall not confuse 
their derived categories.

If $\A$ is a sheaf of  $\cor$-algebras as above, then $\beta_M\A$ is a ring-object in the tensor category $\II[\cor_M]$.
Since
\eqn
&&\Hom[\cor_M](\A,\hom[\icor_M](\shm,\shm))
\simeq\Hom[{\II[\cor_M]}](\beta_M\A\tens \shm,\shm),
\eneqn 
we get equivalences of categories
\eqn
&&\md[\beta_M\A]\simeq \II[\A],\quad \Derb(\beta_M\A)\simeq\Derb(\JA).
\eneqn

\begin{remark}
Our notations differ from  those  of~\cite[\S\,5.4, \S\,5.5]{KS01}. 

\noi\setlength{\my}{\textwidth}\addtolength{\my}{-2.5ex}
$\bullet$\ \parbox[t]{\my}{For a ring object $\shb$ in $\II[\cor_M]$,  $\md[\shb]$ in our notation was denoted by $\II[\shb]$ in~\cite{KS01}.}

\noi
$\bullet$\ \parbox[t]{\my}%
{For a sheaf of rings  $\A$,  $\II[\A]$  in our notation was denoted by $\II[\beta\A]$ and 
$\Ind(\mdcp[\A])$ in our notation was denoted by $\II[\A]$ in~\cite{KS01}.}

\smallskip\noi
See~\cite[Exe.~3.4, Def.~4.1.2, Def.~5.4.4, Exe.~5.3]{KS01}. 
\end{remark}

We have the quasi-commutative diagram
\eq\label{diag:mdAIA}
&&\ba{c}\xymatrix@C=8ex{
\md[\A]\ar[d]\ar@<+.7ex>[r]^-{\beta_M}&\II[\A]\ar@<+.7ex>[l]^-{\alpha_M}\ar[d]\\
\md[\cor_M]\ar@<+.7ex>[r]^-{\beta_M}&\II[\cor_M]\ar@<+.7ex>[l]^-{\alpha_M}.
}\ea\eneq

For  $\shm\in\Derb(\A)$, $\shn\in\Derb(\A^\op)$ and $K\in\Derb(\JA)$ one gets the objects, functorially in $\shm$, $\shn$, $K$:
\eqn
&&\rhom[{\A}](\shm,K)\in\RD^+(\icor_M),\quad \shn\ltens[\A]K\in\RD^-(\icor_M).
\eneqn
They are characterized by
\eqn
&&\Hom[{\Der(\icor_M)}]\bl L,\rhom[{\A}](\shm,K)\br
\simeq\Hom[{\Der(\A)}]\bl\shm,\rhom[{\icor_M}](L,K)\br,\\
&&\Hom[{\Der(\icor_M)}]\bl\shn\ltens[\A]K, L\br
\simeq \Hom[{\Der(\A)}]\bl\shn,\rhom[{\icor_M}](K,L)\br
\eneqn
for any $L\in\Der(\icor_M)$.

\begin{proposition}\label{pro:nobeta}
Let $\shm\in\Derb(\A)$, $\shn\in\Derb(\A^\op)$ and $\shk\in\Derb(\JA)$.
There are natural isomorphisms:
\eqn
\rhom[{\A}](\shm,\shk)&\simeq&\rihom[\beta_M\A](\beta_M\shm,\shk) 
\quad\text{in $\RD^+(\icor_M)$,}\\
\shn\ltens[\A]\shk&\simeq&\beta_M\shn\ltens[\beta_M\A]\shk
\quad\text{in $\RD^-(\icor_M)$.}
\eneqn
\end{proposition}
\begin{proof}
Let $L\in\Derp(\icor_M)$. We have the sequence of isomorphisms 
\begin{align*}
\Hom[\RD(\icor_M)](L, \rhom[\A](\shm,\shk))
&\simeq\Hom[\RD(\A)](\shm,\rhom[\icor_M](L,\shk))\\
&\simeq\Hom[\RD(\beta_M\A)](\beta_M\shm,\rihom(L,\shk))\\
&\simeq\Hom[\RD(\icor_M)](L,\rihom[\beta_M\A](\beta_M\shm,\shk)).
\end{align*}
The second formula is proved similarly. 
\end{proof}

\begin{notation}\label{not:ring}
For $\shm\in\Derb(\JA)$, $\shn\in\Derb(\JA^\op)$ and $\shk\in\Derb(\JA)$, we shall use the notations 
$\rihom[\beta\A](\shm,\shk)$ and $\shn\ltens[\beta\A]\shk$, 
objects of $\RD(\icor_M)$. 
\end{notation}

Let us briefly recall a few basic formulas.

We consider the following situation: $f\cl M\to N$ is a continuous map of good topological spaces and  
 $\shr$ is a  sheaf of $\cor$-algebras on $N$.

In the sequel,  $\Der[\dagger]$ is $\Der$, $\Derb$, $\Derp$ or $\Derm$.

\begin{theorem}\label{th:adjIB}
\banum
\item
The functor $\opb{f}\cl \II[\cor_N]\to \II[\cor_M]$ 
induces a  functor 
$\opb{f}:\Der[\dagger](\JR)\to \Der[\dagger](\JfR)$.
\item
The functor 
$\oim{f}\cl\II[\cor_M]\to  \II[\cor_N]$ 
induces a functor 
$\roim{f}\cl \Der[\dagger](\JfR)\to \Der[\dagger](\JR)$.
\item
The functor 
$\eeim{f}\cl \II[\cor_M]\to  \II[\cor_N]$ 
induces a functor 
$\reeim{f}\colon \Der[\dagger](\JfR)\to \Der[\dagger](\JR)$.
\item
the functor $\reeim{f}\cl \Derp(\JfR)\to \Derp(\JR)$ admits a 
 a right adjoint, denoted by $\epb{f}$.
 \eanum
\end{theorem}

\begin{theorem}\label{th:projIB}
\banum
\item
For $G\in \Derm(\JR)$ and $F\in \Derp(\JfR)$,
one has the isomorphism
\eqn
\rihom[\beta_N\shr](G,\roim{f}F)
    &\simeq&\roim{f}\rihom[\opb{f}\beta_N\shr](\opb{f}G,F).
\eneqn
\item
For $G\in \Derp(\JR)$ and $F\in \Derm(\JfR)$,
one has the isomorphism
\eqn
\rihom[\beta_N\shr](\reeim{f}F,G)
    &\simeq&\roim{f}\rihom[\opb{f}\beta_N\shr](F,\epb{f}G).
\eneqn
\item \lp{\em Projection formula.}\rp\ 
\index{projection formula!for indsheaves}%
For $F\in \Derm(\JfR)$ and $G\in \Derm(\JR^\rop)$,
one has the isomorphism 
\eqn
&&G\ltens[\beta_N\shr] \reeim{f}F
\simeq \reeim{f}(\opb{f}G\ltens[\opb{f}\beta_N\shr]F).
\eneqn
\item\lp{\em Base change formula.}\rp \ 
\index{base change formula!for indsheaves}%
Consider the Cartesian square  of good topological spaces
\eq
&&\xymatrix{
M'\ar[r]^{f'}\ar[d]_{g'}\ar@{}[rd]|{\square}  &N'\ar[d]^g  \\
M\ar[r]^f                &N.
}\eneq
There are  natural isomorphisms of functors   from $\Der[\dagger](\JfR)$ to  
$\Der[\dagger](\JgR)$
\eq
&&\reeim{f'}\opb{g'}\simeq\opb{g}\reeim{f},\label{eq:basech1}\\
&&\roim{f'}\epb{g'}\simeq\epb{g}\roim{f}.\label{eq:basech2}
\eneq
\eanum
\end{theorem}
Note that  Theorem~\ref{th:betaaihom} below has no counterpart in classical sheaf theory.
\begin{theorem}\label{th:betaaihom}
Let $\sha$ be a sheaf of $\cor_M$-algebras, let $F\in \Derb(\cor_M)$, let 
$\shk\in \Derb(\JAop)$ and let 
$\shl\in \Derb(\A)$.
 Then one has the isomorphism:
 \eq\label{eq:ihombeta}
 &&\rihom(F,\shk)\ltens[\A] \shl\isoto \rihom(F,\shk\ltens[\A]\shl).
 \eneq
\end{theorem}
Thanks to Proposition~\ref{pro:nobeta}, isomorphism~\eqref{eq:ihombeta} may also be formulated as
 \eq\label{eq:ihombetabis}
 &&\rihom(F,\shk)\ltens[\beta_M\A] \beta_M\shl\isoto \rihom(F,\shk\ltens[\beta_M\A]\beta_M\shl).
 \eneq
Also note that~\eqref{eq:ihombeta} is no more true if we relax the hypothesis that  $F\in \Derb(\cor_M)$.

\subsection{Sheaves on the subanalytic site}

Recall first that,
for real analytic manifolds $M$, $N$ and a closed subanalytic subset $S$ 
of $M$, we say that
a map $f\cl S\to N$ is subanalytic if its graph is subanalytic in 
$M\times N$. 
One denotes by $\sha^\R_S$ the sheaf of continuous $\R$-valued subanalytic maps on $S$.
A {\em subanalytic space}
\index{subanalytic!space}%
$(M,\sha^\R_M)$, or simply $M$ for short,  is an $\R$-ringed space locally isomorphic to $(S,\sha^\R_S)$ 
for a closed subanalytic subset $S$ of a real analytic manifold.
A morphism of subanalytic spaces is a morphism of $\R$-ringed spaces. 
Then we obtain the category of subanalytic spaces. 

We can define the notion of subanalytic subsets of a subanalytic space,
as well as $\R$-constructible sheaves on a subanalytic space. 

\begin{definition}
Let $M$ be a subanalytic space, $\Op_M$ 
\glossary{$\Op_M$}%
\glossary{$\Op_\Msa$}%
the category of its open subsets, the morphisms being the inclusion. One 
denotes by $\Op_\Msa$ the full subcategory of $\Op_M$ consisting of 
 subanalytic  relatively compact open subsets. 
 The site $\Msa$ is obtained by  deciding that a family  $\{U_i\}_{i\in I}$
of subobjects of $U\in\Op_{\Msa}$ is a covering of $U$ if
there exists a finite subset $J\subset I$ such that  $\bigcup_{j\in J}U_j=U$. 
\index{subanalytic!site}%
\glossary{$\Msa$}%
One calls $\Msa$ the subanalytic site associated to $M$.
\end{definition}

Note that 
\eq\label{eq:MVsa1}
&&\hs{2ex}\left\{\hs{.5ex}\parbox{65ex}{
a presheaf $F$ on $\Msa$ is a sheaf if and only if $F(\emptyset)=0$ and for any pair $(U_1,U_2)$ in $\Op_\Msa$, the sequence below  is exact:\\
\hs{5ex}$0\to F(U_1\cup U_2)\to F(U_1)\oplus F(U_2)\to F(U_1\cap U_2)$.
}\right.
\eneq

Let us denote by \glossary{$\rho_M$}%
\eq
&&\rho_M\cl M\to\Msa
\eneq
the natural morphism of sites and, as usual, by $\md[\cor_\Msa]$ the Grothendieck category of sheaves of $\cor$-modules on 
$\Msa$. Hence, $(\opb{\rho_M},\oim{\rho_M})$ is a pair of adjoint functors. 

The functor $\opb{\rho_M}$ also admits a left adjoint, denoted by $\eim{\rho_M}$. For $F\in \md[\cor_M]$, $\eim{\rho_M}F$ is the sheaf associated to the presheaf $U\mapsto F(\overline U)$, $U\in \Op_{\Msa}$. 
Hence we have the two pairs of adjoint
functors $(\opb{\rho_M},\oim{\rho_M})$ and $(\eim{\rho_M},\opb{\rho_M})$
\eqn
&& \xymatrix{
{\md[\cor_M]}\ar@<-1.5ex>[rr]_-{\eim{\rho_M}}\ar@<1.5ex>[rr]^-{\oim{\rho_M}}&&{\md[\cor_\Msa]}\ar[ll]|-{\;\opb{\rho_M}\;}.
}
\eneqn
The functor $\oim{\rho_M}$ is fully faithful.

One denotes by $\sinddlim$ the inductive limit in the category $\md[\cor_{\Msa}]$. Inductive limits do not commute with the functor $\oim{\rho_M}$. 

\begin{remark}
It would be possible to develop the theory of subanalytic sheaves,
 that is sheaves on the subanalytic site, 
and in particular the six operations (see~\cite{Pr08}). 
However, in these Notes, we prefer to embed the category of subanalytic sheaves into that of indsheaves, as we shall do now. 
\end{remark}
Denote by 
\glossary{$\rc[\cor_M]$}%
\glossary{$\rcc[\cor_M]$}%
$\rc[\cor_M]$ the small abelian category of $\R$-constructible sheaves (see~\cite{KS90} for an exposition) and denote by 
$\rcc[\cor_M]$ the full subcategory consisting of sheaves with compact support.  Recall that $\Derb(\rc[\cor_M])\simeq\Derb_\Rc(\cor_M)$. 
Set
\glossary{$\IIrc[\cor_M]$}%
\eqn
&&\IIrc[\cor_M]={\rm Ind}(\rcc[\cor_M]).
\eneqn
The fully faithful functor $\rcc[\cor_M]\to \mdcp[\cor_M]$
induces a fully faithful functor
$\IIrc[\cor_M]\to \II[\cor_M]$, by which we regard
$\IIrc[\cor_M]$ as a full subcategory of $\II[\cor_M]$. 

We say that an indsheaf on $M$ is a {\em subanalytic} indsheaf 
\index{subanalytic!indsheaf}%
if it is isomorphic to an object of $\IIrc[\cor_M]$.

We have a quasi-commutative diagram of categories in which all arrows are exact and fully faithful:
\eq\label{dia00}&&\ba{c}
\xymatrix@C=7ex{
\rc[\cor_M]\ar[d]\ar[r]^{\iota_M^{rc}}&\IIrc[\cor_M]\ar[d]\\
\md[\cor_M]\ar[r]^-{\iota_M}&\II[\cor_M].
}\ea
\eneq

\begin{proposition}
The restriction of the functor $\oim{\rho_M}$ to the subcategory $\rc[\cor_M]$
is exact and fully faithful.
\end{proposition}
We have a natural functor 
\eq\label{eq:2sinddlim}
&&\lambda_{M}\cl\IIrc[\cor_{M}]\to \md[\cor_{\Msa}],\quad\inddlim[i]F_i\mapsto\inddlim[i]\oim{\rho_M} F_i,
\eneq
where the first $\sinddlim$ is taken in the category $\IIrc[\cor_{M}]$ and the second one is taken in the category 
$\md[\cor_\Msa]$. 

\begin{theorem}
The functor $\lambda_M$ in~\eqref{eq:2sinddlim} is an equivalence.
\end{theorem}
In other words, subanalytic indsheaves are usual sheaves
on the subanalytic site.
By this result, the embedding $\IIrc[\cor_M]\into\II[\cor_M]$ gives an exact and  fully faithful functor
\eq
&&\tw\iota_M\cl \md[\cor_{\Msa}]\to\II[\cor_M].\label{def:tw}
\eneq
Note that for $G\in\md[\cor_\Msa]$, one has
\eqn
&&\tw\iota_MG\simeq\inddlim[{\oim{\rho_M}F\to G}]F, \mbox{ where }F\in\rc[\cor_M].
\eneqn
Also note that
\eqn
&&\tw\iota_M\hom(F,G)\simeq\ihom(F,\tw\iota_MG)\mbox{ for } F\in\rc[\cor_M],\, G\in\md[\cor_\Msa].
\eneqn
We have the following diagrams, where the one in the left is {\em non} commutative and the one in the right is commutative
 (see Diagram~\ref{diag:NC1} for the case $M=\rmpt$):
\eq\label{diag:NC2}
&&\ba{l}\xymatrix{
\md[\cor_{\Msa}]\ar[rd]^-{\tw\iota_M}&\\
\md[\cor_M]\ar[u]^-{\oim{\rho_M}}\ar[r]_-{\iota_M}\ar@{}[ru]|(.33){NC}&\II[\cor_M],
}\quad
\xymatrix{
\md[\cor_{\Msa}]\ar[rd]_-{\tw\iota_M}\ar[r]^\sim&\IIrc[\cor_M]\ar[d]\\
&\II[\cor_M].
}\ea
\eneq
The functors $\iota_M$ and $\tw\iota_M$ are exact but $\oim{\rho_M}$ is not right exact in general. 

\begin{lemma}
The two diagrams below commute:
\eq\label{diag:NC3}
&&\ba{l}\xymatrix{
\md[\cor_\Msa]\ar[rd]^-{\tw\iota_M}\ar[d]_-{\opb{\rho_M}}&\\
\md[\cor_M]&\II[\cor_M]\ar[l]^-{\alpha_M},
}\quad
\xymatrix{
\md[\cor_\Msa]\ar[rd]^-{\tw\iota_M}&\\
\md[\cor_M]\ar[u]^-{\eim{\rho_M}}\ar[r]_-{\beta_M}&\II[\cor_M].
}\ea
\eneq
\end{lemma}
\begin{proof}
(i)  Let us prove the commutation of the diagram on the left. Since all functors in the diagram commute with  inductive limits, we are reduced to prove the isomorphism $\opb{\rho_M}\oim{\rho_M} F
\simeq\alpha_M\tw\iota_M\oim{\rho_M} F$ 
for $F\in\rcc[\cor_M]$ and the result is clear in this case.

\spa
(ii) Let us prove the commutation of the diagram on the right. 
Again all functors in the diagram commute with  inductive limits. 
We shall first prove that 
\eq\label{eq:beta=iotal}
&&\parbox{60ex}{
the functor $\beta_M$ factors as $\beta_M=\tw\iota_M\circ\lambda_M$ for a functor $\lambda_M\cl \md[\cor_M]\to \md[\cor_\Msa]$. 
}\eneq
First consider the case of  $F=\cor_U$ for $U$ open and relatively compact  in $M$.  In this case, 
\eqn
&&\beta_M\cor_U\simeq\inddlim[V\ssubset U]\cor_V,\mbox{ $V$ open in $M$}
\eneqn
and we may assume that $V$ is subanalytic. Hence $\beta_M\cor_U$ is a subanalytic indsheaf.
Since any $F\in\md[\cor_M]$ is obtained by taking direct sums and cokernels of sheaves of the type $\cor_U$ and the subcategory of subanalytic indsheaves is stable by these operations, 
$\beta_MF$ is a subanalytic indsheaf for any $F\in \md[\cor_M]$ and we get~\eqref{eq:beta=iotal}.
It remains to prove that  $\lambda_M\simeq\eim{\rho_M}$. 
Let $F\in\md[\cor_M]$ and $G\in\md[\cor_\Msa]$. Using (i) and the fact that $\tw\iota_M$ is fully faithful, we have
\eqn
\Hom(\eim{\rho_M}F,G)
&\simeq&\Hom(F,\opb{\rho_M}G)
\simeq\Hom(F,\alpha_M\tw\iota_MG)\\
&\simeq&\Hom(\beta_MF,\tw\iota_MG)
\simeq\Hom(\tw\iota_M\lambda_M  F,\tw\iota_MG)\\
&\simeq&\Hom(\lambda_M F,G).
\eneqn
\end{proof}
We denote by $\Derb_\Irc(\icor_M)$ 
\glossary{$\Derb_\Irc(\icor_M)$}%
the full subcategory of $\Derb(\icor_M)$ consisting of objects with subanalytic indsheaves as cohomologies. By~\cite[Th~7.1]{KS01}, we have:
\begin{theorem}
The functor $\tw\iota_M$ induces an equivalence of triangulated categories
\eq
&&\Derb(\cor_\Msa)\isoto\Derb_\Irc(\icor_M).
\eneq
\end{theorem}
\begin{proposition}\label{proopsa1}
Let $M$ be a subanalytic space. 
\bnum
\item
Let $K,L\in\Derb_\Irc(\icor_M)$. Then $K\tens  L\in\Derb_\Irc(\icor_M)$.
\item
 Let $K\in\Derb_\Irc(\icor_M)$ and let $F\in\Derb_\Rc(\cor_M)$. 
 Then $\rihom(F,K)\in\Derb_\Irc(\icor_M)$.
 \enum
\end{proposition}
 \begin{proposition}\label{proopsa2}
Let $f\cl M\to N$ be a morphism of subanalytic spaces. 
\bnum
\item
For $L\in \Derb_\Irc(\icor_N)$,
we have $\opb{f}L\in\Derb_\Irc(\icor_M)$ and 
$\epb{f}L\in\Derb_\Irc(\icor_M)$.
\item
For $K\in\Derb_\Irc(\icor_M)$, we have $\reeim{f}K\in\Derb_\Irc(\icor_N)$.

\enum
\end{proposition}
The next result will be of a constant use. 
\begin{proposition}\label{pro:eqvIrcDbsa}
A morphism $u\cl K\to L$ in $\Derb_\Irc(\icor_M)$ is an isomorphism if and only if, for any relatively compact subanalytic open subset $U$ of $M$ and any $n\in\Z$, $u$ induces an isomorphism
$\Hom[\Derb(\icor_M)](\cor_U[n],K)
\isoto\Hom[\Derb(\icor_M)](\cor_U[n],L)$.
\end{proposition}

\subsection{Some classical sheaves on the subanalytic  site}
In this subsection, we take $\C$ as the base field $\cor$. 
\begin{notation}
Let $X$ be a complex manifold and let $\D_X$ be the sheaf of differential operators, as in \S~\ref{subsec:Dmod}.
According to Proposition~\ref{pro:nobeta}, for $\shm\in\Derb(\D_X)$, 
we get the functors
\eqn
\rhom[\shd_X](\shm,\scbul)&\cl& \Derb(\JD_X)\to \Derp(\iC_X),\\
\scbul\ltens[\D_X]\shm&\cl&\Derb(\JD_X^\rop)\to \Derm(\iC_X),\\
\scbul\Dtens\scbul&\cl&\Derm(\JD_X)\times\Derm(\JD_X) \to\Derm(\JD_X).
\eneqn
There are  similar constructions  with right $\D_X$-modules.

If $M$ is a real analytic manifold, we denote by $\D_M$ the sheaf of finite-order differential operators
 with real analytic coefficients. Denoting by $X$ a complexification of $M$, 
 we have 
 $\D_M\simeq\D_X\vert_M$ and the notations above apply with $\D_X$ replaced by  $\D_M$.
\end{notation}

\subsubsection{Tempered and Whitney functions and distributions}
In this subsection and the next ones,  
$M$ denotes a real analytic manifold.

As usual, we denote by 
\glossary{$\Cinf$}%
\glossary{$\Cinfom$}%
\glossary{$\Db_M$}%
\glossary{$\shb_M$}%
$\Cinf$ (resp.\ $\Cinfom$) the sheaf of  $\C$-valued  functions
of class $\mathrm{C}^\infty$ (resp.\ real analytic) and  by $\Db_M$ 
(resp.\ $\shb_M$) 
the sheaf of Schwartz's distributions (resp.\ Sato's hyperfunctions).
We also use the notation $\rA_M=\Cinfom$.

\begin{definition}
Let $U$ be an open subset of $M$  and $f\in \Cinf(U)$. 
One says that $f$ has
\index{polynomial growth}%
{\it  polynomial growth} at $p\in M$
if  $f$ satisfies the following condition:
for a local coordinate system
$(x_1,\dots,x_n)$ around $p$, there exist
a sufficiently small compact neighborhood $K$ of $p$
and a positive integer $N$
such that
\eq
&&\sup\limits_{x\in K\cap U}\big(\dist(x,K\setminus U)\big)^N\vert f(x)\vert
<\infty\,.\label{eq:moderate}
\eneq
Here, $\dist(x,K\setminus U)\seteq\inf\set{|y-x|}{y\in K\setminus U}$,
and 
we understand that the left-hand side of \eqref{eq:moderate} is $0$
if $ K\cap U=\emptyset$ or $K\setminus U=\emptyset$. 
Hence $f$ has polynomial growth at any point of $U$. 
We say that $f$ is  
\index{tempered!$\mathrm{C}^\infty$-functions}%
tempered at $p$ if all its derivatives
have polynomial growth at $p$. We say that $f$ is {\em tempered} 
if it is tempered at any point of $M$.
\end{definition}
An important property of subanalytic subsets is given by the lemma
below. (See Lojasiewicz~\cite{Lo59} and also~\cite{Ma66} 
for a detailed study of its consequences.)

\begin{lemma}\label{th:Loj1}
Let $U$ and $V$ be two relatively compact open subanalytic subsets of $\R^n$.
There exist a positive integer $N$ and $C>0$ such that
\eqn
&&\dist\big(x,\R^n\setminus (U\cup V)\big)^N
\leq C \big(\dist(x,\R^n\setminus U)+\dist(x,\R^n\setminus V)\big).
\eneqn
\end{lemma}
For an open subanalytic subset  $U$ in $M$,
denote by $\Cinft[M](U)$ the subspace
of $\Cinf(U)$ consisting of tempered $\mathrm{C}^\infty$-functions. 

Denote by $\Dbt_M(U)$ the image 
of the restriction map $\sect(M;\Db_M)\to\sect(U;\Db_M)$,
and call it the space of {\em tempered distributions} on $U$. 
\index{tempered!distributions}%
Using Lemma~\ref{th:Loj1} and~\eqref{eq:MVsa1} one proves:
\begin{itemize}
\item
the presheaf  $U\mapsto \Cinft[M](U)$ is a sheaf on  $\Msa$,
\item
the presheaf  $U\mapsto \Dbt_M(U)$ is a sheaf on $\Msa$. 
\end{itemize}
\glossary{$\Cinft[\Msa]$}%
\glossary{$\Dbt_{\Msa}$}%
One denotes them by $\Cinft[\Msa]$ and $\Dbt_{\Msa}$. 

\medskip

For a closed subanalytic subset $S$ in $M$, 
denote by $\cI^\infty_{M,S}$ the space
of $\mathrm{C}^\infty$-functions  defined on $M$ which vanish up to 
infinite order on $S$. 
In~\cite{KS96}, one  introduced  the sheaf:
\eqn
\C_U\wtens \Cinf&:=&V\longmapsto \cI^\infty_{V,V\setminus U}
\eneqn
and showed  that it  uniquely extends  to an exact functor 
\eqn
&&\scbul\wtens \shc^{\infty}_M\cl \mdrc[\C_M]\to \md[\C_M].
\eneqn
One denotes by $\Cinfw[\Msa]$ the sheaf on $\Msa$ given by 
\eqn
&&\Cinfw[\Msa](U)=\sect\bl M;H^0(\RD'_M\cor_U)\wtens \shc_M^{\infty}\br, 
\ U\in \Op_{\Msa}.
\eneqn
If $\RD'_M\C_U\simeq\C_{\ol U}$, then $\Cinfw[\Msa](U)\simeq \Cinf(M)/
\cI^\infty_{M,\ol U}$ is the space of Whitney functions on
$\ol{U}$. 
It is thus natural to call 
\glossary{$\Cinfw[\Msa]$}%
$\Cinfw[\Msa]$ the sheaf of 
\index{Whitney functions}%
Whitney $\mathrm{C}^\infty$-functions on $\Msa$. 

Note that the sheaf $\oim{\rho_M}\shd_{M}$ does not operate on
the sheaves  $\Cinft[\Msa]$,
$\Dbt_\Msa$, $\Cinfw[\Msa]$ but  $\eim{\rho_M}\shd_{M}$ does.

\begin{notation}
Recall the exact and fully faithful functor $\tw\iota_M\cl\md[\C_\Msa]
\to\md[\iC_M]$  in \eqref{def:tw}. 
 We denote by $\Cinfw[M]$, $\Cinft[M]$ and $\Dbt_M$  the indsheaves 
 \glossary{$\Cinfw[M]$}%
 \glossary{$\Cinft[M]$}%
 \glossary{$\Dbt_M$}%
 $\tw\iota_M\Cinfw[\Msa]$, $\tw\iota_M\Cinft[\Msa]$ 
and $\tw\iota_M\Dbt_\Msa$ and calls them 
\index{indsheaves!of Whitney functions}%
\index{indsheaves!of tempered $\mathrm{C}^\infty$-functions}%
\index{indsheaves!of tempered distributions}%
the {\em
indsheaves of Whitney functions, tempered $\mathrm{C}^\infty$-functions 
and  tempered distributions}, respectively. 
\end{notation}

We have  monomorphisms  of indsheaves 
\eqn
&&\xymatrix@R=4ex{
\wwb{\Cinfw[M]\ }\ar@{>->}[r]&\wwb{\Cinft[M]\ }\ar@{>->}[r]\ar@{>->}[d]&
\wwb{\Cinf[M]}\ar@{>->}[d]\\
&\Dbt_M\ \ar@{>->}[r]&\Db_M.
}
\eneqn

Let $F\in\Derb_\Rc(\C_M)$. One has the isomorphisms in $\Derb(\C_M)$:
\eq\label{eq:thom}
&&\ba{rcl}\opb{\rho_M}\rhom(\roim{\rho_M}F,\Dbt_{\Msa})&\simeq&\rhom[\iC_M](F,\Dbt_M)\\[1ex]
&\simeq&\thom(F,\Db_M),\ea
\eneq
where the functor 
\eqn
&&\thom(\scbul,\Db_M)\cl\Derb_\Rc(\C_M)^\rop\to\Derb(\C_M)
\eneqn
 was defined in~\cite{Ka80,Ka84} as the main tool for the proof of the Riemann-Hilbert correspondence for regular holonomic D-modules. 

We also have
\eqn
&&\rhom[\iC_M](F,\Cinfw[M])\simeq \RD'_M F\wtens\Cinf[M].
\eneqn

We shall see in Subsection~\ref{subsection:dualWhitTemp} that there is a kind of duality between the indsheaves $\Cinfw[M]$ and $\Dbt_M$.

\subsubsection{Operations on tempered distributions}
Let us describe without detailed  proofs the behaviour of the indsheaf of tempered distributions with respect to direct  and  inverse images
(see \cite{KS01}).
 In~\cite{KS96} these operations are treated in the language of the functor $\thom$ 
\glossary{$\thom$}%
 introduced in~\cite{Ka84},      but we prefer to use  
the essentially equivalent language of indsheaves.

For a real analytic manifold $M$
and for a  morphism of real analytic manifolds $f\cl M\to N$,  we denote by
\begin{itemize}
\item
$\dim M$ \  the dimension of $M$, 
\glossary{$\dim M$}%
\item
 $\rA^{(\dim M)}_M$ \ the sheaf of real analytic forms of top degree,  
\glossary{$\rA$}%
 \item
 $\Theta_M$ \ the sheaf of real analytic vector fields, 
\glossary{$\Theta_M$}%
 \item
 $\ori_M$ the orientation sheaf,  
\glossary{$\ori_M$}%
\item
 $\shv_M\eqdot \rA^{(\dim M)}_M \tens\ori_M$ \ the sheaf of real analytic densities on $M$, 
\glossary{$\shv_M$}%
 \item
 $\Dbtv_M\eqdot \shv_M\ltens[\rA_M]\Dbt_M$ \ the indsheaf of tempered distributions densities, 
\glossary{$\Dbtv_M$}%
 \item
 $\D_{M\to N}=\rA_M\tens[\opb{f}\rA_N]\opb{f}\D_N$ \ the transfer bimodule. 
\glossary{$\D_{M\to N}$}%
  \end{itemize}

\begin{proposition}\label{pro:etensDbt}
Let $M$ and $N$ be two real analytic manifolds.
There exists a natural morphism
\eq\label{eq:etensDbt}
&&\Dbt_M\etens \Dbt_N\to \Dbt_{M\times N}\mbox{ in }\Derb({\rm I}(\D_M\etens \D_N)).
\eneq
\end{proposition}
The next result is a feormulation of a theorem of~\cite{Ka84}.
\begin{theorem}\label{th:eimDbt}
Let $f\cl M\to N$ be a morphism of  real analytic manifolds. There exists  a natural isomorphism
\eq\label{eq:opbDbt}
&&\Dbtv_M\ltens[\D_M]\D_{M\to N}\isoto\epb{f}\Dbtv_N\mbox{ in }\Derb(\JfD^\op_N).
\eneq
\end{theorem}
\begin{proof}[Sketch of proof]
(i) First, we construct the morphism in~\eqref{eq:opbDbt}.
By adjunction it is enough to construct a morphism
\eq
&&\reeim{f}(\Dbtv_M\ltens[\D_M]\D_{M\to N})\to\Dbtv_N.
\label{mor:Dbf}
\eneq

Denote by 
$\Sp_\scbul(\shm)$ the Spencer complex of a coherent $\D_M$-module $\shm$. There is a quasi-isomorphism 
$\Sp_\scbul(\shm)\to\shm$, where
$\Sp_k(\shm)$ is the $\D_M$-module 
$\D_M\tens[\rA_M]\bigwedge^k\Theta_M\tens[\rA_M]\shm$.
Then $\Sp_\scbul(\D_{M\to N})$ gives a resolution
of $\D_{M\to N}$ as a $(\D_M,\opb{f}\D_N)$-bimodule
locally free over $\D_M$.
Note that $\Dbtv_M\tens[\D_M] \Sp_k(\D_{M\to N})$ is acyclic 
with respect to the functor $\eeim{f}$ for any $k$. Hence,  in order to construct morphism \eqref{mor:Dbf},
it is enough to construct a morphism of complexes in $\II[\D_X]$ 
\eq\label{eq:opbDbt2}
&&\eeim{f}\bl\Dbtv_M\tens[\D_M] \Sp_\scbul(\D_{M\to N})\br\to\Dbtv_N.
\eneq
Set for short 
$$\shk_\scbul=\Dbtv_M\tens[\D_M] \Sp_\scbul(\D_{M\to N})
\simeq \Dbtv_M\tens[\rA_M]\bigwedge^\scbul\Theta_M\tens[f^{-1}\rA_N]f^{-1}\D_N.$$ 
Then we have
$\eeim{f}(\shk_0)=\eeim{f}(\Dbtv_M)\tens[\rA_N]\D_N$.
The integration of distributions gives a morphism
\eq\label{eq:opbDbt3}
&&\int_f : \eeim{f}(\Dbtv_M)\to \Dbtv_N.
\eneq
Since $\Dbtv_N$ is a right $\D_N$-module, we obtain
the morphism $u\cl \eeim{f}(\shk_0)\to \Dbtv_N$.
By an explicit calculation, one checks that the composition
\eqn
&&\eeim{f}(\shk_1)\To[d_1]\eeim{f}\shk_0\To[u]\Dbtv_N
\eneqn
vanishes. This defines morphism \eqref{mor:Dbf}
and hence  the morphism in~\eqref{eq:opbDbt}.

\medskip\noi
(ii)\  
One can treat separately the case of a closed embedding and a submersion.

\smallskip\noi
(a)\ If $f\cl M\to N$ is a closed embedding, 
the result follows from the isomorphism
\eqn
&& \rihom(\oim{f}\C_M,\Dbtv_N)\simeq \Dbtv_M\tens[\D_M]\D_{M\to N}.
\eneqn

\medskip\noi
(b)\ When $f$ is a submersion, one reduces to the case where $M=N\times\R$ and $f$ is the projection. 
Let $F\in\Derb_\Rc(\cor_M)$ such that $f$ is proper on $\Supp(F)$ and let us apply the functor 
$\roim{f}\rhom(F,\scbul)$
to the morphism~\eqref{eq:opbDbt}. Using $\rhom(F,\scbul)\simeq \alpha_M\circ\rihom(F,\scbul)$,
we get the morphism
\eq\label{eq:opbDbt4}
&&\hs{3ex}
\ba{rcl}\roim{f}\bl\rhom(F,\Dbtv_M)\ltens[\D_M]\D_{M\to N}\br&\to&\roim{f} \rhom(F,\epb{f}\Dbtv_N)\\[1ex]
&\simeq&\rhom(\reim{f}F,\Dbtv_N).\ea
\eneq
By Proposition~\ref{pro:eqvIrcDbsa}, it remains to prove that~\eqref{eq:opbDbt4} is an isomorphism.

One then reduces to the case where $F=\C_Z$ for a closed subanalytic subset 
$Z$ of $N\times\R$ proper over $N$. Then, by using the structure of subanalytic sets, one reduces to the case 
where $\opb{f}(x)\cap Z$ is a closed interval for each $x\in f(Z)$.
Finally, one proves that the sequence below is exact.
\eqn
&&0\To \eim{f}\sect_Z\Db_M\To[\partial_t] \eim{f}\sect_Z\Db_M\To
[{\int}_\R\,(\cdot)\,dt]\sect_{f(Z)}\Db_N\to 0.
\eneqn
\end{proof}

One often needs to compactify real analytic  manifolds. In order to check that the construction does not depend on the 
 choice of compactifications, the next lemma is useful.

\begin{lemma}\label{le:sacompactif}
Consider a morphism $f\cl M\to N$ of real analytic manifolds and let $V\subset N$ be a subanalytic open subset. Set $U=\opb{f}V$ and assume that $f$ induces an isomorphism of real analytic manifolds $U\isoto V$. Then
\eq
&&\rihom(\C_U,\Dbt_M)\simeq \epb{f}\rihom(\C_V,\Dbt_N).
\eneq
\end{lemma}
\Proof
By Theorem~\ref{th:eimDbt}, we have
\begin{align*}
\epb{f}\rihom(\C_V,\Dbt_N)
&\simeq\rihom(f^{-1}\C_V,\epb{f}\Dbt_N)\\
&\simeq\rihom(\C_U,\Dbtv_M\ltens[\D_M]\D_{M\to N}).
\end{align*}
Since the morphism of $\D_M$-modules
$\D_M\to\D_{M\to N}$
is an isomorphism on $U$,
it induces an isomorphism 
$$\rihom(\C_U,\Dbtv_M\ltens[\D_M]\D_M)\isoto\rihom(\C_U,\Dbtv_M\ltens[\D_M]\D_{M\to N}).$$
\QED

\begin{remark}
By choosing $N=\rmpt$ and $F=\C_U$ for $U$  open  subanalytic, we obtain that 
$\RHom(\reim{f}\C_U,\C)\simeq\rsect(U;\omega_M)$ is isomorphic to the de Rham complex 
with coefficients in $\Dbt_M(U)$. This is a vast generalization of a well-known theorem 
of Grothendieck~\cite{Gr66} 
which asserts that the cohomology of the complementary of an algebraic  hypersurface $S$ may be calculated as the de Rham complex 
with coefficients  in  the sheaf of meromorphic functions with poles on $S$. 
This  result has been  generalized to the semi-analytic setting by 
Poly~\cite{Po74}. 
\end{remark}

\subsubsection{Whitney and tempered holomorphic functions}
Let $X$ be a complex manifold. We denote  by $X^c$ the complex conjugate manifold to $X$ and by $X_\R$ the  underlying real analytic  manifold.

We define the following  indsheaves 
\glossary{$\Oww[X]$}%
\glossary{$\Ow[X]$}%
\glossary{$\Ot[X]$}%
\glossary{$\Ovt_X$}%
\eqn\label{eq:defOt1}
\Oww[X]&\eqdot&\beta_X\OO,\\
\Ow[X]&\eqdot&\rhom[\shd_{X^c}](\sho_{X^c},\Cinfw[X_\R])\simeq\Omega_{X^c}\ltens[\shd_{X^c}]\Cinfw[X_\R]\,[-d_X],\\
\Ot[X]&\eqdot&\rhom[\shd_{X^c}](\sho_{X^c},\Dbt_{X_\R})\simeq\Omega_{X^c}\ltens[\shd_{X^c}]\Dbt_{X_\R}\,[-d_X],\\
\Ovt_X&\eqdot&\Omega_X\tens[\sho_X]\Ot[X].
\eneqn
The first three
are objects of  $\Derb_\Irc(\JD_X)$ while the last one
is an object of $\Derb_\Irc(\JD_X^\rop)$. 
Hence $\Ot$ is isomorphic to the Dolbeault complex
with coefficients in $\Dbt_{X_\R}$:
$$0\To\Dbt_{X_\R}\To[\ol\partial]\Db_{X_\R}^{\mathrm{t}\;(0,1)}\To[\ol\partial]\cdots
\To[\ol\partial]\Db_{X_\R}^{\mathrm{t}\;(0,d_X)}\To0,$$
where  $\Db_{X_\R}^{\mathrm{t}\;(0,p)}\seteq\Omega_{X^c}^p\tens[{\OO[X^c]}]\Dbt_{X_\R}$
is situated in degree $p$. 

One calls $\Ow[X]$  and $\Ot[X]$  the 
\index{indsheaves!of Whitney holomorphic functions}%
\index{indsheaves!of tempered holomorphic functions}%
\index{tempered!holomorphic functions}%
{\em indsheaves of Whitney and tempered holomorphic functions}, respectively. 
We have the morphisms in the category  $\Derb(\JD_X)$:
\eqn
&&\Oww\to\Ow[X]\to\Ot[X]\to\OO.
\eneqn
One proves the isomorphism
\eq\label{eq:defOt2}
\Ot[X]&\simeq&\rhom[\shd_{X^c}](\sho_{X^c},\Cinft[{X_\R}])\mbox{ in }\Derb(\JD_X).
\eneq

Note that the object $\Ot[X]$ is not
concentrated in degree zero if $d_X>1$. Indeed, with the subanalytic topology, only finite coverings 
are allowed. If one considers for example the open subset $U\subset\C^n$, 
the difference of an open ball of radius $R$ and a closed ball of 
radius $r$ with $0<r<R$, 
then the Dolbeault complex will not be exact after any finite covering.

\begin{example}\label{exa:OtOwCDb}
(i) Let $Z$ be a closed complex analytic subset of the complex manifold $X$. We have the isomorphisms in $\Derb(\D_X)$:
\eqn
&&
\ba{ll}\rhom[{\iC_X}](\RD_X'\C_Z,\Oww)\simeq(\OO)_Z&\mbox{(restriction)},\\[1ex]
\rhom[{\iC_X}](\RD_X'\C_Z,\Ow)\simeq\OO\widehat{\vert}_Z&
\mbox{(formal completion)},\\[1ex]
\rhom[{\iC_X}](\C_Z,\Ot)\simeq\rsect_{[Z]}(\OO)&\mbox{(algebraic cohomology)},
\\[1ex]
\rhom[{\iC_X}](\C_Z,\OO)\simeq\rsect_{Z}(\OO)&\mbox{(local cohomology)}.\ea
\eneqn
(ii) Let  $M$ be a real analytic manifold such that $X$ is a complexification of $M$. We have the isomorphisms  in $\Derb(\D_M)$:
\eqn
&&\ba{ll}
\rhom[{\iC_X}](\RD_X'\C_M,\Oww)\vert_M\simeq \rA_M &\mbox{(real analytic functions)},
\\[1ex]
\rhom[{\iC_X}](\RD_X'\C_M,\Ow)\vert_M\simeq \Cinf[M]&\mbox{($\mathrm{C}^\infty$-functions)},
\\[1ex]
\rhom[{\iC_X}](\RD_X'\C_M,\Ot)\vert_M\simeq\Db_M&\mbox{(distributions)},\\[1ex]
\rhom[{\iC_X}](\RD_X'\C_M,\OO)\vert_M\simeq\shb_M&\mbox{(hyperfunctions)}.
\ea
\eneqn
\end{example}

\subsubsection{Duality between Whitney and tempered functions }\label{subsection:dualWhitTemp}

We shall use the theory of topological $\C$-vector spaces of type 
$\FN$ (Fr\'echet nuclear spaces) or $\DFN$ 
(dual of  Fr\'echet nuclear  spaces). 
The categories of $\FN$ spaces and $\DFN$ spaces are quasi-abelian and 
the topological duality functor induces a contravariant equivalence between
the  category of $\FN$ spaces and $\DFN$ spaces.
It induces therefore an equivalence of triangulated categories
$$\Derb(\FN)^\rop\simeq \Derb(\DFN).$$

\begin{proposition}[{\cite[Prop.~2.2]{KS96}}]\label{pro:dualWhitTemp1}
Let $M$ be a real analytic manifold and let $F\in\rc[\C_M]$. 
Then, there exist natural topologies of type $\FN$ 
\glossary{$\FN$}%
 on 
$\sect(M;F\wtens\Cinf[M])$ and of type $\DFN$ 
\glossary{$\DFN$}%
on $\sect_c(M;\hom[\iC_M](F,\Dbtv_M))$, 
and they are dual to each other. 
\end{proposition}
Here, as usual, $\sect_c(M;\scbul)$ is the 
functor of global sections with compact support. 

Hence for any open subset $U$ of $M$, we have
\eqn
&&\sect(U;F\wtens\Cinf[M])\To\Hom[\C]\bl\sect_c(U;\hom[\iC_M](F,\Dbtv_M)),\C\br\\
&&\hs{30ex}\simeq\sect\bl U;\RD_M\hom[\iC_M](F,\Dbtv_M)\br,\eneqn
which induces a morphism of sheaves
$F\wtens\Cinf[M]\to \RD_M\hom[\iC_M](F,\Dbtv_M)$ and then a pairing
\eq\label{eq:dualWTpairO}
&&\bl F\wtens\Cinf[M]\br\tens\hom[\iC_M](F,\Dbtv_M)\to\omega_M.
\eneq

Let $X$ be a complex manifold, let $\shm\in\Derb_{\coh}(\shd_X)$ 
and let $F,G\in\Derb_{\Rc}(\C_X)$. Set for short 
\eqn
\shw(\shm,F)&\eqdot& \rhom[\shd_X](\shm,F\wtens\sho_X),\\
\sht(F,\shm) &\eqdot& \rhom[\iC_X](F,\Ovt_X\,[d_X])\ltens[\shd_X]\shm\\
W(G,\shm,F)&\eqdot&\RHom\bl G,\rhom[\shd_X](\shm,F\wtens\sho_X)\br,\\
T_c(F,\shm,G) &\eqdot& \rsect_c\bl X;\rhom[\iC_X](F,\Ovt_X\,[d_X])\ltens[\shd_X]\shm\tens G\br.
\eneqn

Then \eqref{eq:dualWTpairO}
induces a pairing
\eq\label{eq:dualWTpair}
&&\shw(\shm,F)\tens\sht(F,\shm)\to\omega_X.
\eneq
and a pairing 
\eq\label{eq:dualWTpairT}
&&W(G,\shm,F)\tens T_c(F,\shm,G)\to\C,
\eneq

\begin{theorem}[{\cite[Theorem 6.1]{KS96}}]\label{thm:KSduality}
The two objects $W(G,\shm,F)$ and $T_c(F,\shm,G)$ are well-defined in the categories 
$\Derb(\FN)$ and $\Derb(\DFN)$, respectively, and are dual to each other through the pairing~\eqref{eq:dualWTpairT}. 
\end{theorem}

Now we assume that 
$\shm\in\Derb_\hol(\shd_X)$ and we consider the following assertions.
\setlength{\my}{\textwidth}
\addtolength{\my}{-13ex}
\eq\label{cond:const}
&&\hs{3ex}\left\{\parbox{\my}{
\bna
\item
the object $\shw(\shm,F)\seteq\rhom[\shd_X](\shm,F\wtens\sho_X)$ is $\R$-constructible,
\item
the object $\sht(F,\shm)\seteq\rhom[\iC_X](F,\Ovt_X\,[d_X])\ltens[\shd_X]\shm$ is $\R$-constructible,
\item conditions (a) and (b) are satisfied,  and 
the two complexes in (a) and (b) are dual to each other in the category $\Derb_\Rc(\C_X)$, that is, 
$\shw(\shm,F)\simeq\RD_X\sht(F,\shm)$.
\ee}
\right.
\eneq
\begin{lemma}
The assertions  {\rm(a)} and {\rm(b)}  are equivalent and imply {\rm(c)} . 
\end{lemma}
\begin{proof}
Assume for example that (b) is true. 
The pairing \eqref{eq:dualWTpair}
induces a morphism
\eq\label{mor:shwsht}
\shw(\shm,F)\to \RD_X(\sht(F,\shm)).
\eneq
For any relatively compact open subanalytic subset $U$, 
$\rsect_c(U;\sht(F,\shm))$ has finite-dimensional cohomologies by (b), and 
the morphism induced by \eqref{mor:shwsht}
$$\rsect(U;\shw(\shm,F))\to \rsect\bl U;\RD_X(\sht(F,\shm))\br\simeq
\Hom\bl\rsect_c(U;\sht(F,\shm)),\C\br$$
is an isomorphism by Theorem~\ref{thm:KSduality}.
Hence \eqref{mor:shwsht} is an isomorphism,
which implies (a) and (c). 
\end{proof}

\begin{theorem}\label{th:constdualDmod}
Let $\shm\in\Derb_\hol(\D_X)$ and $F\in\Derb_\Rc(\C_X)$. 
Then assertions {\rm (a), (b), (c)} in~\eqref{cond:const} hold true.
\end{theorem}
This result will be proved in Corollary~\ref{cor:newconstruct} below. Note that it solves a conjecture 
in~\cite[Conjecture\;6.2]{KS03}.

Applying this result in the situation of Example~\ref{exa:OtOwCDb}~(ii), we get:
\begin{corollary}\label{cor:KSduality0} 
Let $M$ be a real analytic manifold, $X$ a complexification of $M$ 
and let $\shm\in\Derb_\hol(\D_X)$.
Then the two objects $\rhom[\D_X](\shm,\Cinf[M])$ and $\Dbtv_M\ltens[\D_X]\shm$ belong to 
$\Derb_\Rc(\C_M)$ and are dual to each other. 
Namely, we have
 $\RD_M\rhom[\D_X](\shm,\Cinf[M])\simeq\Dbtv_M\ltens[\D_X]\shm$. 
\end{corollary}

\begin{corollary}\label{cor:KSduality} 
Assume that $\shm\in\Derb_\hol(\D_X)$,
$F\in\Derb_{\Rc}(\C_X)$ and $\Supp(F)$ is compact.
Then the complexes 
$$\text{$\rsect(X; \rhom[\shd_X](\shm,F\wtens\sho_X))$ and 
$\rsect(X;\rhom[\iC_X](F,\Ovt_X\,[d_X])\ltens[\shd_X]\shm)$}$$
 have finite-dimensional cohomologies
 and~\eqref{eq:dualWTpair} induces a perfect pairing for all $i\in\Z$
\[
H^{-i}\rsect(X;\shw(\shm,F)) \tens H^i\rsect(X;\sht(F,\shm))\to\C.
\]
\end{corollary}
\begin{remark}
It follows immediately from~\cite{Ka78,Ka84} that (b), hence (a) and (c), are true when  
$F\in\Derb_{\Cc}(\C_X)$.
\end{remark}

In~\cite{BE04}, S.~Bloch and H.~Esnault  proved  directly a similar result on an algebraic curve $X$ when assuming that $\shm$ is a meromorphic connection with poles on a divisor $D$ and $F=\C_X$. 
They interpret the duality pairing by considering sections of the type 
$\gamma\tens\epsilon$, where $\gamma$ is a  cycle  with boundary on $D$ and $\epsilon$ is a horizontal section of the connection 
on $\gamma$ with exponential decay on $D$. Their work has been extended to  higher dimension  by M.~Hien~\cite{Hi09}.

\section[Tempered solutions]{Tempered solutions of D-modules}\label{section:tempered}

\subsection{Tempered de Rham and Sol functors}
\index{de Rham functor!tempered}%
\index{solution functor!tempered}%
Setting $\Ovt_X\eqdot\Omega_X\tens[\sho_X]\Ot$,  we define 
the tempered de Rham and solution functors  by
\glossary{$\drt_X$}%
\glossary{$\solt_X$}%
\begin{align*}
\drt_X &\cl \Derb(\D_X) \to \Derm(\iC_X),&\shm &\mapsto \Ovt_X \ltens[\D_X] \shm, \\
\solt_X &\cl \Derb(\D_X)^\op \to \Derp({\iC_X}), &\shm& \mapsto \rhom[\D_X] (\shm,\Ot).
\end{align*}
One has
\[
\sol_X \simeq \alpha_X\solt_X,
\quad
\dr_X \simeq \alpha_X\drt_X.
\]
For $\shm\in\Derb_\coh(\D_X)$, one has
\eq\label{eq:dualdrsol}
\solt_X(\shm) \simeq \drt_X(\Ddual_X\shm)[-d_X].
\eneq

The next result is a reformulation of a theorem of~\cite{Ka84} (see also~\cite[Th.~7.4.1]{KS01})

\begin{theorem}\label{thm:ifunct0}
Let $f\cl X\to Y$ be a morphism of complex manifolds. 
There is an isomorphism in $\Derb(\JfD^\rop_Y)$:
\eq\label{eq:funct1a}
&&\Ovt_X\ltens[\shd_X]\shd_{X\to Y}\,[d_X]\isoto \epb f \Ovt_Y\,[d_Y].
\eneq
\end{theorem}
\begin{proof}
Consider isomorphism~\eqref{eq:opbDbt} with $M=X_\R$ and $N=Y_\R$ and apply 
 $\scbul\ltens[\D_{Y^c}]\OO[Y^c]$. We get the result since 
 \eqn
&& \scbul\ltens[\D_{X\times X^c}]\D_{X\times X^c\to Y\times Y^c}\ltens[\D_{Y^c}]\OO[Y^c]\\
 &&\hspace{20ex}\simeq
 \scbul\ltens[\D_{X\times X^c}]\D_{X\times X^c\to Y\times Y^c}\ltens[\D_{Y^c}]\D_{Y^c\to\rmpt}\\
 &&\hspace{20ex}\simeq
 \scbul\ltens[\D_{X\times X^c}]\D_{X\times X^c\to Y\times Y^c}\ltens[\D_{Y\times Y^c}]\D_{Y\times Y^c\to Y}\\
&&\hspace{20ex}\simeq
 \scbul\ltens[\D_{X\times X^c}]\D_{X\times X^c\to Y}\\
&&\hspace{20ex}\simeq
 \scbul\ltens[\D_X]\D_{X\to Y}\ltens[\D_{X^c}]\OO[X^c].
 \eneqn
  \end{proof}
  Note that this isomorphism \eqref{eq:funct1a} 
is equivalent to the isomorphism
\eq\label{eq:funct1abis}
&&\shd_{Y\from X}\ltens[\shd_X]\Ot[X]\,[d_X]\isoto \epb f \Ot[Y]\,[d_Y]
\quad\text{ in $\Derb(\JfD_Y)$.}
\eneq
\begin{corollary}\label{cor:ifunct1}
Let $f\cl X\to Y$ be a morphism of complex manifolds and let
$\shn\in\Derb(\D_Y)$. Then \eqref{eq:funct1a} induces  the isomorphism
\eq
&&\drt_X(\Dopb f\shn)\, [d_X] \simeq \epb f \drt_Y(\shn) \,[d_Y]
\quad\text{in $\Derb(\iC_X)$.}\label{eq:funct1}
\eneq
\end{corollary}
\begin{proof}
Apply $\scbul\ltens[\opb{f}\D_Y]\opb{f}\shn$ to isomorphism~\eqref{eq:funct1a}.
\end{proof}

\Cor
For any complex manifold $X$, we have
$$\drt_X(\OO)\simeq\C_X[d_X].$$
\encor

\begin{corollary}\label{cor:ifunct5}
Let $f\cl X\to Y$ be a morphism of complex manifolds. There is a natural morphism  
\eq\label{eq:funct2e}
&&\opb{f}\Ovt_Y\ltens[\opb{f}\shd_Y]\shd_{Y\from X}\To\Ovt_X \mbox{ in }
\Derb(\JD_X^{\;\rop}).
\eneq
\end{corollary}
\begin{proof}
(i) Assume that $f$ is a closed embedding. We have
 \eqn
 \opb{f}\Ovt_Y\ltens[\opb{f}\shd_Y]\shd_{Y\from X}&\simeq&\epb{f}\reeim{f}(\opb{f}\Ovt_Y\ltens[\opb{f}\shd_Y]\shd_{Y\from X})\\
 &\simeq&\epb{f}(\Ovt_Y\ltens[\shd_Y]\reim{f}\shd_{Y\from X})\\
 &\simeq&\epb{f}\Ovt_Y\ltens[\opb{f}\shd_Y]\opb{f}\shd_{Y\from X}\\
 &\simeq&\Ovt_X\ltens[\shd_X]\shd_{X\to Y}\ltens[\opb{f}\shd_Y]\shd_{Y\from X}\,[d_X-d_Y]\\
 &\simeq&\Ovt_X.
 \eneqn
(ii) Assume that $f$  is submersive. We have
\eqn
\rhom[\shd_X^\rop](\shd_{Y\from X},\Ovt_X)&\simeq&\Ovt_X\ltens[\D_X]
\rhom[\shd_X^\rop](\shd_{Y\from X},\D_X)\\
&\simeq&\Ovt_X\ltens[\D_X]\D_{X\to Y}\,[d_Y-d_X]\\
&\simeq&\epb{f}\Ovt_Y\,[2d_Y-2d_X]\simeq\opb{f}\Ovt_Y.
\eneqn
Then use
\eqn
&&\rhom[\shd_X^\rop](\shd_{Y\from X},\Ovt_X)\ltens[\opb{f}\shd_Y]\shd_{Y\from X}
\to\Ovt_X.
\eneqn
\end{proof}

Note that morphism \eqref{eq:funct2e}
is equivalent to the morphism in $\Derb(\JD_X)$
$$\shd_{X\to Y}\ltens[\opb{f}\shd_Y]\opb{f}\Ot[Y]\To\Ot.$$

The next result is a kind of Grauert direct image theorem
for tempered holomorphic functions. It will be generalised to D-modules in Corollary~\ref{cor:ifunct2}. Its proof uses difficult results of functional analysis.

\begin{theorem}[{{\rm Tempered Grauert theorem}\;{\cite[Th.~7.3]{KS96}}}]\label{th:tGrauert}
\index{tempered Grauert theorem}%
Let $f\cl X\to Y$ be a morphism of complex manifolds, let $\shf\in\Derb_\coh(\OO)$ and assume that $f$ is proper on $\Supp(\shf)$. Then there is a natural isomorphism
\eq\label{eq:tGrauert1}
&&\reeim{f}(\Ot\ltens[\OO]\shf)\simeq\Ot[Y]\ltens[{\OO[Y]}]\reim{f}\shf.
\eneq
\end{theorem}
\begin{proof}[An indication on the proof]
 It is enough to prove that for any $G\in\rc[\C_Y]$, we have
\eq\label{eq:tGrauert2}
&&\rhom(\opb{f}G,\Ot\ltens[\OO]\shf)\simeq\rhom(G,\Ot[Y]\ltens[{\OO[Y]}]\reim{f}\shf).
\eneq
Since $\shf$ and $\reim{f}\shf$ are coherent, \eqref{eq:tGrauert2} is equivalent to
\eq\label{eq:tGrauert3}
&&\rhom(\opb{f}G,\Ot)\ltens[\OO]\shf\simeq\rhom(G,\Ot[Y])\ltens[{\OO[Y]}]\reim{f}\shf.
\eneq
Such a formula is proved in~\cite[Th.~7.3]{KS96}.
\end{proof}

\begin{corollary}[{\cite[Th.~7.4.6]{KS01}}]\label{cor:ifunct2}
Let $f\cl X\to Y$ be a morphism of complex manifolds.
Let $\shm\in\Derb_\qgood(\D_X)$ and assume that $f$  is proper on $\Supp(\shm)$. 
Then there is an  isomorphism in $\Derb(\iC_Y)$
\eq\label{eq:ifunct4}
\drt_Y(\Doim f\shm) &\isoto& \roim f\drt_X(\shm).\label{eq:funct2}
\eneq
\end{corollary}
\begin{proof}
Applying the functor $\reeim{f}(\scbul\ltens[\D_X]\shm)$ to the morphism~\eqref{eq:funct2e} 
we obtain the morphism in~\eqref{eq:funct2}. To check it is an isomorphism, we reduce to the case where $\shm=\D_X\tens[\OO]\shf$ 
with  a coherent $\OO$-module $\shf$ such that $f$ is proper on $\Supp(\shf)$. Then we apply Theorem~\ref{th:tGrauert}.
\end{proof}

\begin{corollary}\label{cor:ifunct2b}
Let $f$ and $\shm$ be as in Corollary~\ref{cor:ifunct2}. Then we have the isomorphism 
\eq\label{eq:ifunct44}
&&\Doim{f}(\Ot\Dtens\shm)\simeq \Ot[Y]\Dtens\Doim{f}\shm\mbox{ in }\Derb(\JD_Y).
\eneq
\end{corollary}
\begin{proof}
We have 
\eqn
\Ovt_Y\Dtens\Deim{f}\shm&\simeq&\Ot[Y]\ltens[\D_Y](\D_Y\Dtens\Deim{f}\shm)\\
&\simeq&\Ot[Y]\ltens[\D_Y]\Deim{f}(\Dopb{f}\D_Y\Dtens\shm)\\
&\simeq&\drt_Y(\Deim{f}(\D_{X\to Y}\Dtens\shm)),
\eneqn
where the second isomorphism follows from the projection formula~\eqref{eq:DDprojform}. Applying Corollary~\ref{cor:ifunct2},
we obtain
\eqn
\Ovt_Y\Dtens\Deim{f}\shm&\simeq&\roim{f}(\Ovt_X\ltens[\D_X](\D_{X\to Y}\Dtens\shm)).
\eneqn
On the other-hand, we have 
\eqn
\Ovt_X\ltens[\D_X](\D_{X\to Y}\Dtens\shm)&\simeq&(\Ovt_X\Dtens\shm)\ltens[\D_X]\D_{X\to Y}.
\eneqn
Therefore,
\eqn
\Ovt_Y\Dtens\Deim{f}\shm&\simeq&\Doim{f}(\Ovt_X\Dtens\shm).
\eneqn
To conclude, use the equivalence of categories $\Derb(\D_Y^\rop)\simeq\Derb(\D_Y)$ given by
$\shm^\mop = \Omega_X\ltens[\OO]\shm$. 
\end{proof}

\begin{remark}\label{rem:isofunct1}
If one replaces~\eqref{eq:funct1a}  with its non-tempered version, then 
the formula is no more true, contrarily to 
 isomorphism~\eqref{eq:funct2} which remains true by Theorem~\ref{th:Dprojform}.
\end{remark}

\subsection{Localization along a hypersurface}
In order to prove Theorem~\ref{thm:ifunct4} below,
a generalized form of the Riemann-Hilbert correspondence
for regular holonomic D-modules,  we need some lemmas. 

If $S\subset X$ is a closed hypersurface, denote by $\OO(*S)$
\glossary{$\OO(*S)$}%
 the sheaf of meromorphic functions with poles at $S$. It is a regular holonomic $\D_X$-module (see Definition~\ref{def:reghol} below) and it is a flat $\OO$-module. For $\shm\in\Derb(\JD_X)$, set
\eqn
&&\shm(*S) = \shm \Dtens \OO(*S).
\eneqn

\begin{lemma}\label{le:GrotDR1}
Let $S$ be a closed complex hypersurface in $X$. There are isomorphisms
\eq\ba{rcl}
\Ot(*S)&\simeq&\rihom(\C_{X\setminus S},\Ot)\mbox{ in }\Derb(\JD_X),\\[1ex]
\OO(*S)&\simeq&\rhom[\iC_X](\C_{X\setminus S},\Ot)\mbox{ in }\Derb(\D_X).
\ea\eneq
\end{lemma}
\begin{proof}
(i) The second isomorphism follows from the first one by applying the functor $\alpha_X$. 

\spa
(ii) By taking  the Dolbeault resolution of $\Ot$ we are reduced to prove a similar result with $\Dbt_{X_\R}$ instead of $\Ot$. More precisely, consider a real analytic manifold $M$, a real analytic map $f\cl M\to\C$. Set  $S=\{f=0\}$  and denote by 
$j\cl (M\setminus S)\into M$ the open embedding.  
Define the sheaf $\rA_M[1/f]$ as the inductive limit of the sequence of embeddings $\rA_M\To[f]\rA_M\To[f]\cdots$. 
Equivalently, $\rA_M[1/f]$ is the subsheaf of $\oim{j}\opb{j}\rA_M$ consisting of sections $u$ such that there locally exists an integer $m$ with  $f^m\cdot u\in \rA_M$. Set
\eqn
&& \Dbt_M[1/f]\eqdot\Dbt_M\tens[\rA_M](\rA_M[1/f]).
\eneqn
 (Note that $\Dbt_M[1/f]$ is isomorphic to the inductive limit of the sequence of morphisms  $\Dbt_M\To[f]\Dbt_M\To[f]\cdots$.)
It is enough to prove the isomorphism
\eq\label{eq:loja}
&&\Dbt_M[1/f]\simeq\rihom(\C_{M\setminus S},\Dbt_M),
\eneq
or, equivalently, the isomorphism for any open relatively compact subanalytic subset $U$ of $M$
\eq\label{eq:loja2}
\sect(U;\Dbt_\Msa[1/f])&\simeq&\sect(U\setminus S;\Dbt_\Msa).
\eneq
This follows from the fact that
$f\cl \sect(U\setminus S;\Dbt_\Msa)\to\sect(U\setminus S;\Dbt_\Msa)$
is bijective. (See also Lojasiewicz~\cite{Lo59}.)
\end{proof}
In the sequel, we set for  a closed complex analytic hypersurface $S$
\eq\label{eq:Ot*D}
&&\Ot(*S)\eqdot\Ot\Dtens\OO(*S)\simeq\rihom(\C_{M\setminus S},\Ot).
\eneq
\begin{lemma}\label{le:GrotDR2}
Let $S$ be a closed complex hypersurface in $X$. There are isomorphisms
\eq
&&\hs{2ex}\Omega_X\ltens[\D_X]\Ot(*S)\isoto\Omega_X\ltens[\D_X]\OO(*S)\simeq\rhom(\C_{X\setminus S},\C_X)\,[d_X].
\eneq
\end{lemma}
\begin{proof}
It follows from Lemma~\ref{le:GrotDR1} that 
\eqn
\Omega_X\ltens[\D_X]\Ot(*S)
&\simeq&\rihom(\C_{X\setminus S},\Omega_X\ltens[\D_X]\Ot).
\eneqn
Then the result follows from the isomorphisms 
\eqn
\Omega_X\ltens[\D_X]\Ot&\simeq&\Omega_X\ltens[\D_X]\OO\simeq\C_X\,[d_X].
\eneqn
\end{proof}

\section{Regular holonomic D-modules}\label{section:reghol}

\subsection{Regular normal form for holonomic modules}
For the notion of 
\index{D-module!regular holonomic}%
\index{regular holonomic D-module}%
regular holonomic D-modules, refer e.g.\ to \cite[\S5.2]{Ka03} and~\cite{KK81}.
\begin{definition}\label{def:reghol}
Let $\shm$ be a holonomic $\D_X$-module,
 $\Lambda$ its characteristic variety in $T^*X$ and
$\shi_\Lambda$ the ideal of $\gr(\D_X)$ of functions vanishing on $\Lambda$.
We say that $\shm$ is {\em regular} if
there exists locally a good filtration on 
$\shm$ such that $\shi_\Lambda\cdot\gr(\shm)=0$. 
\end{definition}
One can prove that the full subcategory $\mdrh[\D_X]$ of $\mdc[\D_X]$ consisting of regular holonomic $\D_X$-modules is 
\glossary{$\mdrh[\D_X]$}%
\glossary{$\Derb_{\reghol}(\D_X)$}%
a thick abelian subcategory, stable by duality. 
Denote by  $\Derb_{\reghol}(\D_X)$ the full 
subcategory of $\BDC(\D_X)$ whose objects have regular
holonomic cohomologies. Then $\Derb_{\reghol}(\D_X)$ is triangulated.

For a coherent $\D_X$-module $\shm$, denote by $\SSupp(\shm)$ 
\glossary{$\SSupp(\shm)$}%
the set of $x\in X$ such that $\shm$ is not 
a coherent $\OO$-module on a neighborhood of $x$.

\begin{definition}\label{def:regnormal}
Let $X$ be a complex manifold and $D\subset X$ a normal crossing divisor. 
We say that a holonomic $\D_X$-module $\shm$ has 
\index{normal form!regular}%
\emph{regular  normal form} along $D$ if locally on $D$, 
for a local coordinate system $(z_1,\dots,z_n)$ on $X$  such that $D=\{z_1\cdots z_r=0\}$,
$\shm\simeq\shd_X/\shi_\lambda$ for $\lambda=(\lambda_1,\dots,\lambda_r)\in
(\C\setminus\Z_{\geq0})^r$.
Here, $\shi_\lambda$ is the left ideal generated by the operators 
$(z_i\partial_i-\lambda_i)$ and $\partial_j$ for  $i\in \{1,\dots,r\}$, $j\in\{r+1,\dots,n\}$.
\end{definition}
One shall be aware that the property of being of normal form is not stable by duality. 
Note that, for $\lambda=(\lambda_1,\dots,\lambda_r)\in\C^m$,
$\shd_X/\shi_\lambda\isoto (\shd_X/\shi_\lambda)(*D)$ if and only if $\lambda_i\in\C\setminus\Z_{\geq0}$ for any $i\in \{1,\dots,r\}$.

Of course, if a holonomic $\D_X$-module has regular normal form, then it is regular holonomic. 

\begin{lemma}\label{le:normaDR}
Let $\shl$ be a holonomic module with regular normal form along $D$.   Then
we have the natural isomorphism $\sol_X(\shl)\tens\C_{X\setminus D}\isoto\sol_X(\shl)$. 
\end{lemma}
\begin{proof}
It is enough to prove that $\sol_X(\shl)\vert_D\simeq0$. In a local coordinate system $(z_1,\dots,z_n)$ as in Definition~\ref{def:regnormal}, set $Z_i=\{z_i=0\}$. 
Setting $P_i=z_i\partial_i-\lambda_i$ with $\lambda_i\in\C\setminus\Z_{\geq0}$,  it is enough to check that $P_i$ induces an isomorphism $P_i\cl\OO\vert_{Z_i}\isoto\OO\vert_{Z_i}$, which is clear. 
\end{proof}

\begin{lemma}\label{lem:reduxreg}
Let $P_X(\shm)$ be a statement concerning a complex manifold $X$ and a regular holonomic object $\shm\in\BDC_\reghol(\D_X)$. Consider the following conditions.
\bna
\item
Let $X=\Union\nolimits_{i\in I}U_i$ be an open covering. Then $P_X(\shm)$ is true if and only if $P_{U_i}(\shm|_{U_i})$ is true for any $i\in I$.
\item
If $P_X(\shm)$ is true, then $P_X(\shm[n])$ is true for any $n\in\Z$.
\item
Let $\shm'\to\shm\to\shm''\to[+1]$ be a distinguished triangle in $\BDC_\reghol(\D_X)$. If $P_X(\shm')$ and $P_X(\shm'')$ are true, then $P_X(\shm)$ is true.
\item
Let $\shm$ and $\shm'$ be regular holonomic $\D_X$-modules. 
If $P_X(\shm\dsum\shm')$ is true, then $P_X(\shm)$ is true.
\item Let $f\colon X\to Y$ be a projective morphism and let $\shm$be a good regular holonomic $\D_X$-module. If $P_X(\shm)$ is true, then $P_Y(\Doim f\shm)$ is true.
\item If $\shm$ is a regular holonomic $\D_X$-module with a regular normal form along a normal crossing divisor of $X$, then $P_X(\shm)$ is true.
\ee
If conditions {\rm (a)--(f)} are satisfied, 
then $P_X(\shm)$ is true for any complex manifold $X$ and any $\shm\in\BDC_\reghol(\D_X)$.
\end{lemma}
\begin{proof}[Sketch of proof]
(i) If $D$ is a normal crossing hypersurface of $X$ and $\shm$ is a regular holonomic $\shd_X$-module satisfying\\
$\scbul$ $\shm\simeq \shm(*D)$,\\
$\scbul$  $\SSupp(\shm)\subset D$,\\
 then, locally on $X$, there exists a filtration 
 \eqn
 &&\shm=\shm_0\supset\shm_1\supset\cdots\supset\shm_l\supset\shm_{l+1}=0
 \eneqn
 such that $\shm_j/\shm_{j+1}$ has regular normal form.
 It follows that in this case, $P_X(\shm)$ is true.

\spa
(ii) Let us take a closed complex analytic subset $Z$ of $X$
such that the support of $\shm$ is contained in $Z$.
We argue by induction on the dimension  $m$ of $Z$.
There exists a morphism $f\cl  W\to Z$
such that\\[.5ex]
(1) $W$ is non singular with dimension $m$,\\
(2) $f$ is projective,\\
(3) there exists a closed complex analytic subset $S$
of $Z$ with dimension $<m$ such that\\
\hs{4ex}\textbullet\ $f^{-1}(Z\setminus S)\to Z\setminus S$ is an isomorphism,\\
\hs{4ex}\textbullet\  $D\seteq\opb{f}S$ is a normal crossing hypersurface of $W$,\\
\hs{4ex}\textbullet\ 
{$\SSupp(H^{m-d_X}\Dopb{g}\shm)\subset D$, 
where $g$ is the composition $W\To[f]Z\hookrightarrow X$.}

\medskip

We have
 \eqn
 &&\bl\Dopb{g}\shm\br(*D)\simeq  \bl \bl H^{m-d_X}\Dopb{g}\shm\br(*D)\br[d_X-m].
 \eneqn
Then by step (i), $P_{W}\bl(\Dopb{g}\shm)(*D)\br$ is true.
Hence $ P_X\bl\Doim{g}\bl(\Dopb{g}\shm)(*D)\br\br$ is true.
Let us consider a distinguished triangle
\eqn
&&\shm\To \Doim{g}\bl(\Dopb{g}\shm)(*D)\br[m-d_X]\To \shn\To[1].
\eneqn
Since $\Supp(\shn)\subset S$, $P_X(\shn)$ is true by the induction hypothesis.
Hence $P_X(\shm)$ is true.
\end{proof}
\begin{remark}
In fact, we could remove condition (d) in the regular case. We keep it by analogy with the irregular case (Lemma~\ref{lem:reduxreg}). 
\end{remark}

\subsection{Real blow up}\label{subsection:realblowup}
A classical tool in the study of differential equations is the
\index{real blow up}%
 real blow up, and we shall use this construction in the proof of 
Theorems~\ref{thm:ifunct4}, \ref{th:irrRH1}
and in the definition of normal form given in \S\,\ref{subsection:normalform}.

Recall that $\C^\times$ denotes $\C\setminus\{0\}$ and $\R_{>0}$ the multiplicative group of positive real numbers. Consider the 
action of $\R_{>0}$ on $\C^\times\times\R$:
\eq\label{eq:muaction}
&&\R_{>0}\times(\C^\times\times\R)\to\C^\times\times\R,\quad (a,(z,t))\mapsto (az,\opb{a}t)
\eneq
and set
\eqn
&&\tw\C^\tot=(\C^\times\times\R)/ \R_{>0},\,
\tw\C^{\geq0}=(\C^\times\times\R_{\geq0})/ \R_{>0},
\tw\C^{>0}=(\C^\times\times\R_{>0})/ \R_{>0}.
\eneqn
One denotes by  $\varpi^\tot$ the map:
\eq\label{eq:maptwom}
&&\varpi^\tot\cl \tw\C^\tot\to \C, \quad (z,t)\mapsto tz.
\eneq
Then we have
$$\tw\C^\tot\supset \tw\C^{\geq0}\supset \tw\C^{>0}\isoto\C^\times.$$

Let $X=\C^n\simeq\C^r\times\C^{n-r}$ and let $D$ be the divisor $\{z_1\cdots z_r=0\}$. Set 
\eqn
&&\twX^\tot=(\tw\C^\tot)^r\times\C^{n-r},\, 
\twX^{>0}=(\tw\C^{>0})^r\times\C^{n-r},\,
 \twX=(\tw\C^{\geq0})^r\times\C^{n-r}.
\eneqn
Then $\twX$ is the closure of $\twX^{>0}$ in $\twX^\tot$.
\glossary{$\twX^\tot$}%
\glossary{$\twX^{>0}$}%
\glossary{$\twX$}%
\glossary{$\varpi$}%
The map $\varpi^\tot$ in~\eqref{eq:maptwom} defines the map 
\eqn
&&\varpi\cl \twX\to X.
\eneqn
The map $\varpi$ is proper and induces an isomorphism
\eqn
&&\varpi\vert_{\twX^{>0}}\cl \twX^{>0}=\vpi^{-1}(X\setminus D)\isoto X\setminus D.
\eneqn

We call $\twX$ the {\em real blow up} along $D$.

\begin{remark}
The real manifold $\twX$ (with boundary)
as well as the map $\varpi\cl \twX\to X$ may be intrinsically defined for a complex manifold $X$ and a normal crossing divisor $D$, but $\twX^\tot$ is only intrinsically defined as a germ of a manifold in a neighborhood of  $\tw X$. 
\end{remark}
We set
\eq\ba{rcl}\label{eq:sheafDbtwX}
\Dbt_\twX&\eqdot&\ihom(\C_{\twX^{>0}},\Dbt_{\twX^\tot})\vert_{\twX}\\
&\simeq&\epb{\varpi}\ihom(\C_{X\setminus D},\Dbt_{X_\R}),
\ea\eneq
where the last isomorphism follows from Lemma~\ref{le:sacompactif}.
Note that
$\Dbt_\twX$ is an object of $\II[\opb{\varpi}\D_{X}\tens\opb{\varpi}\D_{X^c}]$.

Now we set
\glossary{$\Ot[\twX]$}%
\glossary{$\At$}%
\glossary{$\DA_\twX$}%
\eq\ba{rcl}\label{eq:sheafOtwX}
\Ot[\twX]&\eqdot&\rhom[\opb{\varpi}\D_{X^c}](\opb{\varpi}\OO[X^c],\Dbt_\twX),\\
\At&\eqdot&\alpha_\twX\Ot[\twX],\\
\DA_\twX&\eqdot&\At\ltens[\opb{\varpi}\OO]\opb{\varpi}\D_X.
\ea\eneq
Then $\At$ and $\DA_\twX$ are concentrated in degree $0$,  and hence
they are sheaves of $\C$-algebras on $\twX$. 
Indeed, $\At$ is the subsheaf of
$\oim{j}\opb{j}\opb{\vpi}\OO$ consisting of holomorphic functions tempered
at any point of $\twX\setminus\twX^{>0}=\opb{\vpi}(D)$.
Here, $j\cl\twX^{>0}\into \twX$ is the inclusion.
Clearly, $\Dbt_\twX$ is an object of $\II[\DA_\twX\tens\opb{\varpi}\D_{X^c}]$,
and hence 
\eq
&&\Ot[\twX]  \mbox{ is an object of } \Derb(\JD^\At).
\eneq

By using~\eqref{eq:sheafDbtwX}, we get the isomorphism 
\eq\label{eq:sheafOttwX}
&&\Ot[\twX]\simeq\epb{\varpi}\Ot(*D)\mbox{ in }\Derb({\rm I}\opb{\varpi}\D_X).
\eneq

Recall that the map $\varpi$ is proper, and hence $\reeim{\varpi}\simeq\roim{\varpi}$.

\begin{lemma}\label{le:eimebpvarpi}
Let $\shf\in\Derb(\iC_M)$. If $\shf\isoto\rihom(\C_{X\setminus D},\shf)$, then we have
$\reeim{\varpi}\epb{\varpi}\shf\isoto\shf$.
\end{lemma}
\begin{proof}
One has
\eqn
\reeim{\varpi}\epb{\varpi}\shf&\simeq&\roim{\varpi}\epb{\varpi}\rihom(\C_{X\setminus D},\shf)\\
&\simeq&\roim{\varpi}\rihom(\opb{\varpi}\C_{X\setminus D},\epb{\varpi}\shf)\\
&\simeq&\rihom(\reeim{\varpi}\opb{\varpi}\C_{X\setminus D},\shf)\simeq\shf.
\eneqn
\end{proof}
As a corollary, we obtain the isomorphism
\eq
&&\roim{\varpi}\Ot[\twX]\simeq\Ot(*D)\mbox{ in }\Derb(\JD_X).
\label{eq:Otimage}
\eneq

For $\shn\in\Derb(\D^\At)$, we set
\eq
&&\drt_\twX(\shn)=\Ovt_\twX\ltens[\DA_\twX]\shn,\\
&&\solt_\twX(\shn)=\rhom[\DA_\twX](\shn,\Ot[\twX]).
\eneq
\glossary{$\Ovt_\twX$}%
Here $\Ovt_\twX\seteq\vpi^{-1}\Omega_X\tens[\vpi^{-1}\OO]\Ot[\twX]$,
an object of $\Derb\bl\II[(\D^\At)^\rop]\br$.

\smallskip
For  $\shm\in \Derb(\D_X)$ we set:
\glossary{$\shm^\tA$}%
\eq\label{eq:MA}
&&\shm^\tA\eqdot\DA_\twX\ltens[\opb{\varpi}\D_X]\opb{\varpi}\shm\in \Derb(\D^\At).
\eneq

\Lemma\label{lem:twXX}
 For  $\shm\in \Derb(\D_X)$, we have
\eq
&&\epb{\varpi}\drt_X(\shm(*D))\simeq \drt_\twX(\shm^\tA),\\
&&\roim{\varpi}\drt_\twX(\shm^\tA)\simeq\drt_X(\shm(*D)).
\eneq
\enlemma
\Proof
By \eqref{eq:sheafOttwX}, we have
\begin{align*}
\epb{\varpi}\drt_X(\shm(*D))
&\simeq
\epb{\varpi}\bl\Ovt_X\ltens[{\D_X}] \shm(*D)\br\\
&\simeq
\epb{\varpi}\bl\Ovt_X(*D)\ltens[{\D_X}] \shm\br\\
&\simeq
(\epb{\varpi}\Ovt_X(*D))\ltens[{\vpi^{-1}\D_X}]\vpi^{-1} \shm\\
&\simeq
\Ovt_\twX\ltens[\D^\At]\D^\At\ltens[{\vpi^{-1}\D_X}]\vpi^{-1} \shm\\
&\simeq
\Ovt_\twX\ltens[\D^\At]\shm^\tA\simeq\drt_\twX(\shm^\tA).
\end{align*}
Hence we obtain the first isomorphism.

Since $$\drt_X(\shm(*D))\isoto\rihom\bl\C_{X\setminus D},\drt_X(\shm(*D))\br,$$
the second isomorphism follows from Lemma~\ref{le:eimebpvarpi}.
\QED

\begin{proposition}\label{pro:normalfLA}
Let $\shl$ be a holonomic $\D_X$-module with regular normal form along $D$. Then, locally on $\twX$, 
\eqn
&&\shl^\tA\simeq \At \simeq\OO^\tA\mbox{ in }\Derb(\D^\At).
\eneqn
\end{proposition}
\begin{proof}
Let us keep the notations of  Definition~\ref{def:regnormal}.
We may assume that $\shl=\D_X/\shi_\lambda$.
Since $z^\lambda\eqdot\prod\limits_{i=1}^r z_i^{\lambda_i}$ is a locally invertible section of  $\At$, 
the result follows from
\eqn
&&(z_i\partial_i-\lambda_i)z^{\lambda}=z^{\lambda}z_i\partial_i.
\eneqn
\end{proof}

\subsection{Regular Riemann-Hilbert correspondence}

We shall first prove the regularity theorem for
regular holonomic D-modules,  namely, any solution of such a D-module
is tempered.
\begin{theorem}\label{thm:regularity}
Let $\shm\in\Derb_\reghol(\D_X)$. 
Then there are isomorphisms:
\eq
&&\drt_X(\shm) \isoto \dr_X(\shm)  \mbox{ in }\Derb(\iC_X),\label{eq:RHiso2}\\
&&\solt_X(\shm) \isoto \sol_X(\shm)\mbox{ in }\Derb(\iC_X).\label{eq:RHiso2sol}
\eneq
\enth

\Proof
(i) Note that, thanks to~\eqref{eq:dualdrsol}, the  isomorphism in~\eqref{eq:RHiso2sol}  is equivalent to the isomorphism in~\eqref{eq:RHiso2}
 for $\Ddual_X\shm$.
We shall only prove~\eqref{eq:RHiso2}.

\spa
(ii) We shall apply Lemma~\ref{lem:reduxreg}. Denote by $P_X(\shm)$ the statement which asserts that the morphism in \eqref{eq:RHiso2} is an isomorphism. 

\vs{1ex}\noi
(a)--(d)  of this lemma are clearly satisfied. 

\vs{1ex}\noi
(e) follows from isomorphism~\eqref{eq:funct2}
in Corollary~\ref{cor:ifunct2} and its non-tempered version, isomorphism~\eqref{eq:oimdr} in Theorem~\ref{th:Dprojform}.

\medskip\noi
(f) Let us check property (f).
Let $\shm$ be a holonomic $\D_X$-module with regular normal form along 
a normal crossing divisor $D$.

We want to prove the isomorphism $\drt_X(\shm)\isoto\alpha_X\drt_X(\shm)$.
Since $\roim{\varpi}\drt_\twX(\shm^\tA)\isoto\drt_X(\shm)$ by 
Lemma~\ref{lem:twXX}
and 
since  $\roim{\varpi}$ commutes with $\alpha$, we are reduced to prove the isomorphism 
\eqn
&&\drt_\twX(\shm^\tA)\isoto\alpha_\twX\drt_\twX(\shm^\tA).
\eneqn
This is a local problem on $\twX$ and we may apply Proposition~\ref{pro:normalfLA}. 
Hence it is enough to show
\eqn
&&\drt_\twX(\OO^\tA)\isoto\alpha_\twX\drt_\twX(\OO^\tA),
\eneqn
which follows from
$$
\drt_\twX(\OO^\tA)\simeq\C_\twX[d_X].$$

This completes the  proof of property (f).
\QED

The following theorem is a generalized form of the Riemann-Hilbert correspondence
for regular holonomic D-modules (see Remark~\ref{rem:Bj}).
\index{Riemann-Hilbert correspondence!generalized regular}%
\begin{theorem}[{Generalized regular Riemann-Hilbert correspondence}]\label{thm:ifunct4}
Let $\shm\in\Derb_\reghol(\D_X)$. 
There is an isomorphism functorial in $\shm$
\eq
&&\Ot \Dtens \shm \isoto \rihom(\solt_X(\shm),\Ot)\mbox{ in }\Derb(\JD_X).\label{eq:RHiso1}
\eneq
\end{theorem}
\Proof
(i) The morphism in~\eqref{eq:RHiso1} is obtained by adjunction from 
the composition of the morphisms
\eq
\Ot[X] \Dtens \shm \ltens\rihom[\shd_X](\shm,\Ot)
&\to&\Ot\ltens[\beta\OO]\Ot\to\Ot.
\label{eq:RHmor}
\eneq

\spa
(ii) We shall apply Lemma~\ref{lem:reduxreg}. Denote by $P_X(\shm)$ the statement which asserts that the morphism in~\eqref{eq:RHiso1} is an isomorphism. 

\spa
Properties (a)--(d) of this lemma are clearly satisfied. 

\spa
(e)  By Corollary~\ref{cor:ifunct2b}, we have 
\eq
&& \Ot[Y]\Dtens\Doim{f}\shm\simeq \Doim{f}(\Ot\Dtens\shm).\label{eq:OtXY}
\eneq

On the other hand we have
$$\solt_Y(\Doim{f}\shm)\simeq  \reeim{f}\,\solt_X(\shm)[d_X-d_Y]$$
by \eqref{eq:dualdrsol},
\eqref{eq:ifunct4} and Theorem~\ref{th:oimopbDdual} (i).
Hence we have
\eqn
&&\rihom(\solt_Y(\Doim{f}\shm), \Ot[Y])\\
&&\hs{10ex}\simeq\rihom\bl\reeim{f}\,\solt_X(\shm)[d_X-d_Y], \Ot[Y]\br\\
&&\hs{10ex}\simeq\roim{f}\rihom\bl\solt_X(\shm)[d_X-d_Y], \epb f\Ot[Y]\br.
\eneqn
By \eqref{eq:funct1abis}, we have
$$\epb f \Ot[Y]\simeq 
\shd_{Y\from X}\ltens[\shd_X]\Ot[X]\,[d_X-d_Y].$$
Hence we have
\eqn
&&\rihom(\solt_Y(\Doim{f}\shm), \Ot[Y])\\
&&\hs{10ex}\simeq\roim{f}\rihom(\solt_X(\shm),\shd_{Y\from X}\ltens[\shd_X]\Ot)\\
&&\hs{10ex}\simeq\roim{f}\bl\shd_{Y\from X}\ltens[\shd_X]\rihom(\solt_X(\shm),\Ot)\br\\
&&\hs{10ex}\simeq\Doim{f}\rihom(\solt_X(\shm),\Ot).
\eneqn
Combining with \eqref{eq:OtXY}, we finally  obtain
\eqn
\Ot[Y]\Dtens\Doim{f}\shm&\simeq&\Doim{f}(\Ot\Dtens\shm)\\
&\simeq &\Doim{f}\rihom(\solt_X(\shm),\Ot)\\
&\simeq&\rihom(\solt_Y(\Doim{f}\shm), \Ot[Y]).
\eneqn
Here the second isomorphism follows from $P_X(\shm)$.

\vs{2ex}\noindent
(f) Let us check property (f) for~\eqref{eq:RHiso1}.
Hence, we assume that $\shm$ has regular normal form along $D$. 

 By Lemmas~\ref{le:normaDR} and~\ref{le:GrotDR1} we have
\eqn
\rihom\bl\sol_X(\shm),\Ot\br
&\simeq&\rihom\bl\sol_X(\shm)\tens\C_{X\setminus D},\Ot\br\\
&\simeq&\rihom\bl\sol_X(\shm),\rihom(\C_{X\setminus D},\Ot)\br\\
&\simeq&\rihom\bl\sol_X(\shm),\roim{\vpi}\Ot[\twX]\br\\
&\simeq&\roim{\vpi}\rihom\bl\vpi^{-1}\sol_X(\shm),\Ot[\twX]\br\\
&\simeq&\roim{\vpi}\rihom\bl\sol_\twX(\shm^\tA),\Ot[\twX]\br.
\eneqn
Here the last isomorphism follows from
$$\C_{\twX^{>0}}\ltens \vpi^{-1}\sol_X(\shm)
\simeq \C_{\twX^{>0}}\ltens \sol_\twX(\shm^\tA).$$

On the other-hand, we have
\eqn
\Ot\Dtens\shm\simeq\Ot(*D)\Dtens\shm&\simeq&
(\roim{\vpi}\Ot[\twX])\Dtens \shm\\
&\simeq&\roim{\vpi}(\Ot[\twX]\ltens[\vpi^{-1}\OO]\vpi^{-1}\shm)\\
&\simeq&\roim{\vpi}(\Ot[\twX]\ltens[\At]\shm^\tA).
\eneqn
Hence it is enough to show that

\eq
&&\Ot[\twX]\ltens[\At]\shm^\tA\to \rihom(\sol_\twX(\shm^\tA),\Ot[\twX])
\label{eq:twXB}
\eneq
is an isomorphism. Note that this morphism is obtained 
from a similar morphism  to \eqref{eq:RHmor} by adjunction.
By Proposition~\ref{pro:normalfLA}, $\shm^\tA$ is locally isomorphic to 
$\tA_\twX$.
Then $\sol_\twX(\shm^\tA)\simeq \C_\twX$,
and it is obvious that \eqref{eq:twXB} is an isomorphism. 

\QED
\begin{remark}\label{rem:Bj}
Isomorphism~\eqref{eq:RHiso2} already appeared in~\cite{Ka84}. Isomorphism~\eqref{eq:RHiso1} (with a different formulation) is essentially due to Bj\"ork~\cite[Th.~7.9.11]{Bj93}.
  \end{remark}

Applying the functor $\alpha_X$ to isomorphism~\eqref{eq:RHiso1}, we get the Riemann-Hilbert correspondence for regular holonomic D-modules:

\begin{corollary}[{Regular Riemann-Hilbert correspondence \cite{Ka80}}]\label{cor:RegRH3}
Let $\shm\in\Derb_\reghol(\D_X)$. There is an isomorphism in $\Derb(\D_X):$
\eq\label{eq:RHiso3}
&& \shm \simeq\rhom[\iC_X](\sol_X(\shm),\Ot).
\eneq
\end{corollary}
\index{Riemann-Hilbert correspondence!regular}%
\begin{corollary}\label{cor:ifunct4}
Let $\shm\in\Derb_\reghol(\D_X)$ and let $\shl\in\Derb(\D_X)$. 
Then isomorphism~\eqref{eq:RHiso1} induces the isomorphism 
\eq
&&\drt(\shl\Dtens\shm)\simeq \rihom(\solt_X(\shm),\drt(\shl)).\label{eq:RHiso1cor}
\eneq
\end{corollary}
\begin{proof} 
We have
\eqn
\drt(\shl\Dtens\shm)&=&\Ovt_X\ltens[\D_X](\shl\Dtens\shm)\\
&\simeq&(\Ovt_X\Dtens\shm)\ltens[\D_X]\shl\\
&\simeq&\rihom(\solt_X(\shm),\Ovt_X)\ltens[\D_X]\shl\\
&\simeq&\rihom(\solt_X(\shm),\Ovt_X\ltens[\D_X]\shl).
\eneqn
Here, the last isomorphism follows from Theorem~\ref{th:betaaihom}, using the fact that 
$\solt(\shm)\isoto\sol(\shm)$. 
\end{proof}
As an application of isomorphism~\eqref{eq:RHiso2sol}, we get:
\begin{corollary}
Let $\shm\in\Derb_\reghol(\D_X)$ and let $F\in\Derb_\Rc(\C_X)$. Then we have the  natural isomorphism
\eqn
&&\rhom[\shd_X](\shm,\rhom[\iC_X](F,\Ot))\isoto\rhom[\shd_X](\shm,\rhom(F,\OO)).
\eneqn
\end{corollary}
Let $M$ be a real analytic manifold and $X$ a complexification of $M$. 
Choosing for $F$ the object $\RD'_X\C_M$, 
we get the isomorphism between the complexes of 
distribution solutions and hyperfunction solutions of $\shm$:
\eqn
&&\rhom[\shd_X](\shm,\Db_M)\isoto\rhom[\shd_X](\shm,\shb_M).
\eneqn

\begin{remark}\label{rem:RegRH}
Of course, isomorphism~\eqref{eq:RHiso1}  is no more true if one replaces $\Ot$ with $\OO$. For example, 
choosing $\shm=\OO(*Y)$ for $Y$ a closed hypersurface, the left-hand is the sheaf of meromorphic functions with poles on $Y$ and the right-hand side the sheaf of holomorphic functions with possibly essential singularities on $Y$. 
\end{remark}

\subsection{Integral transforms with regular kernels}
\index{integral transforms!regular}%
Consider morphisms of complex manifolds
\eqn
&&\ba{c}\xymatrix@C=6ex@R=3ex{
&S\ar[ld]_-f\ar[rd]^-g&\\
{X}&&{Y}.
}\ea\eneqn

\begin{notation}
(i) For $\shm\in\Derb(\D_{X})$ and $\shl\in\Derb(\D_S)$ one sets
\glossary{$\Dconv$}%
\eq\label{eq:Dcconv}
&&\shm\Dconv\shl\eqdot\Doim{g}(\Dopb{f}\shm\Dtens\shl).
\eneq
\glossary{$L\conv$}%
\glossary{$\Phi_L$}%
\glossary{$\Psi_L$}%
(ii) For $L\in \Derb(\iC_S)$,  $F\in \Derb(\iC_{X})$ and $G\in\Derb(\iC_Y)$ one sets
\eqn
&&\ba{l} L\conv G\eqdot\reeim{f}(L\tens\opb{g}G),\\
\Phi_L(G)=  L\conv G,\quad \Psi_L(F)=\roim{g}\rihom(L,\epb{f}F).
\ea\eneqn
\end{notation}
Note that we have a pair of adjoint functors 
\eq
&&\xymatrix{
\Phi_L\cl \Derb(\iC_Y)\ar@<0.5ex>[r]&\Derb(\iC_X)\ar@<0.5ex>[l]\cl \Psi_L
}\eneq

\begin{theorem}\label{th:7412}
Let $\shm\in\Derb_\qgood(\D_X)$, 
let $\shl\in\Derb_\reghol(\D_{S})$ and set $L\eqdot\sol_{S}(\shl)$. 
 Assume that $\opb{f}\Supp(\shm)\cap\Supp(\shl)$ is proper over $Y$ and that $\shl$ is good.
Then there is a natural isomorphism in $\Derb(\iC_Y)${\rm:}
\eq\label{eq:7412}
&&\Psi_L\bl\drt_X(\shm)\br\,[d_X-d_S]\simeq\drt_Y(\shm\Dconv\shl).
\eneq
\end{theorem}
 Note that any regular holonomic $\D$-module is good. 
\begin{proof}
Applying Corollaries~\ref{cor:ifunct1},~\ref{cor:ifunct2} and~\ref{cor:ifunct4}, we get:
\eqn
\drt_Y(\shm\Dconv\shl)&=& \drt_Y(\Doim{g}(\Dopb{f}\shm\Dtens\shl))\\
&\simeq&\roim{g} \drt_S(\Dopb{f}\shm\Dtens\shl)\\
&\simeq&\roim{g}  \rihom(\solt_S(\shl), \drt_S(\Dopb{f}\shm))\\
&\simeq&\roim{g}  \rihom(L, \epb{f}\drt_X(\shm))\,[d_X-d_S]\\
&=&\Psi_L(\drt_X(\shm))\,[d_X-d_S].
\eneqn
\end{proof}

By applying the functor $\RHom(G,\scbul)$ with $G\in\Derb(\iC_Y)$ to both sides of~\eqref{eq:7412}, one gets

\begin{corollary}[{\cite[Th.~7.4.13]{KS01}}]\label{cor:7412}
Let $\shm\in\Derb_\qgood(\D_X)$,  let $\shl\in\Derb_\reghol(\D_{S})$ and let  $L\eqdot\sol_{S}(\shl)$. 
 Assume that $\opb{f}\Supp(\shm)\cap\Supp(\shl)$ is proper over $Y$ and that $\shl$ is good.
Let $G\in\Derb(\iC_Y)$. Then one has the isomorphism
\eq\label{eq:7412b}
&&\RHom[\iC_X]\bl L\circ G,\drt_X(\shm)\br[d_X-d_S]\\
&&\hs{15ex}\simeq\RHom[\iC_Y]\bl G,\drt_Y(\shm\Dconv\shl)\br.\nn
\eneq
\end{corollary}
Note that a similar formula holds when replacing $\Ot[X]$ and $\Ot[Y]$  with their non tempered versions $\sho_X$ and $\sho_Y$ (and indsheaves with usual sheaves), but the hypotheses are different. Essentially, $\shm$ has to be coherent, $f$ non characteristic for $\shm$ 
and $\Dopb{f}\shm$ has to be transversal to the holonomic module $\shl$. On the other hand,  we do not need the regularity assumption on $\shl$. 
See~\cite{DS96} for such a non tempered formula (in a more particular setting). 

However, if one removes the hypothesis that the holonomic module $\shl$ is regular in Theorem~\ref{th:7412}, 
formula~\eqref{eq:7412} does not hold anymore  and we have to replace  $\Ot[X]$ with its enhanced version, as we shall see in the next sections.

\subsection{Irregular D-modules\,{\rm:} an example}\label{subsection:exairreg}
In this subsection we recall an example treated in~\cite{KS03} which 
emphasizes the role of the sheaf $\Ot$ 
in the study of irregular holonomic D-modules. 

Let $X=\C$ endowed with the holomorphic coordinate  $z$. Define 
\eqn
&&
U=X\setminus\{0\},\quad j\cl U\into X\mbox{ the open embedding}.
\eneqn
Consider the differential operator  $P=z^2\partial_z+1$
and the $\shd_X$-module $\shl\eqdot\shd_X\exp(1/z)\simeq\shd_X/\shd_X P$.

Notice first that $\Ot$ is concentrated in degree $0$ (since
$\dim X=1$) and it is a sub-indsheaf of $\OO$. Therefore the
morphism $H^0(\solt_X(\shl))\to H^0(\sol_X(\shl))\simeq\C_U$ is a monomorphism.
It follows that for $V\subset X\setminus\{0\}$ a connected open subset, 
$\sect(V;H^0\solt(\shm))\neq 0$ if and only if $V\subset U$ and
$\exp(1/z)\vert_V$ is tempered.

Denote by ${\ol B}_\vep$ the closed ball with center $(\vep,0)$
and radius $\vep$ and set
\eqn
&&U_\vep=X\setminus {\ol B}_\vep=\{z\in\C\setminus\{0\};\Re(1/z)<1/2\vep\}.
\eneqn
One proves  that
$\exp(1/z)$ is tempered (in a neighborhood of $0$)
 on an open subanalytic subset $V\subset X\setminus\{0\}$ if and only
if ${\rm Re}(1/z)$ is bounded on $V$, that is, if and only if
$V\subset U_\vep$ for some $\vep>0$. We get
 the isomorphism
\eq\label{eq:solt=indlim}
&&
\solt(\shl)\tens\C_{U}\simeq\inddlim[\vep>0]\C_{U_\vep}.
\eneq
Note that $\solt(\shl)\tens\C_{U}$ is concentrated in degree $0$. 

Since $\solt(\shl)\simeq\drt(\Ddual\shl)$ and $\Ddual\shl\simeq\Ddual\shl(*\{0\})$, we get that 
\eqn
&&\solt(\shl)\simeq\rihom(\C_U,\solt(\shl))
\simeq\rihom(\C_U,\solt(\shl)\tens\C_U).
\eneqn
Therefore,
\eqn
&&\solt(\shl)\simeq\rihom(\C_{U},\inddlim[\vep>0]\C_{U_\vep}),\\
&&H^0(\solt(\shl))\simeq\inddlim[\vep>0]\C_{U_\vep},\\
&& H^1(\solt(\shl))\simeq\inddlim[\vep>0]\ext{1}(\C_U,\C_{U_\vep})\simeq
\C_{\{0\}},\\
&&\sol(\shl)\simeq\alpha_X\solt(\shl)\simeq\rhom(\C_U,\C_U),\\
&&H^0(\sol(\shl))\simeq\C_U,\quad H^1(\sol(\shl))\simeq\C_{\{0\}}.
\eneqn

The functor $\solt$  is not fully faithful since 
the $\shd_X$-modules $\shd_X\exp(1/z)$ and $\shd_X\exp(2/z)$ 
have the same indsheaves of tempered holomorphic solutions 
although they are not isomorphic.

However, $\solt_X(\D_X\exp(1/z))\not\simeq\solt_X(\D_X\exp(1/z^m))$ 
for any $m>1$. 

Hence, the functor $\solt$ is sensitive enough to distinguish 
$m\in\Z_{>0}$ in 
$\D_X\exp(z^{-m})$ but it is not 
sensitive enough to distinguish $c\in\R_{>0}$ in $\D_X\exp(cz^{-1})$.

In order to capture $c$, we need to work 
in the framework of enhanced indsheaves,
which we are going to explain in the next sections.

\section{Indsheaves on bordered spaces}\label{section:bordered}

\subsection{Bordered spaces}
\index{bordered spaces!}%
\begin{definition}\label{def:bspace}
\glossary{$(M,\bM)$}%
The category of \emph{bordered spaces} is the category whose objects are pairs $(M,\bM)$ with $M\subset \bM$ an open embedding of good topological spaces.
Morphisms $f\colon (M,\bM) \to (N,\bN)$ are continuous maps $f\colon M\to N$ such that
\begin{equation}
\label{eq:Hbord}
\overline\Gamma_f \to \bM \text{ is proper}. 
\end{equation}
Here $\Gamma_f\subset M\times N$ is the graph of $f$ an $\ol\Gamma_f$ is its closure in 
$\bM\times\bN$.

The composition of $(L,\bL) \to[g] (M,\bM) \to[f] (N,\bN)$ is 
given by $f\circ g\colon L\to N$ (see Lemma~\ref{lem:bordcom} below), and the identity $\id_{(M,\bM)}$ is given by $\id_{M}$.
\end{definition}

\begin{lemma}\label{lem:bordcom}
Let $f\colon (M,\bM) \to (N,\bN)$ and $g\colon (L,\bL) \to (M,\bM)$ be morphisms of bordered spaces. 
Then the composition $f\circ g$ is a morphism of bordered spaces.
\end{lemma}

One shall identify a space $M$ and the bordered space $(M,M)$.
Then, by using the identifications $M=(M,M)$ and $\bM = (\bM,\bM)$, there are natural morphisms of bordered spaces 
\[
M \to (M,\bM) \to \bM.
\]
Note however that $(M,\bM) \to M$ is 
a morphism of bordered spaces if and only if $M$ is a closed subset of $\bM$.

We can easily see that the category of bordered spaces admits products:
\eq(M,\bM)\times (N,\bN)\simeq (M\times N,\bM\times\bN).\eneq

\medskip
Let $(M,\bM)$ be a bordered space. Denote by 
$i\cl \bM\setminus M\to \bM$ the closed embedding. By identifying 
$\Derb(\cor_{\bM\setminus M})$ with 
its essential image in $\Derb(\cor_{\bM})$ by the fully faithful functor $\reim i \simeq \roim i$, 
the restriction functor $F\mapsto F|_M$ induces an equivalence
\eqn
&&\Derb(\field_{\bM}) / \Derb(\field_{\bM\setminus M})\isoto\Derb(\field_M).
\eneqn
This is no longer true for indsheaves. Therefore one sets
\glossary{$\Derb(\icor_{(M,\bM)})$}%
\eqn
&&\Derb(\icor_{(M,\bM)})\eqdot \Derb(\icor_{\bM})/\Derb(\icor_{\bM\setminus M})
\eneqn
where $\Derb(\icor_{\bM\setminus M})$ is identified with its essential image in $\Derb(\icor_{\bM})$  by $\reeim{i}\simeq\roim{i}$, as for usual sheaves. 

Recall that if $\sht$ is a triangulated category and $\shi$ a subcategory, one denotes by ${}^\bot\shi$ and $\shi^\bot$ the left and right orthogonal to $\shi$ in $\sht$, respectively:
\eqn
&&{}^\bot\shi\seteq\set{A\in\sht}{\text{$\Hom[\sht](A,B)=0$ for any $B\in\shi$}},\\
&&\shi^\bot\seteq\set{A\in\sht}{\text{$\Hom[\sht](B,A)=0$ for any $B\in\shi$}}.
\eneqn
\begin{proposition}\label{pro:bord}
Let $(M,\bM)$ be a bordered space. Then we have 
\eqn
\Derb(\icor_{\bM\setminus M}) 
&=& \{F\in\Derb(\icor_{\bM}); \cor_M \tens F \simeq 0 \}\\
&=& \{F\in\Derb(\icor_{\bM}); \rihom(\cor_M, F) \simeq 0 \},\\
{}^\bot \Derb(\icor_{\bM\setminus M}) 
&=& \{F\in\Derb(\icor_{\bM});  \cor_M \tens F\isoto F \},\\
\Derb(\icor_{\bM\setminus M})^\bot &=& \{F\in\Derb(\icor_{\bM}); F \isoto \rihom(\cor_M, F) \}.
\eneqn
Moreover, there are  equivalences
\eqn
\Derb(\icor_{(M,\bM)}) &\isoto& \Derb(\icor_{\bM\setminus M})^\bot, \quad F \mapsto \rihom(\cor_M, F),\\
\Derb(\icor_{(M,\bM)}) &\isoto& {}^\bot \Derb(\icor_{\bM\setminus M}), \quad F\mapsto \cor_M \tens F,
\eneqn
with quasi-inverse induced by the quotient functor.
\end{proposition}

\begin{corollary}\label{cor:bord}
For $F,G\in\Derb(\icor_{\bM})$ one has
\begin{align*}
\Hom[\Derb(\icor_{(M,\bM)})](F,G)
&\simeq \Hom[\Derb(\icor_{\bM})](\cor_M\tens F,G) \\
&\simeq \Hom[\Derb(\icor_{\bM})](F,\rihom(\cor_M,G)).
\end{align*}
\end{corollary}
\glossary{$\tens$}%
\glossary{$\rihom$}%
The functors $\tens$ and $\rihom$ in $\Derb(\icor_{\bM})$ induce well defined functors (we keep the same notations)
\eqn
\tens &\colon& \Derb(\icor_{(M,\bM)}) \times \Derb(\icor_{(M,\bM)}) \to \Derb(\icor_{(M,\bM)}), \\
\rihom &\colon& \Derb(\icor_{(M,\bM)})^\op \times \Derb(\icor_{(M,\bM)}) \to \Derb(\icor_{(M,\bM)}).
\eneqn

\subsection{Operations}
Let $f\colon (M,\bM) \to (N,\bN)$ be a morphism of bordered spaces, and recall that $\Gamma_f$ denotes the graph of the associated map $f\colon M \to N$.
Since $\Gamma_f$ is closed in $M\times N$, it is locally closed in $\bM\times \bN$. One can then consider
the sheaf $\cor_{\Gamma_f}$ on $\bM\times \bN$.
Let $q_1\colon\bM\times\bN\to\bM$ and $q_2\colon\bM\times\bN\to\bN$ be
the projections.

\begin{definition}
\label{def:fbordered}
Let $f\colon (M,\bM) \to (N,\bN)$ be a morphism of bordered spaces.
For $F\in\Derb(\icor_{\bM})$ and $G\in\Derb(\icor_{\bN})$, we set
\glossary{$\reeim{f}$}%
\glossary{$\roim{f}$}%
\glossary{$\opb{f}$}%
\glossary{$\epb{f}$}%
\eqn
\reeim{f} F &=& \reeim{q_2}(\cor_{\Gamma_f}\tens\opb{q_1}F), \\
\roim{f} F &=& \roim{q_2}\rihom(\cor_{\Gamma_f},\epb{q_1}F), \\
\opb f G &=& \reeim{q_1}(\cor_{\Gamma_f}\tens\opb{q_2}G), \\
\epb f G &=& \roim{q_1}\rihom(\cor_{\Gamma_f},\epb{q_2}G).
\eneqn
\end{definition}

\begin{remark}
Considering a continuous map $f\colon M\to N$ as a morphism of bordered spaces with $M=\bM$ and $N=\bN$, the above functors are isomorphic to the usual external operations for indsheaves.
\end{remark}

\begin{lemma}
\label{lem:coper}
The above definition induces well-defined functors
\begin{align*}
\reeim{f},\roim{f}  &\colon \Derb(\icor_{(M,\bM)}) \to \Derb(\icor_{(N,\bN)}), \\
\opb{f},\epb{f}  &\colon \Derb(\icor_{(N,\bN)}) \to \Derb(\icor_{(M,\bM)}).
\end{align*}
\end{lemma}

\begin{lemma}\label{lem:jM}
Let $j_M\colon(M,\bM)\to \bM$ be the morphism given by the open embedding $M\subset \bM$. Then
\bnum
\item
The functors
$\opb{j_M} \simeq \epb{j_M} \colon \BDC(\icor_{\bM}) \to \BDC(\icor_{(M,\bM)})$
are isomorphic to the quotient functor.
\item
For $F\in\BDC(\icor_{\bM})$ one has the isomorphisms in $\BDC(\icor_{\bM})$
\[
\reeim {j_M}\opb{j_M} F \simeq \cor_M \tens F, \quad
\roim {j_M}\epb{j_M} F \simeq \rihom(\cor_M, F).
\]
\item
The functors $\tens$ and $\rihom$ commute with $\opb{j_M} \simeq \epb{j_M}$.
\item 
The functor $\tens$ commutes with $\reeim{j_M}$ and
the functor $\rihom$ commutes with $\roim{j_M}$.
\enum
\end{lemma}
The operations for indsheaves on bordered spaces satisfy similar properties as for usual spaces.

\begin{lemma}\label{lem:badj}
Let $f\colon (M,\bM) \to (N,\bN)$ and $g\colon (L,\bL) \to (M,\bM)$  be  morphisms of bordered spaces.
\bnum
\item
The functor $\reeim{f}$ is left adjoint to $\epb{f}$.
\item
The functor $\opb{f}$ is left adjoint to $\roim{f}$.
\item
 One has $\reeim{(f\circ g)} \simeq \reeim{f} \circ \reeim{g}$, 
$\roim{(f\circ g)} \simeq \roim{f} \circ \roim{g}$,
$\opb{(f\circ g)} \simeq \opb{g} \circ \opb{f}$ and 
$\epb{(f\circ g)} \simeq \epb{g} \circ \epb{f}$.
\enum
\end{lemma}

\begin{corollary}
If $f\colon (M,\bM) \to (N,\bN)$ is an isomorphism of bordered spaces, then $\roim f \simeq \reeim f$ and $\opb f \simeq \epb f$. Moreover, $\roim f$ and $\opb f$ are quasi-inverse to each other.
\end{corollary}

Most of the formulas for indsheaves on usual spaces extend to bordered spaces.

\begin{proposition}
\label{pro:bproj}
Let $f\colon (M,\bM) \to (N,\bN)$ be a morphism of bordered spaces.
For $F\in\Derb(\icor_{(M,\bM)})$ and $G,G_1,G_2\in\Derb(\icor_{(N,\bN)})$, one has isomorphisms
\begin{align*}
\reeim {f}(\opb {f} G \tens F) & \simeq G \tens \reeim {f} F, \\
\opb {f} (G_1\tens G_2) &\simeq \opb {f} G_1 \tens \opb {f} G_2, \\
\rihom(G,\roim {f} F) & \simeq \roim {f} \rihom(\opb {f} G,F), \\
\rihom(\reeim {f} F, G) & \simeq \roim {f} \rihom(F, \epb {f} G), \\
\epb {f} \rihom(G_1,G_2) & \simeq \rihom(\opb {f} G_1, \epb {f} G_2),
\end{align*}
and a morphism
\[
\opb f \rihom(G_1,G_2) \to \rihom(\opb f G_1,\opb f G_2).
\]
\end{proposition}

\begin{lemma}
\label{lem:bcart}
Consider a Cartesian diagram in the category of bordered spaces
\begin{equation*}
\xymatrix{
(M',\bM') \ar[r]^{f'} \ar[d]^{g'} & (N',\bN') \ar[d]^{g} \\
(M,\bM) \ar[r]^{f}\ar@{}[ur]|-\square & (N,\bN).
}
\end{equation*}
Then there are isomorphisms of functors $\Derb(\icor_{(M',\bM')}) \to \Derb(\icor_{(N,\bN)})$
\[
\opb g\reeim f \simeq \reeim f' g^{\prime-1}, \qquad
\epb g\roim f \simeq \roim f' g^{\prime!}.
\]
\end{lemma}

The notion of proper morphisms of topological spaces
is extended to the case of bordered spaces as follows. 
\begin{definition}\label{def:proper}
The morphism of bordered spaces $f\colon (M,\bM) \to (N,\bN)$ is \emph{proper} if the following two conditions hold:
\index{bordered spaces!proper morphism of}%
\index{proper morphism!of bordered spaces}%
\bna
\item $f\colon M \to N$ is proper,
\item the projection $\overline\Gamma_f \to \bN$ is proper.
\ee
\end{definition}

\begin{lemma}
\label{lem:proper}
The map $f\colon (M,\bM) \to (N,\bN)$ is proper if and only  if the following two conditions hold: 
\bna
\item  $\overline\Gamma_f \times_{\bN} N \subset \Gamma_f$.
\item the projection $\overline\Gamma_f \to \bN$ is proper.
\ee
\end{lemma}

\begin{proposition}
Assume that $f\colon (M,\bM) \to (N,\bN)$ is proper. Then $\reeim f \isoto\roim f$ as functors $\Derb(\icor_{(M,\bM)}) \to \Derb(\icor_{(N,\bN)})$.
\end{proposition}

\section{Enhanced indsheaves}\label{section:enhanced}
In this section, extracted from~\cite{DK13}, one extends some constructions of Tamarkin~\cite{Ta08} to indsheaves on bordered spaces. 
We refer to~\cite{GS12} for a detailed exposition and some complements to Tamarkin's paper. 

\subsection{Tamarkin's construction}
Let $M$ be a smooth manifold and denote by $T^*M$ its cotangent bundle.
Given $F\in\Derb(\cor_M)$, its microsupport $SS(F) \subset T^*M$ (see~\cite{KS90}) describes the codirections of non propagation for the cohomology of $F$. It is a closed conic co-isotropic subset of $T^*M$.

In order to treat co-isotropic subsets of $T^*M$ which are not necessarily conic, Tamarkin adds a real variable $t\in\R$.
Denoting by $(t,t^*)$ the symplectic coordinates of $T^*\R$, consider the full subcategory 
$\Derb_{t^*\leq 0} (\cor_{M\times\R})\subset \Derb(\cor_{M\times\R})$ whose objects $K$ satisfy $SS(K) \subset \{t^*\leq 0\}$. There are equivalences
\[
{}^\bot\Derb_{t^*\leq 0} (\cor_{M\times\R})\ \simeq \Derb(\cor_{M\times\R})/\Derb_{t^*\leq 0} (\cor_{M\times\R}) 
\simeq\Derb_{t^*\leq 0} (\cor_{M\times\R})^\bot
\]
between the quotient category and the left and right orthogonal categories.

Let us recall the description of the first equivalence.

For $K,L\in\Derb(\cor_{M\times\R})$, consider the convolution functor with respect to the $t$ variable
\[
K\ctens L\seteq \reim \mu (\opb q_1 K \tens \opb q_2 L),
\]
where $\mu,\;q_1\,;q_2\cl M\times\R\times\R$ are given by 
$\mu(x,t_1,t_2) = (x,t_1 + t_2)$, $q_1(x,t_1,t_2) = (x,t_1)$ and $q_2(x,t_1,t_2) = (x,t_2)$.

One sets
\begin{align}\label{eq:tgeq0A}
\cor_{\{t\geq 0\}} &= \cor_{\{(x,t)\in M\times\R\;;\; t\in\R,\ t \geq 0\}}, 
\end{align}
and we use similar notation for  
$\cor_{\{t= 0\}}$. These are sheaves on $M\times\R$.

Note that $\cor_{\{t=0\}} \ctens K \simeq K$. Then
\eqn
{}\Derb_{t^*\leq 0} (\cor_{M\times\R})&=& \{K\in\Derb(\cor_{M\times\R});\cor_{\{t\geq 0\}} \ctens K\simeq 0 \},\\
{}^\bot\Derb_{t^*\leq 0} (\cor_{M\times\R})&=& \{K\in\Derb(\cor_{M\times\R});\cor_{\{t\geq 0\}} \ctens K \isoto K \},
\eneqn
and one has an equivalence
\[
\Derb(\cor_{M\times\R})/\Derb_{t^*\leq 0} (\cor_{M\times\R}) \isoto {}^\bot\Derb_{t^*\leq 0} (\cor_{M\times\R}),
\quad K\mapsto \cor_{\{t\geq 0\}} \ctens K.
\]
We will adapt this construction to
the case of indsheaves and a good topological space $M$ in the sequel. 

\subsection{Convolution products} 

Consider the 2-point compactification of the real line 
\glossary{$\ol\R$}%
$\ol\R \eqdot \R
\sqcup\{+\infty,-\infty\}$. Denote by $\BBP^1(\R)=\R\sqcup\{\infty\}$ the real projective line. Then $\ol\R$ has a structure of subanalytic space such that the natural map $\ol\R\to\PR$ is a subanalytic map.

\begin{notation}\label{not:Rinfty}
We will consider the bordered space
\glossary{$\R_\infty$}%
\[
\R_\infty \eqdot (\R,\overline\R).
\]
\end{notation}

Note that $\R_\infty$ is isomorphic to $(\R,\PR)$ as a bordered space.

Consider the morphisms of bordered spaces
\begin{align}
a &\colon \R_\infty \to \R_\infty, \label{eq:muq1q2}\\
\mu, q_1,q_2 &\colon\R_\infty\times\R_\infty \to \R_\infty,\notag
\end{align}
\glossary{$\mu(t_1,t_2)$}%
where $a(t) = -t$, $\mu(t_1,t_2) = t_1+t_2$ 
and $q_1,q_2$ are the natural projections.

For a good topological space $M$, we will use the same notations for the associated morphisms
\begin{align*}
a &\colon M\times\R_\infty \to M\times\R_\infty, \\
\mu, q_1,q_2 &\colon M\times\R_\infty\times\R_\infty \to M\times\R_\infty.
\end{align*}
 We also use  the natural morphisms
\eq\ba{c}
\xymatrix@C=4ex{
M\times\R_\infty \ar[rr]^j \ar[dr]_\pi && M\times\overline\R \ar[dl]^{\overline\pi} \\
&M.
}\ea
\eneq
\begin{notation}\label{not:fR}
We sometimes write $\BDC(\cor_{M\times\fR})$ for
$\BDC(\cor_{M\times\R})$ regarded as a full subcategory $\BDC(\icor_{M\times\R_\infty})$.
\end{notation}

\begin{definition}\label{def:ctens1}
The functors
\glossary{$\ctens$}%
\glossary{$\cihom$}%
\begin{align*}
\ctens &\colon \BDC(\icor_{M\times\R_\infty}) \times \BDC(\icor_{M\times\R_\infty}) \to \BDC(\icor_{M\times\R_\infty}), \\
\cihom &\colon \Der[-](\icor_{M\times\R_\infty})^\op \times \Der[+](\icor_{M\times\R_\infty}) \to \Der[+](\icor_{M\times\R_\infty}),
\end{align*}
are defined by
\begin{align*}
K_1\ctens K_2 &= \reeim {\mu} (\opb q_1 K_1 \tens \opb q_2 K_2), \\
\cihom(K_1,K_2)&= \roim {q_1} \rihom(\opb q_2 K_1, \epb{\mu}K_2).
\end{align*}
\end{definition}
Although we work now on $M\times\ol\R$, we keep the same notations as in~\eqref{eq:tgeq0A} and one sets
\begin{align}\label{eq:tgeq0B}
\cor_{\{t\geq 0\}} &= \cor_{\{(x,t)\in M\times\overline\R\;;\; t\in\R,\ t \geq 0\}}.
\end{align}
We use similar notation for  
$\cor_{\{t= 0\}}$, $\cor_{\{t> 0\}}$, $\cor_{\{t\leq 0\}}$, $\cor_{\{t< 0\}}$,  $\cor_{\{t\not=0\}}$ and $\cor_{\{t=a\}}$, etc.  
These are sheaves on $M\times\overline\R$ whose stalk vanishes at points of $M\times(\overline\R\setminus\R)$. 
We also regard them as objects of $\BDC(\icor_{M\times\R_\infty})$. 
\begin{lemma}
For $K\in\BDC(\icor_{M\times\R_\infty})$ there are isomorphisms
\[
\cor_{\{t= 0\}} \ctens K \simeq K \simeq \cihom(\cor_{\{t= 0\}}, K).
\]
More generally, for $a\in\R$, we have
\[
\cor_{\{t= a\}} \ctens K \simeq \roim{\mu_a} K \simeq \cihom(\cor_{\{t= -a\}}, K),
\]
where $\mu_a\colon M\times\R_\infty\to M\times\R_\infty$ is the morphism induced by the translation $t\mapsto t+a$.
\end{lemma}

\begin{corollary}
The category $\BDC(\icor_{M\times\R_\infty})$ has a structure of commutative tensor category with $\ctens$ as tensor product and $\cor_{\{t= 0\}}$ as unit object.
\end{corollary}

As seen in \eqref{eq:innerhom} below, the functor $\cihom$ is the
inner hom of the tensor category $\BDC(\icor_{M\times\R_\infty})$.
\begin{lemma}
For $K_1,K_2,K_3\in\BDC(\icor_{M\times\R_\infty})$ one has
\eq
&&\ba[t]{l}
\Hom[\BDC(\icor_{M\times\R_\infty})](K_1\ctens K_2,\;K_3)\\[1ex]
\hs{7ex}
\simeq \Hom[\BDC(\icor_{M\times\R_\infty})]\bl K_1,\;\cihom(K_2,K_3)\br,
\ea\label{eq:innerhom}\\
&&\cihom(K_1\ctens K_2,\;K_3)
\simeq \cihom\bl K_1,\;\cihom(K_2,K_3)\br,\nn\\
&&\roim\pi\rihom(K_1\ctens K_2,K_3) \simeq
\roim\pi\rihom(K_1,\cihom(K_2,K_3)).\nn
\eneq
\end{lemma}

The following lemmas are used to define the category of enhanced indsheaves.

\begin{lemma}
\label{lem:cihomrihompi}
For $K_1,K_2\in\BDC(\icor_{M\times\R_\infty})$ and $L\in\BDC(\icor_M)$ one has
\begin{align*}
\opb\pi L \tens (K_1\ctens K_2) & \simeq (\opb\pi L \tens K_1)\ctens K_2, \\
\rihom(\opb\pi L, \cihom(K_1,K_2)) & \simeq \cihom(\opb\pi L \tens K_1,K_2) \\
& \simeq \cihom(K_1,\rihom(\opb\pi L, K_2)).
\end{align*}
\end{lemma}

\begin{lemma}
\label{lem:cihomKt0}
For $K\in\BDC(\icor_{M\times\R_\infty})$ and $L\in\BDC(\icor_M)$ one has
\begin{align*}
\opb\pi L \tens K & \simeq (\opb\pi L \tens \cor_{\{t=0\}})\ctens K, \\
\rihom(\opb\pi L, K) & \simeq \cihom(\opb\pi L \tens \cor_{\{t=0\}},K), \\
\opb a \rihom(K,\epb\pi L) &\simeq
\cihom(K,\cor_{\{t=0\}} \tens\opb\pi L).
\end{align*}
Here $a$ is the involution of $M\times\fR$ given by
$(x,t)\mapsto(x,-t)$.
\end{lemma}

\begin{lemma}
\label{lem:tenspipi}
For $K_1,K_2\in\BDC(\icor_{M\times\R_\infty})$ there are isomorphisms
\begin{align*}
\reeim\pi(K_1\ctens K_2) &\simeq \reeim\pi K_1 \tens \reeim\pi K_2, \\
\roim\pi\cihom(K_1,K_2) &\simeq \rihom(\reeim\pi K_1,\roim\pi K_2).
\end{align*}
\end{lemma}

\begin{corollary}
For any $K\in\BDC(\icor_{M\times\R_\infty})$, one has
\begin{align*}
&\reeim\pi(\cor_{t\geq 0} \ctens K) \simeq 0, \\
&\roim\pi\cihom(\cor_{t\geq 0}, K) \simeq 0.
\end{align*}
\end{corollary}

\begin{lemma}\label{lem:piRinfty} 
For $K\in\BDC(\icor_{M\times\R_\infty})$ and $L\in\BDC(\icor_M)$ one has
\begin{align*}
(\opb\pi L) \ctens K & \simeq \opb\pi(L\tens\reeim\pi K), \\
\cihom(\opb\pi L, K) & \simeq \epb\pi\rihom(L,\roim\pi K), \\
\cihom(K, \epb\pi L) & \simeq \epb\pi\rihom(\reeim\pi K, L).
\end{align*}
\end{lemma}

\begin{proposition}\label{prop:tenhom}
For $K\in\BDC(\icor_{M\times\R_\infty})$, one has a distinguished triangle 
\eqn
&&\opb{\pi}L\To  \cor_{\{t\geq 0\}} \ctens K
\To \cihom(\cor_{\{t\geq 0\}}, K)  \To[+1\;]
\eneqn
with $L=\roim{\pi}(\cor_{\{t\geq 0\}} \ctens K)\simeq\reeim{\pi}\cihom(\cor_{\{t\geq 0\}}, K)$. 
\end{proposition}

\subsection{Enhanced indsheaves}
\index{enhanced!indsheaves}%
\index{indsheaves!enhanced}%
\begin{definition}\label{def:DerT}
Consider the full  triangulated  subcategories of $\BDC(\icor_{M\times\R_\infty})$
\glossary{$\indc_{t^*\leq 0} $}%
\glossary{$\indc_{t^*\geq 0}$}%
\glossary{$\indc_{t^* = 0} $}%
\eqn
\indc_{t^*\leq 0}  &= &\{K ; \cor_{\{t\geq 0\}} \ctens K  \simeq 0\} \\
                               &=& \{K ; \cihom(\cor_{\{t\geq 0\}}, K)  \simeq 0\}, \\
\indc_{t^*\geq 0}  &=& \{K ; \cor_{\{t\leq 0\}} \ctens K  \simeq 0\} \\
                          &= &\{K ; \cihom(\cor_{\{t\leq 0\}}, K)  \simeq 0\}, \\
\indc_{t^* = 0} &=& \indc_{t^* \leq 0} \cap \indc_{t^* \geq 0}.
\eneqn
Consider also the corresponding quotient categories
\glossary{$\TDCpm(\icor_M)$}%
\glossary{$\TDC(\icor_M)$}%
\begin{align*}
\TDCpm(\icor_M) 
&= \indc_{\pm t^*\geq 0}/\indc_{t^*= 0} , \\
\TDC(\icor_M) 
&= \BDC(\icor_{M\times\R_\infty})/\indc_{t^* = 0} .
\end{align*}
They are triangulated categories. 
One calls $\TDC(\icor_M)$ the triangulated category of {\em enhanced indsheaves}.
One defines similarly the categories $\TDCC(\icor_M)$, $\TDCC^+(\icor_M)$ and $\TDCC^-(\icor_M)$.
\end{definition}

Notice that 
\eqn
 \indc_{t^* = 0}&=&\set{K\in \BDC(\icor_{M\times\R_\infty})}{
\opb{\pi}\roim{\pi}K\isoto K}\\
&=&\set{K\in \BDC(\icor_{M\times\R_\infty})}{K\isoto\epb{\pi}\reeim{\pi}K}\\
&=&\set{K\in \BDC(\icor_{M\times\R_\infty})}{
\text{$K\simeq\opb{\pi}L$ for some $L\in\Derb(\icor_M)$}}.
\eneqn 
Therefore,
\eq\label{eq:TDC}
&&\TDC(\icor_M)\simeq \BDC(\icor_{M\times\R_\infty})/\{K\in \BDC(\icor_{M\times\R_\infty});\opb{\pi}\roim{\pi}K\isoto K\}.
\eneq

\glossary{$\TDC(\cor_{M})$}%
One also defines the category $\TDC(\cor_{M}) $ as
\eq\label{eq:TDCforM}
&&\TDC(\cor_{M}) =\BDC(\cor_{{M}\times\R})/\set{K}{\opb{\pi}\roim{\pi}K\isoto K}.
\eneq
Then $\TDC(\cor_{M})$ is a full subcategory of $\TDC(\icor_{M})$.

\begin{proposition}
There are equivalences of triangulated categories
\begin{align*}
\TDCpm(\icor_M) &\simeq \BDC(\icor_{M\times\R_\infty})/\indc_{\pm t^*\leq 0}, \\
\TDC(\icor_M) &\simeq \TDCp(\icor_M) \dsum \TDCm(\icor_M).
\end{align*}
\end{proposition}

This follows from Proposition~\ref{pro:tam} below.

\begin{remark}
The categories  $\TDCp(\icor_M)$ are the analogue of Tamarkin's construction in the framework of indsheaves.
\end{remark}
\begin{proposition}\label{pro:tam}
One has
\eqn
{}^\bot\indc_{\pm t^*\leq 0} &=& \{K ; \cor_{\{\pm t\geq 0\}} \ctens K  \isoto K\}
= \{K ; \cor_{\{\pm t> 0\}} \ctens K \simeq 0\},\\
\indc_{\pm t^*\leq 0}^\bot 
&=&\{K ; K \isoto \cihom(\cor_{\{\pm t\geq 0\}}, K)\} \\
&=& \{K; \cihom(\cor_{\{\pm t> 0\}}, K) \simeq 0\},\\
{}^\bot\indc_{t^* = 0} 
&=& \{K ; (\cor_{\{t\geq 0\}} \dsum \cor_{\{t\leq 0\}}) \ctens K  \isoto K\}\\
&=& \{K ; \cor_{M\times\R}\ctens K \simeq 0 \} 
= \{K ; \reeim\pi K \simeq 0\}, \\[1ex]
\indc^\bot_{t^* = 0} 
&=& \{K ; \cihom(\cor_{\{t\geq0\}}\oplus\cor_{\{t\leq0\}}, K)  \isoto K\} \\
&=&\{K ; \cihom(\cor_{M\times\R}, K) \simeq 0 \} 
=\{K ; \roim\pi K \simeq 0\},\\
{}^\bot\indc_{t^* = 0}&=& {}^\bot\indc_{t^* \geq 0} \dsum {}^\bot\indc_{t^* \leq 0}, \\
\indc^\bot_{t^* = 0}&= &\indc^\bot_{t^* \geq 0} \dsum \indc^\bot_{t^* \leq 0}.
\eneqn
Moreover, one has the equivalences
\eqn
\TDCpm(\icor_M) &\isoto& {}^\bot\indc_{\pm t^*\leq 0}, \quad
K\mapsto \cor_{\{\pm t\geq 0\}} \ctens K,\\
\TDCpm(\icor_M) &\isoto& \indc_{\pm t^*\leq 0}^\bot, \quad
K\mapsto\cihom(\cor_{\{\pm t\geq 0\}}, K),\\
\TDC(\icor_M) &\isoto& {}^\bot\indc_{t^* = 0}, \quad
K\mapsto (\cor_{\{t\geq0\}}\oplus\cor_{\{t\leq0\}}) \ctens K,\\
\TDC(\icor_M) &\isoto& \indc_{t^* = 0}^\bot, \quad
K\mapsto \cihom(\cor_{\{t\geq0\}}\oplus\cor_{\{t\leq0\}}, K),
\eneqn
where the quasi-inverse functors are given by the quotient functors. 
\end{proposition}

These categories are illustrated as follows:
$$
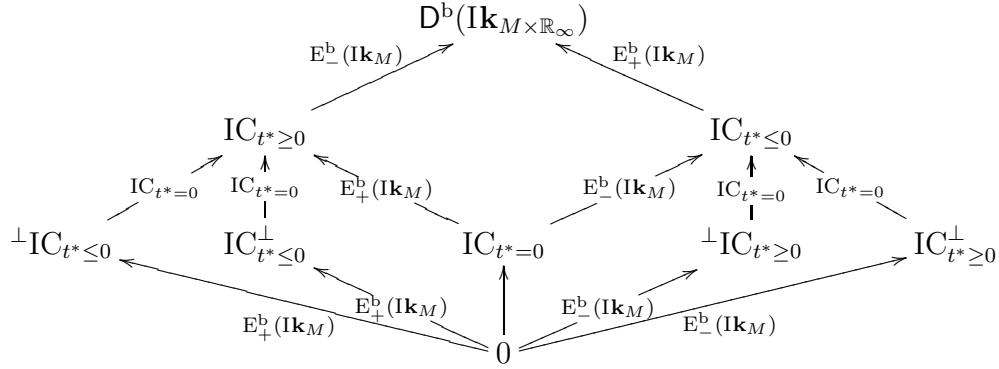
\xymatrix@C=7ex{
&& \BDC(\icor_{M\times\R_\infty})\\
&\indc_{t^*\ge 0}\ar[ru]^-{\TDCm(\icor_M)} &&\indc_{t^*\le 0}\ar[lu]_-{\;\TDCp(\icor_M)} \\
{}^\bot\indc_{t^*\leq 0}\ar[ru]|-{\rule[-.7ex]{0ex}{2.2ex}\indc_{t^*= 0}}
&\indc_{t^*\leq 0}^\bot\ar[u]|-{\rule[-.7ex]{0ex}{2.2ex}\indc_{t^*= 0}}
&\indc_{t^*= 0}\ar[lu]|-{\TDCp(\icor_M)}\ar[ru]|-{\;\TDCm(\icor_M)} 
&{}^\bot\indc_{ t^*\ge 0}\ar[u]|-(.45){\rule[-.7ex]{0ex}{2.2ex}\indc_{t^*= 0}}
&\indc_{ t^*\geq 0}^\bot\ar[lu]|-{\rule[-.7ex]{0ex}{2.2ex}\indc_{t^*= 0}}\\
&&0\ar[u]\ar[lu]|-{\TDCp(\icor_M)}\ar[ru]|-{\TDCm(\icor_M)}
\ar[llu]^-{\;\TDCp(\icor_M)\;}\ar[rru]_-{\;\TDCm(\icor_M)}
}
$$
Here, $\xymatrix{A\ar[r]|-C&B}$ or $\xymatrix{A\ar[r]_-C&B}$ means that $C\simeq B/A$.
\begin{definition}\label{def:Tlr}
One introduces the functors
\glossary{$\Tl$}%
\glossary{$\Tr$}%
\eqn
\Tl &=& (\cor_{\{t\geq0\}}\oplus\cor_{\{t\leq0\}})\ctens(\scbul), \quad \TDC(\icor_M) \to {}^\bot\indc_{t^*= 0} \subset \BDC(\icor_{M\times\R_\infty}), \\ 
\Tr &=&\cihom(\cor_{\{t\geq0\}}\oplus\cor_{\{t\leq0\}}, \scbul),\quad \TDC(\icor_M) \to \indc_{t^*= 0}^\bot \subset \BDC(\icor_{M\times\R_\infty}).
\eneqn
\end{definition}
The functors $\Tl$ and $\Tr$ are the left and right adjoint of the quotient functor 
$\BDC(\icor_{M\times\R_\infty}) \to \TDC(\icor_M)$, and the two compositions 
\eqn
\xymatrix@C=7ex{
\TDC(\icor_M)\ar@<-0.5ex>[r]_-{\Tr}\ar@<0.5ex>[r]^-{\Tl}&\BDC(\icor_{M\times\R_\infty})\ar[r]&\TDC(\icor_M)
}\eneqn
are isomorphic to the identity.

\begin{definition}\label{def:fihom}
One defines the hom functor 
\glossary{$\fihom$}%
\glossary{$\fhom$}%
\glossary{$\FHom$}%
\begin{align}
\fihom\colon \TDC(\icor_M)^\op \times \TDC(\icor_M) &\to \Derp(\icor_M)\\
\fihom(K_1,K_2) &= \roim\pi\rihom(\Tl(K_1),\Tr(K_2)),\nn
\end{align}
and one sets
\eq
&&\fhom= \alpha_M\circ\fihom\;\cl\; \TDC(\icor_M)^\op \times \TDC(\icor_M) \to \Derp(\cor_M),\\
&&\FHom(K_1,K_2) = \rsect(M;\fhom(K_1,K_2))\in\BDC(\cor).
\eneq
\end{definition}
Note that 
\begin{align*}
\fihom(K_1,K_2) &\simeq \roim\pi\rihom(\Tl(K_1),\Tl(K_2)) \\
&\simeq \roim\pi\rihom(\Tr(K_1),\Tr(K_2))
\end{align*} 
and
\begin{align*}
\Hom[\TDC(\icor_{M})](K_1,K_2)
&\simeq H^0 \bl\FHom(K_1,K_2)\br.
\end{align*}

\subsection{Operations on enhanced indsheaves}

\index{Grothendieck's six operations!for enhanced indsheaves}%
By Lemma~\ref{lem:piRinfty} the following definition is well posed.

\begin{definition}\label{def:Tctens}
The bifunctors 
\glossary{$\ctens$}%
\glossary{$\cihom$}%
\begin{align*}
\ctens &\colon \TDC(\icor_M) \times \TDC(\icor_M) \to \TDC(\icor_M), \\
\cihom &\colon \TDCC^-(\icor_M)^\op \times \TDCC^+(\icor_M) \to \TDCC^+(\icor_M)
\end{align*}
are those induced by the bifunctors $\ctens$ and $\cihom$ defined on $\BDC(\icor_{M\times\R_\infty})$.
\end{definition}
For any $K\in\TDC(\icor_M)$ there is an isomorphism in $\TDC(\icor_M)$
\[
\cor_{\{t\geq 0\}} \ctens K \isoto \cihom(\cor_{\{t\geq 0\}}, K),
\]
which follows from Proposition~\ref{prop:tenhom}.

The bifunctor $\ctens$ gives   $\TDC(\icor_M)$ a structure of a commutative tensor category  
with $\cor_{\{t= 0\}}$ as a unit object. 
Note that
\eqn
&&\Tl(\cor_{\{t= 0\}})\simeq\cor_{t\ge0}\soplus\cor_{t\le0},\\
&&\Tr(\cor_{\{t= 0\}})\simeq\cor_{t<0}[1]\soplus\cor_{t>0}[1].
\eneqn
Moreover,
$\cihom$ is the inner hom of the tensor category $\TDC(\icor_M)$:
\begin{lemma}
For $K_1,K_2,K_3\in\TDC(\icor_M)$ there is an isomorphism
\[
\Hom[\TDCC^+(\icor_M)](K_1\ctens K_2,K_3) \simeq \Hom[\TDCC^+(\icor_M)](K_1,\cihom(K_2,K_3)).
\]
\end{lemma}

We have the following orthogonal relations:
\eqn
&&\TDCp(\icor_M)\ctens \TDCm(\icor_M)\simeq0,\\
&& \cihom( \TDCpm(\icor_M), \TDCmp(\icor_M))\simeq0.
\eneqn

\begin{definition}
By Lemma~\ref{lem:piRinfty} one gets functors
\begin{align*}
\opb\pi(\scbul)\tens(\scbul)&\colon \BDC(\icor_M) \times \TDC(\icor_M) \to \TDC(\icor_M), \\
\rihom(\opb\pi(\scbul),\scbul) &\colon \Derm(\icor_M)^\op \times \TDCC^+(\icor_M) \to \TDCC^+(\icor_M).
\end{align*}
\end{definition}

\begin{remark}
The functor $\tens$
\emph{does not} factor through $\TDC(\icor_M)\times\TDC(\icor_M)$,
and the functor $\rihom$\emph{does not} factor through $\TDC(\icor_M)^\op\times\TDC(\icor_M)$.
\end{remark}

Let $f\colon M \to N$ be a continuous map of good topological spaces.
Denote by $\tilde f\colon M\times\R_\infty \to N\times\R_\infty$ 
the associated morphism of bordered spaces.
Then the composition of functors
\glossary{$\reeim{\tilde f}$}%
\glossary{$\roim{\tilde f}$}%
\glossary{$\opb{\tilde f}$}%
\glossary{$\epb{\tilde f}$}%
\begin{align}
\label{eq:oimftilde}
\reeim{\tilde f}, \roim{\tilde f} &\colon \BDC(\icor_{M\times\R_\infty}) \to \BDC(\icor_{N\times\R_\infty}) \to \TDC(\icor_N), \\
\label{eq:opbftilde}
\opb{\tilde f}, \epb{\tilde f} &\colon \BDC(\icor_{N\times\R_\infty}) \to \BDC(\icor_{M\times\R_\infty}) \to \TDC(\icor_M),
\end{align}
factor through $\TDC(\icor_M)$ and $\TDC(\icor_N)$, respectively.

\begin{definition}
\label{def:fT}
One denotes by
\glossary{$\Teeim{f}$}%
\glossary{$\Toim{f}$}%
\glossary{$\Topb{f}$}%
\glossary{$\Tepb{f}$}%
\begin{align*}
\Teeim f, \Toim f &\colon \TDC(\icor_M) \to \TDC(\icor_N), \\
\Topb f, \Tepb f &\colon \TDC(\icor_N) \to \TDC(\icor_M),
\end{align*}
the functors induced by \eqref{eq:oimftilde} and \eqref{eq:opbftilde}, respectively.
\end{definition}

\begin{definition}
For $K\in\TDC(\icor_M)$ and $L\in\TDC(\icor_N)$, 
we define their 
\index{external tensor product!for enhanced indsheaves}%
external tensor product by
\glossary{$\cetens$}%
\[
K \cetens L = \Topb p_1K \ctens \Topb p_2L \in \TDC(\icor_{M\times N}),
\]
where $p_1$ and $p_2$ denote the projections from $M\times N$ to $M$ and $N$, respectively.
\end{definition}

Using Definition~\ref{def:Tlr}, for $F\in\TDC(\icor_M)$ and $G\in\TDC(\icor_N)$ one has
\begin{align*}
\Teeim f F &\simeq \reeim{\tilde f}\Tl(F) \simeq \reeim{\tilde f}\Tr(F), \\
\Toim f F &\simeq \roim{\tilde f}\Tl(F) \simeq \roim{\tilde f}\Tr(F), \\
\Topb f G &\simeq \opb{\tilde f}\Tl(G) \simeq \opb{\tilde f}\Tr(G), \\
\Tepb f G &\simeq \epb{\tilde f}\Tl(G) \simeq \epb{\tilde f}\Tr(G).
\end{align*}

The above operations satisfy analogous properties as 
the  six  operations for sheaves and indsheaves.
\begin{proposition}
\label{pro:Tadj}
Let $f\colon M \to N$ be a continuous map of good topological spaces.
\bnum
\item
The functor $\Teeim{f}$ is left adjoint to $\Tepb{f}$.
\item
The functor $\Topb{f}$ is left adjoint to $\Toim{f}$.
\ee
\end{proposition}

\begin{proposition}\label{pro:Tcomp}
Given two continuous maps of good topological spaces $L \to[g] M \to[f] N$, one has
\[
\Teeim{(f\circ g)} \simeq \Teeim{f} \circ \Teeim{g}, \qquad
\Toim{(f\circ g)} \simeq \Toim{f} \circ \Toim{g}
\]
and
\[
\Topb{(f\circ g)} \simeq \Topb{g} \circ \Topb{f}, \qquad
\Tepb{(f\circ g)} \simeq \Tepb{g} \circ \Tepb{f}.
\]
\end{proposition}

\begin{proposition}
\label{pro:Tproj}
Let $f\colon M \to N$ be a continuous map of good topological spaces.
For $K\in\TDC(\icor_M)$ and $L,L_1,L_2\in\TDC(\icor_N)$, one has isomorphisms
\begin{align*}
\Teeim {f}(\Topb {f} L \ctens K) & \simeq L \ctens \Teeim {f} K, \\
\Topb {f} (L_1\ctens L_2) &\simeq \Topb {f} L_1 \ctens \Topb {f} L_2, \\
\cihom(L,\Toim {f} K) & \simeq \Toim {f} \cihom(\Topb {f} L,K), \\
\cihom(\Teeim {f} K, L) & \simeq \Toim {f} \cihom(K, \Tepb {f} L), \\
\Tepb {f} \cihom(L_1,L_2) & \simeq \cihom(\Topb {f} L_1, \Tepb {f} L_2),
\end{align*}
and a morphism
\[
\Topb f \cihom(L_1,L_2) \to \cihom(\Topb f L_1,\Topb f L_2).
\]
\end{proposition}

\begin{proposition}
\label{pro:Tcart}
Consider a Cartesian diagram of good topological spaces 
\begin{equation*}
\xymatrix{
M' \ar[r]^{f'} \ar[d]^{g'} & N' \ar[d]^{g} \\
M \ar[r]^{f}\ar@{}[ur]|-\square & N.
}
\end{equation*}
Then there are isomorphisms in the category of functors from
$\TDC(\icor_{M})$ to $\TDC(\icor_{N'})$:
\[
\Topb g\Teeim{f} \simeq \Teeim f' \Topb{g'}, \qquad
\Tepb g\Toim{f} \simeq \Toim f' \Tepb{g'}.
\]
\end{proposition}

\begin{lemma}
\label{homepb}
For $f\colon M\to N$ a morphism of good topological spaces, $K\in\TDC(\icor_M)$ and $L\in\TDC(\icor_N)$, one has
\begin{align*}
\roim f \fhom(K,\Tepb f L) &\simeq \fhom(\Teeim f K, L), \\
\roim f \fhom(\Topb f L, K) &\simeq \fhom(L, \Toim f K).
\end{align*}
\end{lemma}

\begin{remark}
Let $f\colon M\to N$ be a morphism of good topological spaces and $L_1,L_2\in\TDC(\icor_N)$.
Since $\alpha$ and $\epb f$ do not commute in general, the isomorphism $\epb f\fhom(L_1,L_2) \simeq \fhom(\Topb f L_1,\Tepb f L_2)$ does not hold in general.
\end{remark}

\subsection{Stable objects}
\index{stable object}%
The notion of stable object which will be introduced below  is related to the notion of torsion object from \cite{Ta08} (see also \cite[\S5]{GS12}).

\begin{notation}\label{not:gg}
Consider the indsheaves on $M\times\bR$
\glossary{$\cor_{\{t\gg0\}}$}%
\glossary{$\cor_{\{t<\ast\}} $}%
\[
\cor_{\{t\gg0\}} \eqdot \inddlim[a\rightarrow+\infty] \cor_{\{t\geq a\}},\qquad
\cor_{\{t<\ast\}}  \eqdot \inddlim[a\rightarrow+\infty] \cor_{\{t< a\}}.
\]
We regard them as objects of $\BDC(\icor_{M\times\R_\infty})$.
\end{notation}
There is a distinguished triangle in $\Derb(\icor_{M\times\R_\infty})$
\eqn
&&\cor_{M\times\R}\to \cor_{\{t\gg0\}} \to \cor_{\{t<\ast\}} \,[1]\to[+1]
\eneqn
and there are isomorphisms in $\BDC(\icor_{M\times\R_\infty})$
\eqn
&&\cor_{\{t\geq -a\}} \ctens \cor_{\{t\gg0\}} \isoto \cor_{\{t\gg0\}} \isoto \cor_{\{t\geq a\}} \ctens \cor_{\{t\gg0\}} \qquad (a\in\R_{\geq 0}). 
\eneqn

\begin{notation}\label{not:Tam}
Denote by $\cor^\Tam_M$ 
\glossary{$\cor^\Tam_M$}%
the object of $\TDC(\icor_M)$ associated with $\cor_{\{t\gg0\}}\in\BDC(\icor_{M\times\R_\infty})$.
More generally, for $F\in\BDC(\cor_M)$, set
\glossary{$F^\Tam$}%
\[
F^\Tam \eqdot \cor^\Tam_M \tens \opb\pi F \in \TDC(\icor_M).
\]
\end{notation}

Note that 
\[
\Tl(\cor^\Tam_M) \simeq \cor_{\{t\gg0\}},
\quad
\Tr(\cor^\Tam_M) \simeq \cor_{\{t<\ast\}}[1].
\]

\begin{proposition}\label{pro:equivTam}
Let $K\in\TDCp(\icor_M)$  \lp equivalently, $K\in\TDC(\icor_M)$ and $K\simeq\field_{\{t\geq 0\}} \ctens K$\rp.
Then the following conditions are equivalent.
\bna
\item 
$\field_{\{t\geq 0\}} \ctens K \isoto \field_{\{t\geq a\}} \ctens K$ for any $a\geq 0$,
\item 
$\cihom(\field_{\{t\geq a\}}, K) \isoto \cihom(\field_{\{t\geq 0\}}, K)$ for any $a\geq 0$,
\item
$\field_{\{t\geq 0\}} \ctens K \isoto \field_M^\Tam \ctens K$,
\item 
$\cihom(\field_M^\Tam, K) \isoto \cihom(\field_{\{t\geq 0\}}, K)$,
\item 
$K \simeq \field_M^\Tam \ctens L$ for some $L\in\TDC(\icor_M)$,
\item
$K \simeq \cihom(\field_M^\Tam, L)$ for some $L\in\TDC(\icor_M)$.
\ee
\end{proposition}

\begin{definition}
A {\em stable} object is an object of $\TDCp(\icor_M)$ that satisfies the equivalent conditions in Proposition~\ref{pro:equivTam}.
\end{definition}

\begin{lemma}\label{lem:kTamtens}
For $F\in\BDC(\cor_{M\times\R_\infty})$ and $K\in\TDC(\icor_M)$, there is an isomorphism in $\TDC(\icor_M)$
\[
\cor_M^\Tam \ctens \cihom(F,K)
\isoto
\cihom(F,\cor_M^\Tam \ctens K).
\]
\end{lemma}

\begin{corollary}
\label{cor:picorT}
For $K\in\TDC(\icor_M)$ and $F\in\BDC(\cor_M)$, we have
\[
\cor_M^\Tam \ctens \rihom(\opb\pi F, K) \simeq 
\rihom(\opb\pi F, \cor_M^\Tam \ctens K).
\]
\end{corollary}

\begin{proposition}\label{pro:stableops}
Let $f\colon M\to N$ be a continuous map of good topological spaces.
Then the functors $\Teeim f$, $\Topb f$ and $\Tepb f$ send stable objects to stable objects.
More precisely, we have{\rm:} 
\bnum
\item
For $K\in\TDC(\icor_M)$ one has
\[
\Teeim f(\cor_M^\Tam \ctens K) \simeq \cor_N^\Tam \ctens \Teeim f K.
\]
\item
For $L\in\TDC(\icor_N)$ one has
\begin{align*}
\Topb f(\cor_N^\Tam \ctens L) &\simeq \cor_M^\Tam \ctens \Topb f L, \\
\Tepb f(\cor_N^\Tam \ctens L) &\simeq \cor_M^\Tam \ctens \Tepb f L.
\end{align*}
\ee
\end{proposition}
\begin{definition}\label{def:fcteepsilon}
One defines the functors 
\eq
&&e_M,\epsilon_M\cl \BDC(\icor_M) \to \TDC(\icor_M),\label{eq:eMepsilonM0}\\
\glossary{$e_M$}%
\glossary{$\epsilon_M$}%
&&\hs{10ex}e_M(F) = \cor_M^\Tam\tens\opb\pi F, 
\quad\epsilon_M(F) = \cor_{\{t\geq0\}}\tens\opb\pi F.\nn
\eneq
\end{definition}
Note that
\eqn
&&e_M(F)\simeq\cor_M^\Tam\ctens\epsilon_M(F) .
\eneqn
\begin{proposition}\label{pro:embed}
The functors $e_M$ and $\epsilon_M$  are  fully faithful.
\end{proposition}

\begin{definition}
We define the duality functor
\index{duality!for enhanced indsheaves}%
\glossary{$\Edual_M$}%
\[
\Edual_M\colon \TDC(\icor_M)\to  \TDC(\icor_M)^\op, \qquad K \mapsto \cihom(K,\omega^\enh_M),
\]
where $\omega^\enh_M \eqdot \cor^\enh_M \tens \opb\pi\omega_M$.
\end{definition}
The functor $\Edual_M$ is related to the usual duality functor for sheaves by the formula:
\eq
&&\Edual_M(\cor_M^\enh\ctens F)\simeq\cor_M^\enh\ctens\opb{a}\RD_{M\times\R}F
\quad\text{in $\TDC(\icor_M)$,}
\label{eq:dualtwo}
\eneq
where $F\in\Derb(\cor_{M\times\R_\infty})$ and $a$ is the involution of $M\times\R$ given by $(x,t)\mapsto (x,-t)$.

\subsection{Constructible enhanced indsheaves}\label{subsection:cons}
In this subsection, we assume that  $M$ is a subanalytic space.
Recall the natural morphism of bordered spaces 
\[
j_M\colon M\times\R_\infty \to M\times\ol\R,
\]
and  the category $\BDC(\cor_{M\times\R_\infty})$ in Notation~\ref{not:fR}.

\begin{definition}
\glossary{$\BDC_\Rc(\cor_{M\times\R_\infty})$}%
One denotes by $\BDC_\Rc(\cor_{M\times\R_\infty})$ the full subcategory of $\BDC(\cor_{M\times\R_\infty})$ whose objects $F$ are such that $\reim{j_{M}}F$ is $\R$-constructible. 

We regard $\BDC_\Rc(\field_{M\times\R_\infty})$ as a full subcategory of $\BDC(\icor_{M\times\R_\infty})$.
\end{definition}

\begin{definition}\label{def:TRc}
One says that an object $K\in\TDC(\icor_M)$ is \emph{$\R$-constructible} 
\index{R@$\R$-constructible!enhanced indsheaves}%
if for any relatively compact subanalytic open subset $U\subset M$ there exists an isomorphism
\[
\text{$\opb\pi\cor_U \tens K \simeq \cor_M^\enh \ctens F$ for some $F\in \BDC_\Rc(\cor_{M\times\R_\infty})$.}
\]
\glossary{$\TDCrc(\icor_M)$}%
One denotes by $\TDCrc(\icor_M)$ the full subcategory of $\TDC(\icor_M)$ consisting of $\R$-constructible objects.
\end{definition}

Clearly, $\R$-constructible objects of $\TDC(\icor_M)$ are stable. One proves that:
\begin{theorem}\label{th:const}\hfill
\bnum
\item
The category $\TDC_\Rc(\icor_M)$ is triangulated.
\item
The property for  $K\in\TDC(\icor_M)$ of being $\R$-constructible is a local property over $M$.
\item
The functors $\ctens$ and $\cihom$ when restricted to $\R$-constructible objects give 
$\R$-constructible objects.
\item 
For $K\in \TDC_\Rc(\icor_M)$,  $\Edual_MK\in  \TDC_\Rc(\icor_M)$ and $\Edual_M\circ\Edual_MK\simeq K$.
\item \label{item:Dihom}
For $K_1,K_2\in \TDC_\Rc(\icor_M)$, $\Edual_M\bl\cihom(K_1,K_2)\br\simeq K_1\ctens\Edual_MK_2$.
\item \label{item:Dfhom}
For $K_1,K_2\in \TDC_\Rc(\icor_M)$, $\fihom(K_1,K_2)\simeq 
\fihom(\Edual_MK_2,\Edual_MK_1)$ and $\fhom(K_1,K_2)\simeq 
\fhom(\Edual_MK_2,\Edual_MK_1)$. 
\item
Let $f\cl M\to N$ be a morphism of subanalytic spaces.
\bna
\item If $G\in \TDC_\Rc(\icor_N)$, then $\Topb{f}G$ and $\Tepb{f}G$
belong to $\TDC_\Rc(\icor_M)$.
\item For $K\in \TDC_\Rc(\icor_M)$, we have
$\Teeim{f}K\in  \TDC_\Rc(\icor_N)$ if
$\Tsupp(K)\seteq\bpi_M\bl\Supp(\roim{j_M}\Tr K)\br$ is proper over $N$.
Here $j_M\cl M\times\fR\to M\times\bR$ is the inclusion and
$\bpi_M\cl M\times\bR\to M$ is the projection. 
\ee 
\enum 
\end{theorem}

Another link between classical $\R$-constructible sheaves and enhanced $\R$-constructible indsheaves is given by the following result, which is new:

\begin{theorem}\label{th:newconstruct}
For $F$, $G\in \TDC_\Rc(\icor_M)$,
the object $\fhom(F,G)$ belongs to $\Derb_{\Rc}(\cor_M)$.
\end{theorem}
\begin{proof}
Since
$$\fhom(F,G)\simeq\fhom(\cor_{\{t\ge0\}},\cihom(F,G)),$$
it is enough to show that
$$\fhom(\cor_{\{t\ge0\}},F)\in\Derb_{\Rc}(\cor_M)$$
for any $F\in \TDC_\Rc(\icor_M)$.
Now it follows from Corollary~\ref{cor:fhom} below 
which is a consequence of the following proposition.
\end{proof}
\Prop
For $K\in \Derb(\icor_{M\times\fR})$, we have
\eq
\roim{\pi}\rihom(\cor_{\{t\ge0\}},\cor_{\{t\gg0\}}\ctens K)
\simeq\roim{\pi}\bl \cor_{\{t>*\}}\tens K\br.
\eneq
Here 
$\cor_{\{t>*\}}\seteq\inddlim[{a\to-\infty}]\cor_{\{ t>a\}}\in\II[\cor_{M\times\fR}]$.
\glossary{$\cor_{\{t>*\}}$}%
\enprop

\Proof
We shall first show that
\eq
&&\ba{l}
H^n\roim{\pi}\rihom(\cor_{\{t\ge0\}},\cor_{\{t\gg0\}}\ctens K)\\[1ex]
\hs{20ex}\simeq
\inddlim[{a\to-\infty}]H^n\roim{\pi}\rihom(\cor_{\{t\ge a\}}, K)\\[1.5ex]
\hs{35ex}\simeq
\inddlim[{a\to-\infty}]H^n\roim{\pi}\bl \cor_{\{t>a\}}\tens K\br.
\ea\label{eq:twoiso}
\eneq

The first isomorphism in \eqref{eq:twoiso} follows from 
\eqn
&&H^n\roim{\pi}\rihom(\cor_{\{t\ge0\}},\cor_{\{t\gg0\}}\ctens  K)\\
&&\simeq
\inddlim[{a\to+\infty}]H^n\roim{\pi}\rihom\bl\cor_{\{t\ge0\}},
\cor_{\{t\ge a\}}\ctens  K\br\\
&&\underset{\mathrm{(a)}}{\simeq}
\inddlim[{a\to+\infty}]H^n\roim{\pi}\rihom\bl\cor_{\{t\ge0\}},
\cihom(\cor_{\{t\ge -a\}},  K)\br\\
&&\simeq
\inddlim[{a\to+\infty}]H^n\roim{\pi}\rihom(\cor_{\{t\ge -a\}}, K).
\eneqn
Here, isomorphism (a) follows from Proposition~\ref{prop:tenhom}
and $\reeim{\pi}(\cor_{\{t\ge0\}})\simeq0$.

Let us next show the second isomorphism in \eqref{eq:twoiso}.
There is a sequence of morphisms
\eqn
&&\inddlim[{a\to-\infty}]H^n\roim{\pi}\bl\cor_{\{t>a\}}\tens  K\br
\To[f_1]\inddlim[{a\to-\infty}]H^n\roim{\pi}\rihom(\cor_{\{t\ge a\}}, K)\\
&&\hs{15ex}\To[f_2]\inddlim[{a\to-\infty}]H^n\roim{\pi}\bl\cor_{\{t>a-1\}}\tens  K\br\\
&&\hs{30ex}\To[f_3]\inddlim[{a\to-\infty}]
H^n\roim{\pi}\rihom(\cor_{\{t\ge a-1\}}, K).
\eneqn
Since $f_2f_1$ and $f_3f_2$ are isomorphisms,
$f_1$ is an isomorphism.

\noi
Thus we have proved \eqref{eq:twoiso}. 

\smallskip
Then \eqref{eq:twoiso} implies that
\eqn
&&H^n\roim{\pi}\rihom(\cor_{\{t\ge0\}},\cor_{\{t\gg0\}}\ctens \opb{j_M}L)
\simeq H^n\roim{\pi}\bl \cor_{\{t>*\}}\tens \opb{j_M}L\br\simeq0
\eneqn
for any $n\in\Z\setminus \{0\}$ and 
any quasi-injective $L\in \II[\cor_{M\times\bR}]$.
Here $j_M\colon M\times\R_\infty \to M\times\bR$
is the natural morphism bordered spaces.
Therefore,
the functor $\roim{\pi}
\rihom(\cor_{\{t\ge0\}},\cor_{\{t\gg0\}}\ctens \opb{j_M}\scbul)$ 
is isomorphic to the right derived functor of
$H^0\roim{\pi}\rihom(\cor_{\{t\ge0\}},\cor_{\{t\gg0\}}\ctens \opb{j_M}\scbul)$.
Similarly,
$\roim{\pi}\bl\cor_{\{t>*\}}\tens \opb{j_M}\scbul\br$ is isomorphic to 
the right derived functor
of $H^0\roim{\pi}\bl\cor_{\{t>*\}}\tens \opb{j_M}\scbul\br$.
Since $H^0\roim{\pi}\rihom(\cor_{\{t\ge0\}},\cor_{\{t\gg0\}}\ctens \opb{j_M}\scbul)$
and $H^0\roim{\pi}\bl \cor_{\{t>*\}}\tens \opb{j_M}\scbul\br$ are isomorphic 
by \eqref{eq:twoiso}, 
we obtain the desired result.
\QED

\Cor\label{cor:fhom} For any $F\in\Derb(\cor_{M\times\fR})$, we have
an isomorphism in $\BDC(\cor_M)$:
\eq
\fhom(\cor_{\{t\ge0\}},\cor^\Tam_M\ctens F)
\simeq\roim{\bpi}\bl \cor_{M\times(\bR\setminus\{-\infty\})}\tens
\roim{j_M}F\br,
\eneq
where $\bpi\cl M\times\bR\to M$ is the projection
and  $j_M\colon M\times\R\to M\times\bR$
is the inclusion.
\encor
\Proof
We have 
\begin{align*}
\roim{\pi}\bl \cor_{\{t>*\}}\tens F\br
&\simeq
\roim{\bpi}\roim{j_M}\bl \cor_{\{t>*\}}\tens F\br\\
&\simeq \roim{\bpi}\bl \cor_{\{+\infty\ge t>*\}}\tens \roim{j_M}F\br,
\end{align*}
where $\cor_{\{+\infty\ge t>*\}}\seteq\inddlim[{a\to-\infty}]\cor_{\{+\infty\ge t>
a\}}\simeq \roim{j_M} \cor_{\{t>*\}}
\in \II[\cor_{M\times\bR}]$. 
Hence we have 
\begin{align*}
\fhom(\cor_{\{t\ge0\}},\cor^\Tam_M\ctens F)
&\simeq\alpha_M\fihom(\cor_{\{t\ge0\}},\cor^\Tam_M\ctens F)\\
&\simeq \alpha_M\roim{\bpi}\bl \cor_{\{+\infty\ge t>*\}}\tens \roim{j_M}F\br\\
&\simeq
\roim{\bpi} \alpha_{M\times\bR}\bl\cor_{\{+\infty\ge t>*\}}\tens \roim{j_M}F\br\\
&\simeq\roim{\bpi} \bl \cor_{M\times(\bR\setminus\{-\infty\})}\tens \roim{j_M}F\br.
\end{align*}
\QED

\subsection{Enhanced indsheaves with ring action}
Let $\A$ be a sheaf of $\cor$-algebras on $M$.
For $\dagger=\hs{2ex}, \mathrm{b}, +, -$, we define
$$\Der[\dagger](\II[\pi^{-1}\A])\seteq
\Der[\dagger](\II[\bpi^{-1}\A])/\Der[\dagger](\II[(\bpi^{-1}\A)\vert_{M\times(\bR\setminus\R)}]),$$
where $\bpi\cl M\times\bR\to M$ is the projection.
Then we set
$$\TDCC^\dagger(\JA)
=\Der[\dagger](\II[\pi^{-1}\A])/\set{K\in\Der[\dagger](\II[\pi^{-1}\A])}{
\text{$K\simeq\pi^{-1}L$ for some $L\in\Der[\dagger](\JA)$}}.$$
We call objects of $\TDC(\JA)$ enhanced indsheaves with $\A$-action. 

We can define also the functors
\glossary{$\ctens[\beta\A]$}%
\glossary{$\cihom[\beta\A]$}%
\eqn
&&\ctens[\beta\A]:\TDC(\JA^\rop)\times \TDC(\JA)\to \TDCC^-(\icor_M),\\
&&\cihom[\beta\A]:\TDC(\JA)^\rop\times \TDC(\JA)\to \TDCC^+(\icor_M),
\eneqn
which satisfy similar properties to $\ctens$ and $\cihom$.

Similarly we can define
\eqn
\ltens[\A]:\TDC(\JA^\rop)\times \Derb(\A)\to \TDCC^-(\icor_M),\\
\rhom[\A]:\Derb(\A)^\rop \times\TDC(\JA)\to \TDCC^-(\icor_M).
\eneqn
If $X$ is a complex manifold and $\A=\D_X$, we can define
\glossary{$\Dtens$}%
\eqn
\Dtens:\TDC(\JD_X^\rop)\times \Derb(\D_X)\to \TDCC^-(\JD_X).
\eneqn

\section{Holonomic D-modules}\label{section:hol}

\subsection{Exponential D-modules}\label{subsection:expDmod}
\index{exponential D-modules}%
Let $X$ be a complex analytic manifold,  $Y\subset X$ a complex analytic hypersurface and set $U=X\setminus Y$.
For $\varphi\in\OO(*Y)$, one sets
\glossary{$\D_X \ex^\varphi$}%
\glossary{$\she^\varphi_{U\vert X}$}%
\begin{align*}
\D_X \ex^\varphi &= \D_X/\set{P}{P\ex^\varphi=0 \text{ on } U}, \\
\she^\varphi_{U|X}&=\D_X \ex^\varphi(*Y).
\end{align*}
Hence $\D_X \ex^\varphi$ is a $\D_X$-submodule of $\she^\varphi_{U|X}$, and $\D_X \ex^\varphi$ as well as
 $\she^\varphi_{U|X}$ are holonomic $\D_X$-modules.
Note that  $\she^\varphi_{U|X}$ is isomorphic to $\OO(*Y)$ as an $\OO$-module,
and the connection $\OO(*Y)\to\Omega^1_X\tens[\OO]\OO(*Y)$ is given by
$u\mapsto du+u d\vphi$. 

For $c\in\R$, set for short
\[
\{\Re \varphi < c\} \seteq \{x\in U;\Re \varphi(x) < c\} \subset X.
\]

\begin{notation}\label{not:<?}
One sets
\glossary{$\C_{\{\Re \varphi <\ast\}}$}%
\glossary{$E^\vphi_{{U\vert X}}$}%
\begin{align*}
\C_{\{\Re \varphi <\ast\}} &\eqdot \inddlim[c\rightarrow+\infty]\C_{\{\Re \varphi < c\}} \in \iC_X, \\
E^\varphi_{U|X} &\eqdot \rihom(\C_U,\C_{\{\Re \varphi <\ast\}}) \in \BDC(\iC_X).
\end{align*}
\end{notation}

For example, denoting by $z \in \C \subset \PP$ the affine coordinate of the complex projective line, one has 
\begin{equation}
\label{eq:HjEt}
H^j E^z_{\C|\PP} \simeq
\begin{cases}
\C_{\{\Re z <\ast\}} &\text{for }j=0, \\
\C_{\{\infty\}} &\text{for }j=1, \\
0 &\text{otherwise}.
\end{cases}
\end{equation}

The next result (see~\cite[Prop.~6.2.2]{DK13}) generalizes~\cite[Proposition~7.3]{KS03} in which the case 
$X=\C$ and $\varphi(z)=1/z$ was treated 
(and recalled in \S\,\ref{subsection:exairreg}).

\begin{proposition}\label{pro:Solphi}
Let $Y\subset X$ be a closed complex analytic hypersurface and set $U=X\setminus Y$.
For $\varphi\in\OO(*Y)$, there is an isomorphism in $\BDC(\iC_X)$
\[
\drt_X(\she^{-\varphi}_{U|X}) \simeq E^\varphi_{U|X}[d_X].
\]
\end{proposition}

\subsection{Enhanced tempered holomorphic functions}\label{subsection:enhhol}
\index{enhanced!tempered holomorphic functions}%
Consider first a real  analytic manifold $M$ and the natural morphism of bordered spaces
\[
j\colon M\times\R_\infty \to M\times\PR .
\]
Let $t$ be a coordinate of $\fR$.
\begin{definition}\label{def:DbT}
One sets $\Dbt_{M\times\R_\infty}\eqdot\epb{j}\Dbt_{M\times\PR}$ and 
one denotes by  
\glossary{$\DbT_M$}%
$\DbT_M\in\Derb(\iC_{M\times\R_\infty})$  the complex, concentrated in degree $-1$ and $0$:
\eq\label{eq:DbTM}
&&\DbT_M\eqdot \Dbt_{M\times\R_\infty}\to[\partial_t-1]\Dbt_{M\times\R_\infty}.
\eneq
\end{definition}
Note that $H^k(\DbT_M) = 0$ for $k\neq-1$.

\begin{proposition}\label{pro:DbTstable}
There are isomorphisms in $\BDC(\iC_{M\times\R_\infty})$
\begin{align*}
\DbT_M &\isoto \cihom(\C_{\{t\geq0\}}, \DbT_M) \\
&\isofrom \cihom(\C_{\{t\geq a\}}, \DbT_M) \quad \text{for any $a\geq 0$}.
\end{align*}
\end{proposition}
Moreover, denoting by $\iota\cl M\times\R\to M\times\R_\infty$  the natural morphism, one has
the isomorphism 
$\opb{\iota}\DbT_M\simeq\opb{\iota}\opb{\pi}\Dbt_M\,[1]$ and therefore:
\eq\label{eq:dbTdbt}
&&\fihom(\C_{\{t=0\}},\DbT_M)\simeq\Dbt_M.
\eneq

Now let $X$ be again a complex manifold. 

\begin{definition}\label{def:OEn}
\glossary{$\OEn_X$}%
\glossary{$\OvE_X$}%
One sets
\eqn
&&\OEn_X=\rhom[\opb{\pi}\D_{X^c}](\opb{\pi}\OO[X^c],\DbT_{X_\R}),
\\
&&\OvE_X= \opb{\pi}\Omega_X \tens[{\opb{\pi}\OO}]\OEn_X.
\eneqn
 We regard them as objects of $\TDC(\JD_X)$ and  $\TDC(\JD_X^\rop)$,
respectively. 
One  calls $\OEn_X$ the 
{\em enhanced indsheaf of  tempered holomorphic functions}.
\end{definition}

\begin{remark}
When $X=\rmpt$, we have 
$\DbT_{X}\simeq \C_{\{t<*\}}[1]$
(see Notation~\ref{not:gg}) as an object of $\Derb(\iC_\fR)$ and 
$\OEn_X\simeq\C_X^\enh$ as an object of
$\TDC(\iC_X)$.
\end{remark}
Applying Proposition~\ref{pro:DbTstable}, we get

\begin{proposition}\label{pro:OEstable}
There are isomorphisms in $\TDC(\JD_X)$
\eqn
\OEn_X &\isoto& \cihom(\C_{\{t\geq 0\}}, \OEn_X) \\
&\isofrom& \cihom(\C_{\{t\geq a\}}, \OEn_X) \quad
\text{for any $a\geq 0$}.
\eneqn
In particular, $\OEn_X$ is a stable object in $\TDC(\JD_X)$.
\end{proposition}

As a consequence of Proposition~\ref{pro:OEstable} and Proposition~\ref{pro:equivTam}, we get the following result.

\begin{corollary}
\label{cor:CTamOT}
There are isomorphisms in $\TDC(\JD_X)$
\eqn
&&\OEn_X \simeq \cihom(\C_X^\enh, \OEn_X)\simeq \C_X^\enh \ctens \OEn_X.
\eneqn
\end{corollary}

Then, using the isomorphisms
\eqn
\fihom(\C_X^\enh,\OEn_X)&\simeq&\fihom(\C_X^\enh, \cihom(\C_X^\enh, \OEn_X))\\
&\simeq&\fihom(\C_X^\enh\ctens \C_X^\enh, \OEn_X)\\
&\simeq&\fihom(\C_{\{t=0\}}\ctens \C_X^\enh, \OEn_X)\\
&\simeq&\fihom(\C_{\{t=0\}}, \OEn_X)
\eneqn
and~\eqref{eq:dbTdbt}, one gets the isomorphism in $\Derb(\JD_{X})$:
\eq\label{eq:oEot}
&&\fihom(\C_X^\enh,\OEn_X)\simeq\Ot.
\eneq

\subsection{Enhanced de Rham and Sol functors}\label{subsection:Edrsol}
\index{de Rham functor!enhanced}%
\index{solution functor!enhanced}%
For $\shm\in\BDC(\D_X)$, set
\glossary{$\drE_X$}%
\glossary{$\solE_X$}%
\eqn
\drE_X(\shm) &\seteq& \OvE_X \ltens[\D_X] \shm, \\
\solE_X(\shm) &\seteq& \rhom[\D_X](\shm,\OEn_X).
\eneqn
We get functors
\begin{align*}
\drE_X &\colon \BDC(\D_X) \to \TDC(\iC_X), \\
\solE_X &\colon \BDC(\D_X)^\op \to \TDC(\iC_X).
\end{align*}
Note that
\[
\solE_X(\shm) 
\simeq \drE_X(\Ddual_X\shm)[-d_X]\quad\text{for $\shm\in\Derb_\coh(\D_X)$.}
\]

By \eqref{eq:oEot}, we have for any $\shm\in\Derb(\D_X)$
\eq
\ba{l}
\drt_X\shm\simeq \fihom(\C_X^\enh,\drE_X\shm),\\[1ex]
\dr_X\shm\simeq \fhom(\C_X^\enh,\drE_X\shm).
\ea\label{eq:DRo}
\eneq

By using Proposition~\ref{pro:Solphi}, one can calculate explicitly $\drE_X(\shm)$ when $\shm$ is an exponential D-module. 

\begin{proposition}\label{pro:Solphi2}
Let $Y\subset X$ be a closed complex analytic hypersurface, and set $U=X\setminus Y$.
For $\varphi\in\OO(*Y)$, there are isomorphisms 
\eqn
\drE_X(\she^\varphi_{U|X})&\simeq&\rihom(\opb{\pi}\C_U,\inddlim[c\to\infty]\C_{\{t\geq\Re\varphi+c\}})\\
&\simeq&\C_X^\enh\ctens\rihom(\opb{\pi}\C_U,\C_{\{t=\Re\varphi\}}).
\eneqn
\end{proposition}
The next results are easy consequences of Theorem~\ref{thm:ifunct0}, Corollary~\ref{cor:ifunct1},
Corollary~\ref{cor:ifunct2} and Corollary~\ref{cor:ifunct2b}.

\begin{theorem}\label{thm:Tfunct}
Let $f\colon X\to Y$ be a morphism of complex manifolds.
\bnum
\item
There is an isomorphism in $\TDC(\II[\opb f\D_Y])$
\[
\Tepb f \OEn_Y[d_Y] \simeq \D_{Y\from X} \ltens[\D_X] \OEn_X [d_X].
\]
\item
For any $\shn\in\BDC(\D_Y)$ there is an isomorphism in $\TDC(\iC_X)$
\eq
\drE_X(\Dopb f \shn)[d_X] \simeq \Tepb f \drE_Y(\shn)[d_Y].
\label{eq:invdr}
\eneq
\item
Let $\shm\in\BDC_\good(\D_X)$, and assume that $\Supp(\shm)$ is proper over $Y$. 
Then, there are isomorphisms in $\TDC(\iC_Y)$ 
\eqn
&&\drE_Y(\Doim f\shm) \simeq  \Toim f\drE_X(\shm),\\
&&\Doim f(\OEn_X\Dtens\shm) \simeq  \OEn_Y\Dtens \Doim{f}\shm.
\eneqn
\ee
\end{theorem}

\subsection{Ordinary linear differential equations and Stokes phenomena}
\label{sec:ordinary}

Let us recall the local theory of ordinary linear differential equations.
Let $X\subset \C$ be open with $0\in X$
and let $\shm$ be a holonomic $\shd_X$-module
such that $\SSupp(\shm)\subset \{0\}$ and
$\shm\simeq\shm(*\{0\})$. Then $\shm$ is a locally free $\OO(*\{0\})$-module
of finite rank.
Let us take a system of generators
$(u_1,\ldots, u_r)$ of $\shm$ as an $\OO(*\{0\})$-module
on a neighborhood of $0$.
Then, setting $\abu$ the column vector consisting of these generators, we have
\eqn
&&\dfrac{d}{dz}\abu=\bA(z)\abu
\eneqn
for an $(r\times r)$-matrix $\bA(z)$ whose components are in $\OO(*\{0\})$.
Then for any $\D_X$-module $\shl$ such that $\shl\simeq\shl(*\{0\})$,
we have
\eq
&&\hom[\D_X](\shm,\shl)=\set{\au\in\shl^r}{\text{$\au$ satisfies
equation \eqref{eq:ord} below}}
\label{eq:solM}
\eneq
where we associate  to $\au$ the morphism from $\shm$ to $\shl$
defined by $\abu\mapsto\au$. 
Here
\eq\dfrac{d}{dz}\au=\bA(z)\au \label{eq:ord}.
\eneq
Now we have the following results on the solutions of the
ordinary linear differential equation \eqref{eq:ord}.
\bnum
\item
there exist linearly independent
$r$ formal (column) solutions $\hbu_j$ $(j=1,\ldots r)$ of \eqref{eq:ord}
with the form 
\eqn
\hbu_j=\ex^{\vphi_j(z)}z^{\la_j}\sum_{k=0}^{r-1}\va_{j,k}(z)(\log z)^k,
\eneqn
where $m\in\Z_{>0}$, $\vphi_j(z)\in z^{-1/m}\C[z^{-1/m}]$, $\la_j\in\C$, and
$$\text{$\va_{j,k}(z)=\sum\limits_{n\in m^{-1}\Z_{\ge0}}  \va_{j,k,n} z^n\in\C[[z^{1/m}]]^r$
with  $\va_{j,k,n}\in\C^r$,}$$
\item 
for any $\theta_0\in \R$ and each $j=1,\dots,r$, there exist an angular neighborhood
$$D_{\theta_0}=\set{z=r\ex^{i\,\theta}}{\text{$|\theta-\theta_0|<\vep$
and $0<r<\delta$}}$$
for sufficiently small $\eps,\delta>0$
and holomorphic (column) solution $\bu_j\in\OO(D_{\theta_0})^r$
of \eqref{eq:ord} defined on $D_{\theta_0}$
such that
$$\bu_j\sim \hbu_j,$$
in the sense that, for any $N>0$, there exists $C>0$ such that
\eq
&&|\bu_j(z)- \hbu_j^{N}(z)|\le C|\ex^{\vphi_j(z)}z^{\la_j+N}|
= C\ex^{\Re(\vphi_j(z))}|z^{\la_j+N}|,\label{eq:asymptotic}
\eneq
where $\hbu_j^{N}(z)$ is the finite partial sum
$$\hbu_j^{N}(z)=
\ex^{\vphi_j(z)}z^{\la_j}\sum_{k=0}^{r-1}
\hs{1ex}\sum\limits_{\substack{n\in m^{-1}\Z_{\ge0},\\n\le N}}\hs{1ex} \va_{j,k,n}  z^n(\log z)^k.$$
Here we choose branches of $z^{1/m}$ and $\log z$ on $D_{\theta_0}$.
\ee

Note that a holomorphic solution $\bu_j$ is not uniquely determined
by the formal solution $\hbu_j$.
In fact, $\bu_j+\sum_{k}c_k\bu_k$ also satisfies
the same estimate \eqref{eq:asymptotic}
whenever $$\text{$\Re(\vphi_k(z))<\Re(\vphi_j(z))$ on $D_{\theta_0}$
if $c_k\not=0$.}$$

\medskip
We can interpret these results as follows.
Let $\vpi\cl\twX\to X$ be the real blow up of $X$ along $\{0\}$
defined 
in \S\;\ref{subsection:realblowup}.
Then $\ex^{-\vphi_j(z)}\bu_j$  gives a section of
$(\At)^r$ on a neighborhood of $\e^{i\,\theta_0}\in\vpi^{-1}(0)$.
Define the $\D_{\twX}^\tA$-module
$$\shl_j\seteq\D_{\twX}^\tA\e^{\vphi_j(z)}
=\DA_\twX/\DA_\twX\bl\frac{d}{dz}-{\vphi'_j}(z)\br.$$
Here we take a branch of $\vphi_j$ on a domain and $\shl_j$
is defined on such a domain.

Then $(\ex^{-\vphi_j(z)}\bu_j)\ex^{\vphi_j}\in(\shl_j)^r$
is a solution of
equation \eqref{eq:ord}, and hence \eqref{eq:solM}
defines a morphism of $\DA_\twX$-modules
$$\shm^\tA\to \shl_j.$$
Collecting such a morphism for all $j$,
we obtain an isomorphism defined on a neighborhood of $\e^{i\,\theta_0}\in\vpi^{-1}(0)$:
\eq
\shm^\tA\isoto \soplus_{j=1}^r\shl_j.
\label{eq:twlocal}
\eneq
Note that
$$\oim{\vpi}\shm^\tA\simeq\shm.$$
However, these isomorphisms \eqref{eq:twlocal}
are not globally  defined.
That is, $\shm^\tA$ is only locally isomorphic to $\soplus_{j=1}^r\shl_j$.
We have
\eq\hom[\DA_\twX](\shl_j,\shl_{j'})\vert_{\vpi^{-1}(0)}
\simeq \C_{U_{j,j'}}\subset \C_{\vpi^{-1}(0)},\label{eq:Ljj}
\eneq
where
\eq
&&U_{j,j'}=\set{p\in\vpi^{-1}(0)}%
{\parbox{40ex}{$\Re(\vphi_j(z))\le\Re(\vphi_{j'}(z))$
on $U\cap \twX^{>0}$ for a neighborhood $U$ of $p$}}.\label{cond:temp}
\eneq 
Indeed, any morphism 
$f\in \hom[\DA_\twX](\shl_j,\shl_{j'})\vert_{\vpi^{-1}(0)}$
should have the form $\e^{\vphi_j}\mapsto \e^{\vphi_j(z)-\vphi_{j'}(z)}\,
\e^{\vphi_{j'}}$ up to a constant multiple, 
and hence $f$ is well-defined if and only if
$\e^{\vphi_j(z)-\vphi_{j'}(z)}$ is tempered. The last condition is equivalent to
the condition in \eqref{cond:temp}. 

Hence the isomorphism class of a $\DA_\twX\vert_{\vpi^{-1}(0)}$-module $\shl$
locally isomorphic to $\soplus_{j=1}^r\shl_j\vert_{\vpi^{-1}(0)}$
is determined
by a topological data,
\index{Stokes matrices}%
the so-called  Stokes matrices.

Assuming that $m=1$ for the sake of simplicity, let us explain them more precisely. 
Let  $\shl$ be a $\DA_\twX\vert_{\vpi^{-1}(0)}$-module
locally isomorphic to $\soplus_{j=1}^r\shl_j\vert_{\vpi^{-1}(0)}$.
We identify $\vpi^{-1}(0)$ with $\R/2\pi\Z$ by $\R/2\pi\Z\ni\theta
\mapsto \ex^{i\,\theta}\in \vpi^{-1}(0)$.
Let us take $\{\theta_1,\ldots,\theta_s\}$
such that $s\ge2$,
$\theta_0<\theta_1<\cdots<\theta_{s-1}<\theta_{s}$
and
$$\vpi^{-1}(0)\bigcap\bigcup_{\substack{
1\le j,\,j'\le r,\\\vphi_j\not=\vphi_{j'}}}\ol{\set{z\in\twX^{>0}}%
{\Re\vphi_j(z)=\Re \vphi_{j'}(z)}}\subset\{\theta_1,\ldots,\theta_s\}.$$
Here we set $\theta_{k+ls}=\theta_k+2\pi l$ for $1\le k\le s$ and $l\in\Z$. 
Set $V_k=\set{\theta}{\theta_{k-1}<\theta<\theta_{k+1}}$
and $W_k=\set{\theta}{\theta_{k}<\theta<\theta_{k+1}}=V_k\cap V_{k+1}$.
Then we have
$\vpi^{-1}(0)=\bigcup_{1\le k\le s}V_k$.  Note that
for any $j,j'\in\{1,\ldots, r\}$ and $k\in\{1,\ldots,s\}$, we have
either
$W_k\cap U_{j,j'}=\emptyset$ or $W_k\subset U_{j,j'}$. In particular,
\eqref{eq:Ljj} implies that
$\hom[\DA_\twX](\shl_j,\shl_{j'})\vert_{W_k}$ is a constant sheaf. 

Hence, any isomorphism
$\shl\isoto\soplus_{j=1}^r\shl_j\vert_{\vpi^{-1}(0)}$ defined on a neighborhood of $\theta_k$
can be extended to an isomorphism defined on $V_k$,
and we have an isomorphism 
$$\psi_k\cl\shl\vert_{V_k}\isoto\soplus_{j=1}^r\shl_j\vert_{V_k}.$$
Let us set
$$\xi_k=\psi_{k+1}\circ\psi_k^{-1}\cl \soplus_{j=1}^r\shl_j\vert_{W_k}\isoto
\soplus_{j=1}^r\shl_j\vert_{W_{k}}.$$
Then $\shl$ is obtained by patching
$\soplus_{j=1}^r\shl_j\vert_{V_k}$ by the $\xi_k$'s.
Each isomorphism $\xi_k$ is given by the matrix
$S_k=(s_{k;i',i})_{1\le i,i'\le r}\in\GL_r(\C)$.
Here
$s_{k;i',i}\in\C$ is given by the  morphism
$$\xymatrix
{\shl_i\vert_{W_k}\;\ar@{>->}[r]&\soplus_{j=1}^r\shl_j\vert_{W_k}\ar[r]^\sim_{\xi_k}&
\soplus_{j=1}^r\shl_j\vert_{W_k}\;\ar@{->>}[r]&\shl_{i'}\vert_{W_k}}$$
through
\eqn
&&\sect\bl W_k;\hom[\DA_\twX](\shl_i,\shl_{i'})\vert_{\vpi^{-1}(0)}\br
\simeq \sect\bl W_k;\C_{U_{i,i'}}\br
\subset\sect\bl W_k; \C_{\vpi^{-1}(0)}\br\simeq\C
\eneqn
due to \eqref{eq:Ljj}.
Hence, we have
$s_{k;i',i}=0$ if $W_k\not\subset U_{i,i'}$.

The matrices $\{S_k\}_{1\le k\le s}$ are called the {\em Stokes matrices}.
Conversely, for a given family of matrices $\{S_k\}_{1\le k\le s}$,
we can find a $\DA_\twX\vert_{\vpi^{-1}(0)}$-module $\shl$
locally isomorphic to $\soplus_{j=1}^r\shl_j\vert_{\vpi^{-1}(0)}$
by patching $\soplus_{j=1}^r\shl_j\vert_{V_k}$ by $\{S_k\}_{1\le k\le s}$.
Then $\shm$ is recovered from $\shl$ by $\shm\simeq\oim{\vpi}\shl$.

\subsection{Normal form}\label{subsection:normalform}

\index{normal form!}%
The results in \S\;\ref{sec:ordinary} 
are generalized to higher dimensions
by T.\ Mochizuki (\cite{Mo09,Mo11}) and K.\ S.\ Kedlaya (\cite{Ke10,Ke11}).
In this subsection, we collect some of their results 
that we shall need. 

Let $X$ be a complex manifold and $D\subset X$ a normal crossing divisor. 
We shall use the notations introduced in \S\,\ref{subsection:realblowup} : in particular 
the real blow up $\vpi\cl\twX\to X$ and 
the notation $\shm^\tA$ of~\eqref{eq:MA}.

\begin{definition}\label{def:normal}
We say that a holonomic $\D_X$-module $\shm$ has \emph{a normal form} along $D$ if\\
${\rm (i)}$ $\shm\simeq \shm(*D)$,\\
${\rm (ii)}$  $\SSupp(\shm)\subset D$,\\
${\rm (iii)}$  for any $x\in\vpi^{-1}(D)\subset\twX$, there exist an open neighborhood $U\subset X$ of $\varpi(x)$ and finitely many $\varphi_i\in\sect(U;\sho_X(*D))$
such that
\eqn
&&(\shm^\tA)|_V \simeq
\left.\left( \bigoplus_i (\she_{U\setminus D|U}^{\varphi_i})^\tA \right)\right|_V
\eneqn
for some open neighborhood $V$ of $x$ with $V\subset\opb\varpi(U)$.
\end{definition}

A ramification 
\index{ramification}%
of $X$ along $D$ on a neighborhood $U$ of $x\in D$ is a finite map
\[
p \colon X' \to U
\]
of the form $p(z') = (z_1^{\prime\, m_1},\dots,z_r^{\prime\, m_r},z'_{r+1},\dots,z'_n)$ for some $(m_1,\dots,m_r)\in(\Z_{>0})^r$.
Here $(z'_1,\dots,z'_n)$ is a local coordinate system on $X'$, $(z_1,\dots,z_n)$ a local coordinate system on $X$
such that   $D=\{z_1\cdots z_r=0\}$.

\begin{definition}
\label{def:quasi-normal}
\index{normal form!quasi-}%
\index{quasi-normal form}%
We say that a holonomic $\D_X$-module $\shm$ has \emph{a quasi-normal form} along $D$ if it satisfies (i) and (ii) in Definition~\ref{def:normal}, and if for any $x\in D$ there exists a ramification $p\colon X'\to U$ on a neighborhood $U$ of $x$ such that $\Dopb p (\shm|_U)$ has a normal form along $\opb p (D\cap U)$.
\end{definition}

\begin{remark}
In the above definition, 
$\Dopb p(\shm|_U)$ as well as $\Doim p\Dopb p(\shm|_U)$ is concentrated in degree zero and
$\shm|_U$ is a direct summand of $\Doim p\Dopb p (\shm|_U)$.
\end{remark}

The next result is an essential tool in the study of holonomic D-module 
and is easily deduced from the 
fundamental work of  Mochizuki~\cite{Mo09,Mo11} 
(see also Sabbah~\cite{Sa00} for preliminary results  and see
Kedlaya~\cite{Ke10,Ke11} for the analytic case).

\begin{theorem}
\label{thm:normal}
Let $X$ be a complex manifold, $\shm$ a holonomic $\D_X$-module and $x\in X$.
Then there exist an open neighborhood $U$ of $x$, a closed analytic hypersurface $Y\subset U$, a complex manifold $X'$ and a projective morphism $f\colon X'\to U$ such that
\bnum
\item $\SSupp(\shm)\cap U\subset Y$,
\item $D\eqdot\opb f(Y)$ is a 
normal crossing divisor  of $X'$,
\item $f$ induces an isomorphism $X'\setminus D \to U \setminus Y$,
\item $(\Dopb f \shm)(*D)$ has a quasi-normal form along $D$.
\ee
\end{theorem}

Remark that, under assumption (iii), $(\Dopb f \shm)(*D)$ is concentrated in degree zero.

Using Theorem~\ref{thm:normal}, one easily deduces the next lemma.

\begin{lemma}\label{lem:redux}
Let $P_X(\shm)$ be a statement concerning a complex manifold $X$ and a holonomic object $\shm\in\BDC_\hol(\D_X)$. Consider the following conditions.
\bna
\item
Let $X=\Union\nolimits_{i\in I}U_i$ be an open covering. Then $P_X(\shm)$ is true if and only if $P_{U_i}(\shm|_{U_i})$ is true for any $i\in I$.
\item
If $P_X(\shm)$ is true, then $P_X(\shm[n])$ is true for any $n\in\Z$.
\item
Let $\shm'\to\shm\to\shm''\To[+1]$ be a distinguished triangle in $\BDC_\hol(\D_X)$. If $P_X(\shm')$ and $P_X(\shm'')$ are true, then $P_X(\shm)$ is true.
\item
Let $\shm$ and $\shm'$ be holonomic $\D_X$-modules. 
If $P_X(\shm\dsum\shm')$ is true, then $P_X(\shm)$ is true.
\item Let $f\colon X\to Y$ be a projective morphism and $\shm$ a good holonomic $\D_X$-module. If $P_X(\shm)$ is true, then $P_Y(\Doim f\shm)$ is true.
\item If $\shm$ is a holonomic $\D_X$-module with a normal form along a normal crossing divisor of $X$, then $P_X(\shm)$ is true.
\ee
If conditions {\rm (a)--(f)} are satisfied, 
then $P_X(\shm)$ is true for any complex manifold $X$ and any $\shm\in\BDC_\hol(\D_X)$.
\end{lemma}

\Proof[Sketch of the proof]
The proof is similar to the regular case (Lemma~\ref{lem:reduxreg}).

We shall only prove here
that $P_X(\shm)$ is true for any holonomic $\D_X$-module $\shm$
which has a quasi-normal form along a normal crossing divisor $D$.

Let $p\cl X'\to U$ be as in Definition~\ref{def:quasi-normal}.
Then $\Dopb p (\shm|_U)$ has a normal form along $\opb p (D\cap U)$.
Hence $P_{X'}\bl\Dopb p (\shm|_U)\br$  is true by hypothesis (f).
Hence
$P_U\bl\Doim p\Dopb p (\shm|_U)\br$ is true by hypothesis (e).
We have a chain of morphisms 
$$\shm\vert_U\to \Doim p\Dopb p (\shm|_U)\to\shm\vert_U,$$
whose composition is equal to $m\id_\shm$ where
$m$ is the number of the generic fiber of $p$.
Hence $\shm\vert_U$ is a direct summand of $\Doim p\Dopb p (\shm|_U)$.
Then, hypothesis (d) implies that $P_U(\shm\vert_U)$ is true.
\QED

\subsection{Enhanced de Rham functor on the real blow up}

By Lemma~\ref{lem:redux}, many statements on holonomic D-modules can be 
reduced to the normal form case.
In order to investigate this case, we shall introduce the enhanced de Rham functor
on the real blow up.

Let $D$ be a normal crossing divisor of a complex manifold $X$ and
let $\vpi\cl \twX\to X$ be the real blow up of $X$ along $D$
as in \S\;\ref{subsection:realblowup}.
Let $j\cl\twX\times \fR\to\twX^\tot\times \PR$ be the canonical morphism of bordered spaces.
Similarly to \eqref{eq:sheafDbtwX},
we set
\eqn
&&\Dbt_{\twX\times\fR}\seteq\opb{j}
\ihom(\C_{\twX^{>0}\times\R},\Dbt_{\twX^\tot\times\PR}).
\eneqn
Then as in  in Definition~\ref{def:DbT} 
one denotes by  $\DbT_\twX\in\Derb(\iC_{\twX\times\R_\infty})$  the complex, concentrated in degree $-1$ and $0$:
\eq\label{eq:DbTtwX}
&&\DbT_\twX\eqdot \Dbt_{\twX\times\R_\infty}\To[\partial_t-1]\Dbt_{\twX\times\R_\infty},
\eneq
and finally  as in Definition~\ref{def:OEn}, one sets
\eqn
&&\OEn_\twX=\rhom[\opb{\pi}\D_{X^c}](\opb{\pi}\OO[X^c],\DbT_{\twX})
,\\
&&\OvE_\twX= \opb{\tilde\pi}\Omega_X \tens[{\opb{\tilde\pi}\OO}]\OEn_\twX,
\eneqn
where $\tilde\pi\cl\twX\times\fR\to X$ is the canonical morphism. 
We regard them as objects of $\TDC(\JD_\twX)$ and $\TDC(\JD_\twX^\rop)$, 
respectively. 
Then
\eq
&&\OEn_{\twX}\simeq \Tepb \vpi\OEn_X(*D)
\quad\text{in $\TDC(\II[\vpi^{-1}\D_X])$,}\label{eqOEywX1}\\
&&\Toim{\vpi}\OEn_{\twX}\simeq 
\OEn_X(*D)\quad\text{in $\TDC(\JD_X)$,}\label{eqOEywX2}
\eneq
where 
\eqn
&&\OEn_X(*D)\seteq\OEn_X\Dtens\OO(*D)\simeq
\rihom(\pi^{-1}\C_{X\setminus D},\OEn_X).
\eneqn

Then, for $\shn\in\Derb(\D_{\twX}^\tA)$, we define the enhanced de Rham functor on $\twX$ by
\eqn
\drE_{\twX}(\shn)&=&\OvE_{\twX} \ltens[{\D_{\twX}^\tA}]\shn,\\
\solE_{\twX}(\shn)&=&\rhom[\D_{\twX}^\tA](\shn,\OEn_{\twX}).
\glossary{$\drE_{\twX}$}%
\glossary{$\solE_{\twX}$}%
\eneqn
Then \eqref{eqOEywX1} and \eqref{eqOEywX2}
imply that
\eq
&&\drE_{\twX}(\shm^\tA)\simeq \Tepb \vpi\drE_{X}(\shm(*D))\quad\text{in 
$\TDC(\iC_\twX)$,}\\
&&\Toim{\vpi}\drE_{\twX}(\shm^\tA)\simeq \drE_{X}(\shm(*D))
\quad\text{in $\TDC(\iC_X)$.}
\eneq
for any $\shm\in\Derb(\D_{X})$.

\subsection{De Rham functor: constructibility and duality}\label{subsec:dual}

\begin{theorem}\label{th:newconstruct2}
Let $\shm\in\Derb_\hol(\D_X)$. Then 
$\drE_X(\shm)$ and $\solE_X(\shm)$ belong to $\TDC_\Rc(\iC_X)$.
\end{theorem}
\Proof[Sketch of the proof]
Using Lemma~\ref{lem:redux}, 
one reduces the proof to the case where $\shm$ has a normal form along a normal crossing divisor $D$.
Let $\vpi\cl \twX\to X$ be the real blow up along $D$.

Then, $\shm^\tA$ is locally isomorphic to a direct sum of
$(\she_{U\setminus D|U}^{\varphi})^\tA$ with
$\varphi\in\sect(U;\sho_X(*D))$.
Since Proposition~\ref{pro:Solphi2} implies that
$\drE_{\twX}\bl(\she_{U\setminus D|U}^{\varphi})^\tA\br\simeq 
\Tepb \vpi\drE_{X}(\she_{U\setminus D|U}^{\varphi})$
is $\R$-constructible, $\drE_{\twX}(\shm^\tA)$ is $\R$-constructible.
Hence $\drE_{X}(\shm)\simeq \Toim{\vpi}\drE_{\twX}(\shm^\tA)$
is $\R$-constructible.
\QED

\begin{corollary}\label{cor:newconstruct}
For any $\shm\in\Derb_\hol(\D_X)$ and $F\in\Derb_\Rc(\C_X)$,
the object $\rhom[\shd_X]\bl\shm,\rhom[\iC_X](F,\Ot)\br$ belongs to $\Derb_\Rc(\C_X)$.
\end{corollary}
\begin{proof}
First note that
$$\rhom[\shd_X]\bl\shm,\rhom[\iC_X](F,\Ot)\br
\simeq \rhom[\iC_X]\bl F, \rhom[\shd_X](\shm,\Ot)\br.$$
By Theorems~\ref{th:newconstruct} and~\ref{th:newconstruct2},  
  $\fhom\bl\opb{\pi}F\tens\C_X^E,\solE_X(\shm)\br$ belongs to $\Derb_\Rc(\C_X)$.
Since $\fihom(\C^\enh_X,\OEn_X)\simeq\Ot$ by~\eqref{eq:oEot}, we get
\eqn
\fihom(\C^\enh_X,\solE_X(\shm))&\simeq&\rhom[\shd_X](\shm,\Ot),
\eneqn
and
\begin{align*}
\fhom\bl\opb{\pi}F\tens\C_X^E,\solE_X(\shm)\br
&\simeq \rhom[\iC_X]\bl F,\fihom(\C_X^E,\solE_X(\shm)\br\\
&\simeq \rhom[\shd_X]\bl\shm,\rhom[\iC_X](F,\Ot)\br.
\end{align*}
\end{proof}

\Lemma\label{lem:DEprod}
Let $X_1$ and $X_2$ be a pair of complex manifolds.
Let $\shm_j\in \Derb_\hol(\D_{X_j})$ $(j=1,2)$.
Then we have a canonical isomorphism
\eq
&&\drE_{X_1}(\shm_1)\cetens\drE_{X_2}(\shm_2)\isoto\drE_{X_1\times X_2}(\shm_1\Detens\shm_2).
\eneq
\enlemma
\begin{proof}[Sketch of the proof.]
Using Lemma~\ref{lem:redux}, one reduces the proof to the case where 
$\shm_1$ and $\shm_2$ are exponential D-modules. 
In this case, the result follows from Proposition~\ref{pro:Solphi2}.
\end{proof}

By using the functorial properties of the enhanced de Rham functor 
proved above, we can show that the enhanced de Rham functor 
commutes with duality.

\begin{theorem}\label{th:holconst}
Let $\shm\in\Derb_\hol(\D_X)$. 
Then, we have the isomorphism
\eqn
&&\drE_X(\Ddual_X\shm)\simeq\Edual_X\drE_X(\shm).
\eneqn
\end{theorem}
Note that
 $\drE_X(\Ddual_X\shm)\simeq\solE_X(\shm)\,[d_X]$.

\medskip
\noi
{\it Idea of the proof.}

\smallskip
Let $\sht$ be a monoidal category  with $\one$ as a unit object. 
Recall that a pair of objects $X$ and $Y$ are dual if and only if 
there exist morphisms
\eqn
X\tens Y&\To[\;\eps\;]&\one,\\
\one&\To[\;\eta\;]&Y\tens X
\eneqn
such that 
the composition
$$X\To[X\tens\;\eta] X\tens Y\tens X\To[\eps\;\tens\, X] X$$
is equal to $\id_X$ and
$$Y\To[\eta\;\tens\, Y] Y\tens X\tens Y\To[Y\tens\;\eps]Y$$
is equal to $\id_Y$.

This criterion of duality  has many variations.

\medskip
\begin{description}
 \item[{Sheaf case:}] Let $M$ be a real analytic manifold, and
let $F$, $G\in \Derb_\Rc(\cor_M)$.  Denote by $\Delta_M$ 
\glossary{$\Delta_M$}%
the diagonal subset of $M\times M$.

Now $F$ and $G$ are dual to each other, i.e., $G\simeq\RD_MF$,
if and only if there exist morphisms
\eqn
F\etens G&\To[\ \eps\ ]&\omega_{\Delta_M},\\
\cor_{\Delta_M}&\To[\ \eta\ ]&G\etens F
\eneqn
such that 
the composition
$$F\etens \cor_{\Delta_M}\To[F\etens\;\eta] F\etens G\etens F\To[\eps\;\etens\, F] \omega_{\Delta_M}\etens F$$
is equal to $\id_F$ via isomorphism \eqref{eq:FG1} below  and
$$\cor_{\Delta_M}\etens G\To[\eta\;\etens\, G] G\etens F\etens G\To[G\etens\;\eps]G\etens  \omega_{\Delta_M}$$
is equal to $\id_G$ via isomorphism \eqref{eq:FG2} below.

\item[{Enhanced indsheaf case:}]Let $F$ and $G\in \TDC_\Rc(\icor_M)$.
They are dual to each other, i.e., $G\simeq\Edual_M F$,
if and only if there exist morphisms
\eq&&
\ba{rcl}F\cetens G&\To[\ \eps\ ]&\omega^\enh_{\Delta_M},\\[1ex]
\cor^\enh_{\Delta_M}&\To[\ \eta\ ]&G\cetens F
\ea
\label{mor:Edual}
\eneq
such that 
the composition
\eq
&&F\cetens \cor^\enh_{\Delta_M}\To[F\cetens\;\eta] F\cetens G\cetens F
\To[\eps\;\cetens\, F] \omega^\enh_{\Delta_M}\cetens F
\label{comp:Edual1}\eneq
is equal to $\id_F$ via the enhanced version of isomorphism \eqref{eq:FG1}
 below  and
\eq
&&\cor_{^\enh\Delta_M}\cetens G\To[\eta\;\cetens\, G] 
G\cetens F\cetens G\To[G\cetens\;\eps]G\cetens  \omega^\enh_{\Delta_M}
\label{comp:Edual2}\eneq
is equal to $\id_G$ via the enhanced version of isomorphism \eqref{eq:FG2}
below. 

\item[{Holonomic  D-module case:}]
Let $X$ be a complex manifold and
let $\delta\cl X\into X\times X$ be the diagonal embedding.
\glossary{$\shb_{\Delta_X}$}%
We set $\shb_{\Delta_X}\seteq\Doim \delta\OO$.

Let $\shm$, $\shn\in \Derb_\hol(\D_X)$.
They are dual to each other, i.e., $\shn\simeq\Ddual_X \shm$,
if and only if there exist morphisms
\eq
&&\ba{rcl}
\shm\Detens \shn&\To[\ \eps\ ]&\shb_{\Delta_X}[d_X],\\[1ex]
\shb_{\Delta_X}[-d_X]&\To[\ \eta\ ]&\shn\Detens \shm
\ea\label{mor:Ddual}
\eneq
such that the composition
\eq &&\hs{3ex}\shm\Detens \shb_{\Delta_X}[-d_X]\To[\shm\Detens\;\eta] 
\shm\Detens \shn\Detens \shm
\To[\eps\;\Detens\, \shm] \shb_{\Delta_X}[d_X]\Detens \shm
\label{comp:Ddula1}
\eneq
is equal to $\id_\shm$ via isomorphism \eqref{eq:DFG1} below and
\eq&&\hs{3ex}\shb_{\Delta_X}[-d_X]\Detens \shn\To[\eta\;\Detens\, \shn] 
\shn\Detens \shm\Detens \shm\To[\shn\Detens\;\eps]
\shn\Detens\shb_{\Delta_X}[d_X]
\label{comp:Ddula2}  \eneq
is equal to $\id_\shn$ via isomorphism \eqref{eq:DFG2}  below. 
\end{description}

Now we shall prove Theorem~\ref{th:holconst}.
Set  $\shn=\Ddual_X \shm$.
Then we have morphisms as in \eqref{mor:Ddual} which satisfy
the  conditions
that the compositions \eqref{comp:Ddula1} and \eqref{comp:Ddula2}
are equal to  $\id_\shm$ and $\id_\shn$, respectively.
Now we shall apply the functor
$\drE$. 
Then we obtain morphisms as in \eqref{mor:Edual}
with $M=X_\R$, $\cor=\C$, $F=\drE_X(\shm)$ and $G=\drE_X(\shn)$.
Note that we have
$$\drE_{X\times X}(\shb_{\Delta_X}[-d_X])\simeq\C^\enh_{\Delta_X},\quad
\drE_{X\times X}(\shb_{\Delta_X}[d_X])\simeq\omega^\enh_{\Delta_X}.
$$
By applying the functor $\drE_{X\times X\times X}$, the morphisms  in
\eqref{comp:Ddula1} and \eqref{comp:Ddula2}
are sent to \eqref{comp:Edual1} and \eqref{comp:Edual2}.
Hence the compositions \eqref{comp:Edual1} and \eqref{comp:Edual2}
are equal to  $\id_F$ and $\id_G$, respectively.
Thus we conclude that $G\simeq\Edual_X F$.
\hfill\qedsymbol

\bigskip
Here is the lemma that we used in the course of the proof of 
Theorem~\ref{th:holconst}.
\Lemma  Let $M$ be a real manifold and  let $F,G \in\Derb(\cor_M)$.
Then we have the isomorphisms
\eq
&&\Hom[\Derb(\cor_{M\times M\times M})]
(F\etens \cor_{\Delta_M},\;\omega_{\Delta_M}\etens G)
\simeq \Hom[\Derb(\cor_M)](F,G),\label{eq:FG1}\\
&&\Hom[\Derb(\cor_{M\times M\times M})]
 (\cor_{\Delta_M}\etens F,\; G\etens\omega_{\Delta_M})
\simeq \Hom[\Derb(\cor_M)](F,G),\label{eq:FG2}
\eneq
where $\Delta_M\subset M\times M$ is the diagonal subset.
\enlemma
\Proof
Define the maps $p_{i_1,\ldots, i_n}$ by
$p_{i_1,\ldots, i_n}(x_1,\ldots,x_m)=(x_{i_1},\ldots,x_{i_n})$.
Then we have a commutative diagram
$$\xymatrix@C=8ex{
M\ar[d]^-{\delta}\ar[r]^-{\delta}\ar@{}[dr]|(.5){\square}&M\times M\ar[d]^-{p_{1,2,2}}\\
M\times M\ar[r]_-{p_{1,1,2}}\ar[d]^-{p_2}&M\times M\times M\\
M.}
$$
In the sequel, we write  for short $\Hom$ instead of $\Hom[\cor_N]$  with
$N=M, M\times M\times M$.
Then we have
\begin{align*}
\Hom(F\etens \cor_{\Delta_M},\;\omega_{\Delta_M}\etens G)
&\simeq \Hom(\reim{p_{1,2,2}}(F\etens \cor_M),\roim{p_{1,1,2}}\epb{p_2}G)\\
&\simeq \Hom(F\etens \cor_M,\epb{p_{1,2,2}}\roim{p_{1,1,2}}\epb{p_2}G)\\
&\simeq \Hom(F\etens \cor_M,\roim{\delta}\epb{\delta}\epb{p_2}G)\\
&\simeq \Hom(\opb{\delta}(F\etens \cor_M),\epb{\delta}\epb{p_2}G)\\
&\simeq \Hom(F,G).
\end{align*}
\QED
Similarly, we have the following D-module version.
Here again, we write  for short $\Hom$ instead of $\Hom[\D_Y]$ with
$Y=X$, $X\times X\times X$. 

\Lemma Let $X$ be a complex manifold and 
let $\shm,\shn \in\Derb_\hol(\D_X)$.
Then we have the isomorphisms
\eq
&&\hs{0ex}\Hom
(\shm\Detens \shb_{\Delta_X}[-d_X],\;\shb_{\Delta_X}[d_X]\Detens \shn)
\simeq \Hom(\shm,\shn),\label{eq:DFG1}\\
&&\hs{0ex}\Hom
 (\shb_{\Delta_M}[-d_X]\Detens \shm,\; \shn\Detens\shb_{\Delta_X}[d_X])
\simeq \Hom(\shm,\shn).\label{eq:DFG2}
\eneq
\enlemma

As  applications of Theorem~\ref {th:holconst}, we obtain 
the following corollaries.

\Prop
Let $f\cl X\to Y$ be a morphism of complex manifolds.
Then, for any $\shn\in\Derb_\hol(\D_Y)$, 
\eq
\solE_X(\Dopb{f}\shn)\simeq\Topb{f}\solE_Y(\shn).\label{eq:inversesol}
\eneq
\enprop
\Proof
We have 
\eqn
\solE_X(\Dopb{f}\shn)&\simeq&
\Edual_X\drE_X(\Dopb{f}\shn)[-d_X]\\
&\simeq&\Edual_X\Tepb{f}\drE_Y(\shn)[-d_Y]\\
&\simeq&\Topb{f}\Edual_X\drE_Y(\shn)[-d_Y]\\
&\simeq&\Topb{f}\solE_Y(\shn).
\eneqn
\QED

\Cor Let $X$ be a complex manifold and $\shm,\shn\in \Derb_\hol(\D_X)$.
Then we have the isomorphisms
\eq
&&\drE_X(\shm\Dtens\shn)\simeq
\Tepb{\delta}\bl\drE_X(\shm)\cetens\drE_X(\shn)\br[d_X],\\
&&\solE_X(\shm\Dtens\shn)\simeq\solE_X(\shm)\ctens\solE_X(\shn),
\eneq
where $\delta\cl X\to X\times X$ is the diagonal embedding.
\encor
\Proof
Since $\shm\Dtens\shn\simeq\Dopb{\delta}(\shm\Detens\shn)$,
it is enough to apply \eqref{eq:invdr} and \eqref{eq:inversesol}.
\QED

\Cor \label{cor: *D}For a closed hypersurface $Y$ of a complex manifold $X$ and
$\shm\in\Derb_\hol(\OO)$, we have
$$\solE_X\bl\shm(*Y)\br\simeq \pi^{-1}\C_{X\setminus Y}\tens \solE_X(\shm).$$
\encor
\Proof
It follows from Theorem~\ref{th:holconst} and isomorphisms
$$\drE_X(\shm(*Y))\simeq \rihom\bl\pi^{-1}\C_{X\setminus Y},\drE_X(\shm)\br$$
and
$$\Edual_X\bl\rihom(\pi^{-1}\C_{X\setminus Y},\drE_X(\shm)\br
\simeq\pi^{-1}\C_{X\setminus Y}\tens \Edual_X\drE_X(\shm)$$
(see Theorem~\ref{th:const} \eqref{item:Dihom}).
\QED
\Cor\label{cor:solE} For a closed hypersurface $Y$ of a complex manifold $X$ and
$\vphi\in\OO(*Y)$, we have
$$\solE_X(\she_{X\setminus Y|X}^{\varphi})\simeq 
\C_X^\enh\ctens \C_{\setp{t=-\Re\vphi}}.$$
\encor
This follows from Proposition~\ref{pro:Solphi2}, 
\eqref{eq:dualtwo} and Theorem ~\ref{th:holconst},
because one has  $\she_{X\setminus Y|X}^{\varphi}\simeq
\Bigl(\Ddual_X\she_{X\setminus Y|X}^{-\varphi}\Bigr)(*Y)$.

\subsection{Enhanced Riemann-Hilbert correspondence}\label{subsection:irregholRH}
The following theorem is the main theorem.
\index{Riemann-Hilbert correspondence!generalized}%
\begin{theorem}[{Generalized Riemann-Hilbert correspondence}]\label{th:irrRH1}
There exists a canonical isomorphism functorial with respect to $\shm\in\Derb_\hol(\D_X)${\rm:}
\eq\label{eq:bjorkmorf}
&&\shm\Dtens\OEn_X\isoto\cihom(\solE_X(\shm),\OEn_X)\mbox{ in }\TDC(\iD_X).
\eneq
\end{theorem}

The proof is parallel with the one of Theorem~\ref{thm:ifunct4}
by reducing  the problem to the case where
$\shm$ is an exponential $\mathrm D$-module.
However, in this case, we 
can treat $\drE_{X}(\shm)$ by Proposition~\ref{pro:Solphi2},
but not $\solE_X(\shm)$.
In order to calculate it, we need the commutativity of the enhanced de Rham functor and the duality functor (see Theorem~\ref{th:holconst} 
and its consequence Corollary~\ref{cor: *D}).
\begin{proof}[Sketch of the proof of Theorem \ref{th:irrRH1}]
First we shall construct a morphism \eqref{eq:bjorkmorf}.
We have a canonical morphism 
$$\OEn_X\ctens[\beta\OO]\OEn_X\To\OEn_X.$$
Hence we have 
\eqn
&&(\shm\Dtens\OEn_X)\ctens\solE_X(\shm)
\To\OEn_X\ctens[\beta\OO]\OEn_X\To\OEn_X
\eneqn
which induces a morphism
\eq
\shm\Dtens\OEn_X\To\cihom(\solE_X(\shm),\OEn_X).
\label{eq:eRH}
\eneq

In order to see that it is an isomorphism, 
we shall apply Lemma~\ref{lem:redux}, 
where $P_X(\shm)$ is the statement that \eqref{eq:eRH}
is an isomorphism. 

We shall only check property (f) of this lemma.
 Hence, we assume that 
$\shm$ has a normal form along a normal crossing divisor $D$.
Then we have
  $\solE_X(\shm)\simeq\opb{\pi}\C_{X\setminus D}\tens \solE_X(\shm)$
by Corollary~\ref{cor: *D}, which implies that
$$\cihom(\solE_X(\shm),\OEn_X)\simeq\cihom\bl\solE_X(\shm),\OEn_X(*D)\br.$$
Let $\vpi\cl\twX\to X$ be the real blow up of $X$ along $D$.
Then we have
$$\shm\Dtens\OEn_X\simeq\Toim{\vpi}(\shm^\tA\ltens[\At]\OEn_\twX)$$
and 
$$\cihom\bl\solE_X(\shm),\OEn_X(*D)\br\simeq
\Toim{\vpi}\cihom(\solE_\twX(\shm^\tA),\OEn_\twX).$$
Hence it is enough to show that
\eq
\shm^\tA\ltens[\At]\OEn_\twX\To\cihom(\solE_\twX(\shm^\tA),\OEn_\twX)
\label{eq:twXsol}
\eneq
is an isomorphism.

Since the question is local and $\shm^\tA$ is locally isomorphic to
a direct sum of exponential D-modules $\bl\she_{X\setminus D|X}^{\varphi}\br^\tA$
with $\vphi\in\OO(*D)$,
we may assume that
$\shm^\tA=\bl\she_{X\setminus D|X}^{\varphi}\br^\tA$.
Since \eqref{eq:twXsol} is the image of \eqref{eq:eRH}
by the functor $\Tepb \vpi$,
it is enough to show that \eqref{eq:eRH} is an isomorphism
when $\shm=\she_{X\setminus D|X}^{\varphi}$.

In this case, Corollary~\ref{cor:solE} implies that
$$\solE_X(\shm)\simeq 
\C_X^\enh\ctens\C_{\{t=-\Re\varphi\}},$$
and we can easily see that \eqref{eq:eRH} is an isomorphism.
\end{proof}

\begin{corollary}\label{cor:irregRH1}
There exists a canonical isomorphism functorial with respect to $\shm\in\Derb_\hol(\D_X)${\rm:}
\eq\label{eq:bjorkmorf1}
&&\shm\Dtens\Ot\isoto\fihom(\solE_X(\shm),\OEn_X)\mbox{ in }\Derb(\iD_X).
\eneq
\end{corollary}
\begin{proof}
Let us apply the functor $\fihom(\C^\enh_X,\scbul)$ to the isomorphism~\eqref{eq:bjorkmorf}.
Since $\fihom(\C^\enh_X,\OEn_X)\simeq\Ot$ by~\eqref{eq:oEot}, we get
\eqn
\fihom(\C^\enh_X,\shm\Dtens\OEn_X)&\simeq&\shm\Dtens\Ot.
\eneqn
On the other-hand, we have
\eqn
&&\fihom(\C^\enh_X,\cihom(\solE_X(\shm),\OEn_X))\\
&&\hspace{10.5ex}\simeq\fihom(\solE_X(\shm),\cihom(\C^\enh_X,\OEn_X))\\
&&\hspace{10.5ex}\simeq\fihom(\solE_X(\shm),\OEn_X).
\eneqn
\end{proof}
\index{Riemann-Hilbert correspondence!enhanced}%
\begin{corollary}[{Enhanced Riemann-Hilbert correspondence}]\label{cor:irregRH2}
There exists a canonical isomorphism functorial with respect to 
$\shm\in\Derb_\hol(\D_X)${\rm:}
\eq\label{eq:bjorkmorf2}
&&\shm\isoto\fhom(\solE_X(\shm),\OEn_X)\ \text{in $\Derb(\D_X)$.}
\eneq
\end{corollary}
\begin{proof}
Apply the functor $\alpha_X$ to~\eqref{eq:bjorkmorf1}.
\end{proof}

By Corollary \ref{cor:irregRH2},
we can show the following full faithfulness of the enhanced de Rham
functor.
\Th\label{th; emb}
For $\shm,\shn\in \Derb_\hol(\D_X)$,
one has an isomorphism
\eqn
&&\rhom[\D_X](\shm,\shn)\isoto \fhom(\drE_{X}\shm,\drE_{X}\shn).
\eneqn
In particular, the functor
\eqn
\drE_{X}:\Derb_\hol(\D_X)\To\TDC_\Rc(\iC_X)
\eneqn
is fully faithful.
\enth
\Proof
By Theorem~\ref{th:holconst} and Theorem~\ref{th:const} \eqref{item:Dfhom},
we have
$$\fhom(\drE_{X}\shm,\drE_{X}\shn)
\simeq\fhom(\solE_X\shn,\solE_X\shm).$$
Now, we have
\eqn
\fhom(\solE_X\shn,\solE_X\shm)
&\simeq&\fhom\bl\solE_X\shn,\rhom[\D_X](\shm, \OEn_X)\br\\
&\simeq&\rhom[\D_X]\bl\shm,\fhom\bl\solE_X\shn,\OEn_X)\br\\
&\simeq&\rhom[\D_X](\shm,\shn).
\eneqn
Here the last isomorphism follows from Corollary~\ref{cor:irregRH2}.
\QED

\begin{remark}Corollary~\ref{cor:irregRH2} and Theorem~\ref{th; emb} 
due to~\cite[Th.~9.6.1, Th.~9.7.1]{DK13} are a natural formulation of 
the Riemann-Hilbert correspondence for irregular D-modules. 
Theorem~\ref{th:irrRH1} due to~\cite[Th.~4.5]{KS14} 
is a generalization to the irregular case of Theorem~\ref{thm:ifunct4} which is itself a generalization/reformulation 
of a theorem of J-E.~Bj\"ork (\cite{Bj93}). 
\end{remark}

\section{Integral transforms}\label{section:IT}

\subsection{ Integral transforms with irregular kernels}\label{subsection:irregkern}
\index{integral transforms!}%
\begin{theorem}\label{th:opforenhD}
Let $X$ be a complex manifold and let $\shl\in\Derb_\hol(\D_X)$ and 
 $\shm\in\Derb(\D_{X})$.  There is a natural isomorphism
\[
\drE_X(\shl\Dtens\shm) \simeq \cihom(\solE_X(\shl), \drE_X(\shm)).
\]
\end{theorem}
\begin{proof}
By Theorem~\ref{th:irrRH1}, we have an isomorphism  in $\TDC(\iD_X)$:
\eq\label{eq:bjorkmorfbis}
&&\shl\Dtens\OEn_X\isoto\cihom(\solE_X(\shl),\OEn_X).
\eneq
Let us apply $\shm^\mop\ltens[\shd_X]\scbul$ to both sides of \eqref{eq:bjorkmorfbis}. We have
\eqn
\shm^\mop\ltens[\shd_X](\shl\Dtens\OEn_X)&\simeq&(\shm\Dtens\shl)^\mop\ltens[\shd_X]\OEn_X\\
&\simeq&\drE_X(\shm\Dtens\shl),
\eneqn
and
\eqn
\shm^\mop\ltens[\shd_X]\cihom(\solE_X(\shl),\OEn_X)
&\underset{{\rm(a)}}{\simeq}&\cihom(\solE_X(\shl),\shm^\mop\ltens[\shd_X]\OEn_X)\\
&\simeq&\cihom(\solE_X(\shl),\drE_X(\shm)).
\eneqn
(We do not give the proof of  isomorphism (a) and refer to
\cite[Lem.~3.12]{KS14}.)
\end{proof}
Consider morphisms of complex manifolds
\eqn
&&\ba{c}\xymatrix@C=6ex@R=4ex{
&S\ar[ld]_-f\ar[rd]^-g&\\
{X}&&{Y}.
}\ea\eneqn

\begin{notation}
(i) For $\shm\in\Derb_\qgood(\D_{X})$ and $\shl\in\Derb_\qgood(\D_S)$ recall that 
one sets
\eqn
&&\shm\Dconv\shl\eqdot\Doim{g}(\Dopb{f}\shm\Dtens\shl).
\eneqn
(ii) For $L\in \TDC(\iC_S)$,  $F\in \TDC(\iC_{X})$ and $G\in\TDC(\iC_Y)$ one sets
\glossary{$L\econv $}%
\glossary{$\Phi_L^\Tam$}%
\eq\label{eq:Ccconv}
&&\ba{l} L\econv G\eqdot\Teeim{f}(L\ctens\Topb{g}G),\\
\Phi_L^\Tam(G)=  L\econv G,\quad \Psi_L^\Tam(F)=\Toim{g}\cihom(L,\Tepb{f}F).
\ea\eneq
\end{notation}
Note that we have a pair of adjoint functors
\eq\label{eq:phipsiadj}
&&\xymatrix{
\Phi_L^\Tam: \TDC(\iC_Y)\ar@<0.5ex>[r]&\TDC(\iC_X): \Psi_L^\Tam\,.
\ar@<0.5ex>[l]
}\eneq

\begin{theorem}\label{th:7413}
Let $\shm\in\Derb_\qgood(\D_X)$, 
$\shl\in\Derb_\ghol(\D_{S})\seteq \Derb_\hol(\D_{S})\cap\Derb_\good(\D_{S})$ and let  $L\eqdot\solE_{S}(\shl)$. 
 Assume that $\opb{f}\Supp(\shm)\cap\Supp(\shl)$ is proper over $Y$.
Then there is a natural isomorphism in $\TDC(\iC_Y)${\rm:}
\eq\label{eq:7413}
&&\Psi_L^\Tam\bl\drE_X(\shm)\br\,[d_X-d_S]\simeq\drE_Y(\shm\Dconv\shl).
\eneq
\end{theorem}
\begin{proof}
The proof goes as in the regular case (Theorem~\ref{th:7412}) by using Theorems~\ref{thm:Tfunct} and~\ref{th:opforenhD}.
\end{proof}

\begin{corollary}\label{cor:7413}
In the situation of {\rm Theorem~\ref{th:7413}}, let $G\in \TDC(\iC_Y)$. Then there is a natural isomorphism in $\Derb(\C)$
\eqn
&&\FHom(L\econv G,\OvE_{X}\ltens[\D_{X}]\shm)\,[d_X-d_S]\\
&&\hspace{30ex}\simeq\FHom(G,\OvE_{Y}\ltens[\D_{Y}](\shm\Dconv \shl)).
\eneqn
\end{corollary}
\begin{proof}
This follows from
Theorem~\ref{th:7413} and the adjunction~\eqref{eq:phipsiadj}.
\end{proof}
Note that Corollary~\ref{cor:7413} is a generalization  of
\cite[Th.7.4.12]{KS01} to not necessarily regular  holonomic  D-modules. 

\subsection{Enhanced Fourier-Sato transform}
\index{enhanced!Fourier-Sato transform|(}%
\index{Fourier-Sato transform!enhanced|(}%
The results in \S\,\ref{section:enhanced} extend to the case where $M$ is replaced with a 
bordered space $\fM$. 
Thus $\pi$ denotes the projection $\fM\times\fR\to\fM$ and $t$ the coordinate of $\R$. 
One defines $\TDC(\icor_\fM)$ as the quotient triangulated category
$\BDC(\icor_{\fM\times\fR})/\set{K}{\opb\pi \roim\pi K\isoto K}$. 

One defines the functors 
\glossary{$e_\fM$}%
\glossary{$\epsilon_\fM$}%
\eq
&&e_\fM,\epsilon_\fM\cl \BDC(\icor_\fM) \to \TDC(\icor_\fM),\label{eq:eMepsilonM}\\
&&\hs{10ex}e_\fM(F) = \cor_\fM^\Tam\tens\opb\pi F, 
\quad\epsilon_\fM(F) = \cor_{\{t\geq0\}}\tens\opb\pi F.\nn
\eneq
Note that $e_\fM(F)\simeq\cor_\fM^\Tam\ctens\epsilon_\fM(F)$.

Then Proposition~\ref{pro:embed} extends to  bordered spaces. 
\begin{proposition}\label{pro:embedL}
The functors $e_\fM$ and $\epsilon_\fM$  are  fully faithful.
\end{proposition}

Let $\VV$ be a finite-dimensional real  vector space, $\VVd$ its dual. 
Recall that the Fourier-Sato transform is an equivalence of categories 
between conic sheaves on $\VV$ and conic sheaves on $\VV^*$.
 References are made to~\cite{KS90}. 
In~\cite{Ta08}, D.~Tamarkin has extended the Fourier-Sato transform 
to no more conic (usual) sheaves, by adding an extra variable. 
Here we generalize this last transform to enhanced ind-sheaves on 
the bordered space $\fV$.

We set $n=\dim\VV$ and we denote by $\ori_\VV$ the orientation $\cor$-module of $\VV$, {\em i.e.,} 
$\ori_\VV=H^n_c(\VV;\cor_\VV)$.
In this subsection, the base field $\cor$ is an arbitrary field. 
We have a canonical isomorphism $\ori_\VV\simeq\ori_{\VVd}$. 
We denote by $\Delta_\VV$ the diagonal of $\VV\times\VV$. 

We consider the bordered space 
\glossary{$\bV$}%
$\fV=(\VV,\bV)$ where $\bV$ is the projective compactification of 
$\VV$, that is
\eqn
&&\bV=\bl(\VV\oplus\R)\setminus\{0\}\br/\R^\times.
\eneqn

We introduce the kernels 
\glossary{$L_\VV$}%
\glossary{$L^a_\VV$}%
\glossary{$L^a_\VVd$}%
\begin{align}\label{eq:eFT1}
&&\ba{c}
L_\VV\eqdot\cor_{\{t=\langle x,y\rangle\}}\in\TDC(\icor_{\fV\times\fVd}),\\[1ex]
L^a_\VV\eqdot\cor_{\{t=-\langle x,y\rangle\}}\in\TDC(\icor_{\fV\times\fVd}),\\[1ex]
L^a_\VVd\eqdot\cor_{\{t=-\langle x,y\rangle\}}\in\TDC(\icor_{\fVd\times\fV}).
\ea\end{align}
Here, $x$ and $y$ denote points of $\VV$ and $\VVd$, respectively.

\begin{lemma}\label{le:compkernels}
One has isomorphisms in $\TDC(\icor_{\fV\times\fV})$
\eq\label{eq:eFT3}
&&\ba{c}
L_\VV\econv L^a_\VVd\isoto\cor_{\Delta_\VV\times\{t=0\}}\tens\ori_\VV\,[-n],\\
L^a_\VVd\econv L_\VV\isoto\cor_{\Delta_{\VVd}\times\{t=0\}}\tens\ori_\VV\,[-n].
\ea\eneq
\end{lemma}

Now we introduce the enhanced Fourier-Sato functors 
\glossary{$\EF[\VV]$}%
\glossary{$\EFa[\VVd]$}%
\begin{align}\label{eq:fourier}
&&\ba{c}
\EF[\VV]\cl \TDC(\icor_{\fV})\to\TDC(\icor_{\fVd}),
\quad \EF[\VV](F)=F\econv L_\VV,\\
\EFa[\VVd]\cl \TDC(\icor_{\fVd})\to\TDC(\icor_{\fV}),\quad \EFa[\VVd](F)=F\econv L^a_\VVd. 
\ea
\end{align}
Applying Lemma~\ref{le:compkernels}, we obtain:

\begin{theorem}[{See~\cite{Ta08}}]\label{th:fourier}
The functors $\EF[\VV]$ and $\EFa[\VVd]\tens\ori_\VV\,[n]$ are equivalences of categories, inverse to each other. In other words, one has the isomorphisms, functorial with respect to $F\in \TDC(\icor_{\fV})$
and $G\in \TDC(\icor_{\fVd})${\rm:}
\eqn
&&\EFa[\VVd]\circ\EF[\VV](F)\simeq F\tens\ori_\VV\,[-n],\\
&&\EF[\VV]\circ\EFa[\VVd](G)\simeq G\tens\ori_\VV\,[-n].
\eneqn
\end{theorem}

\begin{corollary}\label{cor:fourieradjoint}
The two functors $\EF[\VV](\scbul)$ and  $\Psi^\Tam_{L^a_\VV}(\scbul)\tens\ori_\VV\,[-n]$ are isomorphic.
\end{corollary}

\begin{corollary}\label{cor:fourier}
There is an isomorphism functorial in $F_1,F_2\in\TDC(\icor_\fV)${\rm:}
\eq\label{eq:isoFourier}
&&\FHom(F_1,F_2)\simeq\FHom(\EF[\VV](F_1),\EF[\VV](F_2)).
\eneq
\end{corollary}

Recall that one denotes 
by $\Derb_{\R^+}(\cor_\VV)$ the full subcategory of $\Derb(\cor_\VV)$ consisting of conic sheaves (see \cite{KS90}).
Here conic sheaves mean sheaves on $\VV$
constant on any half line $\R_{>0} v$ ($v\in\VV\setminus\{0\}$). 
 We shall denote here by $\FS[\VV](F)$ the Fourier-Sato transform of 
$F\in\Derb_{\R^+}(\cor_\VV)$, which was denoted by $F^\wedge$ in loc.\ cit. The functor  
$\FS[\VV]\cl\Derb_{\R^+}(\cor_\VV)\to\Derb_{\R^+}(\cor_\VVd)$ is an equivalence of categories. 

Recall that one identifies the sheaf $\cor_{\{t\geq0\}}$ with its image in $\TDC(\icor_{\VV\times\R_\infty})$
and that the functor
\eqn
&&\epsilon_\fV\cl\BDC(\cor_\VV) \into \TDC(\icor_\fV),\quad \epsilon_\fV(F)= \cor_{\{t\geq0\}}\tens\opb\pi F
\eneqn 
is a fully faithful embedding by Proposition~\ref{pro:embedL}.

Consider the diagram of categories and functors
\eq\label{diag:fourierefourier}
&&\ba{c}\xymatrix{
\Derb_{\R^+}(\cor_\VV)\ar[rr]^-{\FS[\VV]}\ar[d]_-{\epsilon_\fV}
&&\Derb_{\R^+}(\cor_\VVd)\ar[d]_-{\epsilon_\fVd}\\
\TDC(\icor_\fV)\ar[rr]^-{\EF[\VV]}&&\TDC(\icor_\fVd).
}\ea
\eneq
\begin{theorem}\label{th:fourierefourier}
Diagram~\eqref{diag:fourierefourier} is quasi-commutative.
\end{theorem}
\index{enhanced!Fourier-Sato transform|)}%
\index{Fourier-Sato transform!enhanced|)}%

\subsection{Laplace transform}
\index{Laplace!transform}%
In the sequel, we take $\C$ as the base field $\cor$.
Recall the $\D_X$-module $\she^{\varphi}_{U|X}$ and Notation~\ref{not:<?}. 
We saw in Proposition~\ref{pro:Solphi} that
\eq
&&\solE_X(\she^{\varphi}_{U|X})\simeq\C_X^\Tam\ctens\C_{\{t=-\Re\vphi\}}.\label{eq:laplaceEphiB}
\eneq
We shall apply this result in the following situation.

Let $\WW$ be a complex finite-dimensional vector space of complex dimension $d_\WW$, $\WW^*$ its dual. 
Since $\WW$ is a complex vector space, we shall identify $\ori_\WW$ with $\C$. 
We denote here by 
$\bW$ the projective compactification of $\WW$, we set $\BBH=\bW\setminus\WW$,
and similarly with $\bWd$ and $\BBHd$. We also introduce the bordered spaces
\eqn
&&\fW=(\WW,\bW),\quad \fWd=(\WWd,\bWd).
\eneqn
We set for short
\eqn
&&X=\bW\times\bWd,\, U=\WW\times\WWd,\, Y=X\setminus U.
\eneqn
We shall consider the function 
\eqn
&&\phi\cl\WW\times\WWd\to\C,\quad \phi(x,y)=\langle x,y\rangle.
\eneqn

We introduce the Laplace kernel
\index{Laplace!kernel}%
\eq\label{eq:Laplaceker}
\shl&\eqdot&\she^{\langle x,y\rangle}_{U|X}.
\eneq
Recall  
that the kernel of the enhanced Fourier transform   
with respect to the underlying real vector spaces 
of $\WW$ and $\WWd$ is given by 
\eqn
&&L_{\WW}\eqdot \C_{\{t=\Re\langle x,y\rangle\}}\in\TDC(\iC_{\fW\times\fWd}).
\eneqn
Also recall that we set $L_{\WW}^a\eqdot \C_{\{t=-\Re\langle x,y\rangle\}}$.
\begin{lemma}\label{le:laplaceD2}
One has the isomorphism in $\TDC(\iC_X)$
\eq\label{eq:Laplaceiso21}
&&\solE_{X}(\shl)\simeq \C^\Tam_X\ctens \Teeim{j}L^a_{\WW},
\eneq
where $j\cl \fW\times\fWd\to X$ is the inclusion.
\end{lemma}
\begin{proof}
This follows immediately from  isomorphism~\eqref{eq:laplaceEphiB}.
\end{proof}
In the sequel, we denote by 
\glossary{$\Drm_\WW$}%
$\Drm_\WW$ the Weyl algebra
\index{Weyl algebra}%
$\sect\bl\bW; \D_\bW(*\BBH)\br$ associated with 
$\WW$. We also 
use the $(\Drm_{\WW\times\WWd},\Drm_{\WWd})$-bimodule $\Drm_{\WW\times\WWd\to\WWd}$ similar to the bimodule $\D_{X\to Y}$ in  the theory of 
 D-modules, and finally  we denote by 
$\OOO_\WW$ the ring of polynomials on $\WW$.

The next result is well-known and goes back  to~\cite{KL85} or before.
\begin{lemma}\label{le:laplaceD1}
There is a natural isomorphism 
\eq\label{eq:Laplaceiso1}
&&\D_\bW(*\BBH)\Dconv\shl\simeq \D_\bWd(*\BBHd)\tens\det\WWd.
\eneq
Here, $\det\WWd=\bigwedge^n\WWd$. 
\end{lemma}
\begin{proof}
Using the GAGA principle, we may replace 
$\D_\bW(*\BBH)$ with $\Drm_\WW$, $\D_\bWd(*\BBH)$ with $\Drm_\WWd$, $\shl$ with $\Drm_{\WW\times\WWd} \e^{\langle x,y\rangle}$
and thus 
$\D_\bW(*\BBH)\Dconv\shl$ with 
\eq\label{eq:KatzL}
&&\Drm_{\WWd\from \WW\times\WWd}\ltens[\Drm_{\WW\times\WWd}](\Drm_{\WW\times\WWd} \e^{\langle x,y\rangle}
\ltens[{\OOO_{\WW\times\WWd}}]\Drm_{\WW\times\WWd\to\WWd}).
\eneq
This last object is isomorphic to 
\eqn
\bl\Drm_{\WWd\from\WW\times\WWd}\ltens[{\OOO_{\WW\times\WWd}}]\Drm_{\WW\times\WWd\to\WW}  \br 
\ltens[\D_{\WW\times\WWd}]\Drm_{\WW\times\WWd} \e^{\langle x,y\rangle}.
\eneqn
Since
\eqn
&&\Drm_{\WWd\from\WW\times\WWd}\ltens[{\OOO_{\WW\times\WWd}}]\Drm_{\WW\times\WWd\to\WW} 
\simeq\Drm_{\WW\times\WWd}\tens\det\WWd,
\eneqn
the module~\eqref{eq:KatzL} is isomorphic to $\Drm_{\WW\times\WWd}\e^{\langle x,y\rangle}\tens\det\WWd$.
Finally, one remarks that the natural morphism $\Drm_{\WWd}\to \Drm_{\WW\times\WWd} \e^{\langle x,y\rangle}$ is an isomorphism.
\end{proof}

In the sequel, we shall identify  $\Drm_\WW$ and $\Drm_\WWd$  
 by the correspondence $x_i\leftrightarrow-\partial_{y_i}$, 
$\partial_{x_i}\leftrightarrow y_i$. (Of course, this does not depend on the choice of linear coordinates on $\WW$ and the dual coordinates on $\WWd$.)

\begin{theorem}\label{th:laplace}
 We have an isomorphism in $\Derb((\iDrm_\WW)_\fWd)$
\eq\label{eq:Laplaceiso22}
&&\EF[\WW](\OEn_\fW)\simeq\OEn_\fWd\tens\det\WW\,[-d_\WW].
\eneq
\end{theorem}
\begin{proof}
Set $K=\solE_{\fW\times\fWd}(\shl)$. By Theorem~\ref{th:7413}, we have
\eqn
&&\Psi_K^\Tam(\drE_\fW(\shm))\,[-d_\WW]\simeq\drE_\fWd(\shm\Dconv\shl).
\eneqn
for any $\shm\in \Derb_\qgood(\D_X)$ such that 
$\shm\simeq\shm(*\BBH)$. 
By Lemma~\ref{le:laplaceD2}, $K=\C^\Tam_{\fW\times\fWd}\ctens L_{\WW}^a$, 
and by Corollary~\ref{cor:fourieradjoint}, the functor $\EF[\WW]$ is isomorphic to the functor 
$\Psi^\Tam_{L^a_\WW}\,[-2d_\WW]$. 
Since 
\eqn
&&\cihom\bl\C^\Tam_{\fW},\drE_\fW(\shm)\br\simeq \drE_\fW(\shm),
\eneqn
we have 
\eqn
&&\Psi_K\bl\drE_\fW(\shm)\br[-2d_\WW]
\simeq \Psi_{L^a_\WW}\bl \drE_\fW(\shm)\br[-2d_\WW]
\simeq \EF[\WW]\bl\drE_\fW(\shm)\br.
\eneqn 
Therefore, we obtain 
\eqn
&&\EF[\WW]\bl\drE_\fW(\shm)\br\simeq \drE_\fWd(\shm\Dconv\shl)\,[-d_\WW].
\eneqn
 Now  choose $\shm= \D_\bW(*\BBH)$ and  apply Lemma~\ref{le:laplaceD1}.
Since $\drE_\fW(\shm)\simeq\OvE_\fW$ and 
$\drE_\fWd(\shm\Dconv\shl)\simeq\OvE_\fWd\tens\det\WWd$, we obtain 
\eqn
&&\EF[\WW](\OvE_\fW)\simeq\OvE_\fWd\tens\det\WWd\,[-d_\WW].
\eneqn
Hence, it is enough to remark that
\eqn
&&\OvE_\fW\simeq\OEn_\fW\tens\det\WWd\mbox{ and } \OvE_\fWd\simeq\OEn_\fWd\tens\det\WW.
\eneqn
\end{proof}

\begin{remark}
\bnum
\item Symbolically,  isomorphism~\eqref{eq:Laplaceiso22} 
is given by
$$\OEn_\fWd\tens \det\WW\ni \phi(y)\tens dy
\longmapsto \int \e^{\lan x,y\ran}\phi(y)dy\in \EF[\WW](\OEn_\fW).$$
\item
The identification of $\Drm_\WW$ and $\Drm_\WWd$ 
is given by:
\begin{align*}
\hs{-4ex}\Drm_\WW\ni P(x,\partial_x)
\leftrightarrow Q(y,\partial_y)\in\Drm_\WWd
&\Longleftrightarrow
P(x,\partial_x)\e^{\lan x,y\ran}=Q^*(y,\partial_y)\e^{\lan x,y\ran}\\[1ex]
&\Longleftrightarrow P^*(x,\partial_x)\e^{-\lan x,y\ran}=Q(y,\partial_y)\e^{-\lan x,y\ran}.
\end{align*}
Here $Q^*(y,\partial_y)$ denotes the formal adjoint operator
of $Q(y,\partial_y)\in\Drm_\WWd$.
\enum
\end{remark}
Applying Corollary~\ref{cor:fourier}, we get:

\begin{corollary}\label{cor:laplace}
Isomorphism~\eqref{eq:Laplaceiso1}  together with the enhanced Fourier-Sato isomorphism induce an isomorphism in $\Derb(\Drm_\WW)$, 
functorial in $F\in\TDC(\iC_{\fW})${\rm:}
\eq\label{eq:Laplaceiso23} 
&&\FHom(F,\OEn_\fW)\simeq \FHom\bl\EF[\WW](F),\OEn_\fWd\br\tens\det\WW\,[-d_\WW].
\eneq
\end{corollary}

As a consequence of Corollary~\ref{cor:laplace}, we recover the main result of~\cite{KS97}: 
\begin{corollary}\label{cor:laplace2}
Isomorphism~\eqref{eq:Laplaceiso1}  together with the  Fourier-Sato isomorphism  induces  an isomorphism in $\Derb(\Drm_\WW)$, 
functorial in  $G\in\Derb_{\R^+}(\C_{\WW})${\rm:} 
\eq\label{eq:Laplaceiso3} 
&&\RHom(G,\Ot[\fW])\simeq \RHom\bl\FS_\WW(G),\Ot[\fWd]\br\tens\det\WW\,[-d_\WW].
\eneq
Here, letting
$a_\fW\cl \fW\to \rmpt$ be the projection,
$$\RHom(G,\Ot[{\fW}])\seteq\alpha_\rmpt\roim{a_\fV}\rihom[\iC_\fW](G,\Ot[{\fW}])
\in\Derb(\C).$$ 
\end{corollary}
\begin{proof} 
By Theorem~\ref{th:fourierefourier}, we have 
 $\EF[\WW](\epsilon_\fW(G))\simeq\epsilon_\fWd\FS[\WW](G)$, where $\epsilon_\fW$ is given in~\eqref{eq:eMepsilonM}.
Applying  isomorphism~\eqref{eq:Laplaceiso23}  with  $F=\epsilon_\fW(G)$, we obtain
\eqn
&&\FHom(\epsilon_\fW(G),\OEn_\fW)\simeq \FHom\bl\epsilon_\fWd\FS[\WW](G),\OEn_\fWd\br\tens\det\WW\,[-d_\WW].
\eneqn
We have 
\eqn
\FHom(\epsilon_\fW(G),\OEn_\fW)&\simeq&\alpha_\rmpt\roim{a_\fW}
\fihom(\epsilon_\fW(G),\OEn_\fW)\\
&\simeq&\alpha_\rmpt\roim{a_\fW}\rihom(G,\Ot[\fW])\\
&\simeq&\RHom(G,\Ot[\fW]),
\eneqn
and similarly 
$
\FHom(\epsilon_\fWd\FS[\WW](G),\OEn_\fWd)\simeq\RHom(\FS[\WW](G),\Ot[\fWd])$. 
\end{proof}

\providecommand{\bysame}{\leavevmode\hbox to3em{\hrulefill}\thinspace}

\newpage
\markboth{List of Notations}{List of Notations}
\renewcommand{\indexname}{List of Notations}
\addcontentsline{toc}{section}{\numberline{}List of Notations}
\begin{theindex}

  \item $(M,\bM)$, 56
  \item $E^\vphi_{{U\vert X}}$, 78
  \item $F^\Tam$, 71
  \item $L\conv$, 54
  \item $L\econv $, 98
  \item $L^a_\VV$, 99
  \item $L^a_\VVd$, 99
  \item $L_\VV$, 99
  \item $\At$, 47
  \item $\BDC_\Rc(\cor_{M\times\R_\infty})$, 73
  \item $\BDC_\coh(\D_X)$, 11
  \item $\BDC_\good(\D_X)$, 11
  \item $\BDC_\hol(\D_X)$, 11
  \item $\BDC_\qgood(\D_X)$, 11
  \item $\C_{\{\Re \varphi <\ast\}}$, 78
  \item $\Cinf$, 29
  \item $\Cinfom$, 29
  \item $\Cinft[M]$, 30
  \item $\Cinft[\Msa]$, 30
  \item $\Cinfw[M]$, 30
  \item $\Cinfw[\Msa]$, 30
  \item $\DA_\twX$, 47
  \item $\DFN$, 35
  \item $\D_X \ex^\varphi$, 77
  \item $\D_X$, 10
  \item $\D_{M\to N}$, 31
  \item $\D_{X\from Y}$, 12
  \item $\D_{X\to Y}$, 12
  \item $\DbT_M$, 78
  \item $\Db_M$, 29
  \item $\Dbt_M$, 30
  \item $\Dbt_{\Msa}$, 30
  \item $\Dbtv_M$, 31
  \item $\Dconv$, 54
  \item $\Ddual$, 11
  \item $\Deim{f}$, 12
  \item $\Delta_M$, 90
  \item $\Derb(\cor_M)$, 7
  \item $\Derb(\icor_{(M,\bM)})$, 57
  \item $\Derb_\Irc(\icor_M)$, 27
  \item $\Derb_\Rc(\cor_M)$, 9
  \item $\Derb_{\reghol}(\D_X)$, 44
  \item $\Detens$, 12
  \item $\Doim{f}$, 12
  \item $\Dopb{f}$, 12
  \item $\Drm_\WW$, 101
  \item $\Dtens$, 11, 77
  \item $\EF[\VV]$, 99
  \item $\EFa[\VVd]$, 99
  \item $\Edual_M$, 73
  \item $\FHom$, 67
  \item $\FN$, 35
  \item $\II[\A]$, 21
  \item $\II[\cor]$, 16
  \item $\II[\cor_M]$, 17
  \item $\IIrc[\cor_M]$, 25
  \item $\Msa$, 24
  \item $\OEn_X$, 79
  \item $\OO(*S)$, 42
  \item $\Omega_X$, 10
  \item $\Omega_{X/Y}$, 10
  \item $\Op_M$, 24
  \item $\Op_\Msa$, 24
  \item $\Ot[X]$, 34
  \item $\Ot[\twX]$, 47
  \item $\OvE_X$, 79
  \item $\Ovt_X$, 34
  \item $\Ovt_\twX$, 48
  \item $\Ow[X]$, 34
  \item $\Oww[X]$, 34
  \item $\Phi_L$, 54
  \item $\Phi_L^\Tam$, 98
  \item $\Psi_L$, 54
  \item $\RD'_M$, 8
  \item $\RD_M$, 8
  \item $\R_\infty$, 62
  \item $\SI[\cor_M]$, 17
  \item $\SSupp(\shm)$, 44
  \item $\Supp(F)$, 7
  \item $\TDC(\cor_{M})$, 66
  \item $\TDC(\icor_M)$, 65
  \item $\TDCpm(\icor_M)$, 65
  \item $\TDCrc(\icor_M)$, 73
  \item $\Teeim{f}$, 69
  \item $\Tepb{f}$, 69
  \item $\Theta_M$, 31
  \item $\Theta_X$, 10
  \item $\Tl$, 67
  \item $\Toim{f}$, 69
  \item $\Topb{f}$, 69
  \item $\Tr$, 67
  \item $\alpha_M$, 17
  \item $\bV$, 99
  \item $\beta_M$, 17
  \item $\cetens$, 69
  \item $\chv(\shm)$, 10
  \item $\cihom$, 63, 68
  \item $\cihom[\beta\A]$, 77
  \item $\cor^\Tam_M$, 71
  \item $\cor_{\{t<\ast\}} $, 71
  \item $\cor_{\{t>*\}}$, 75
  \item $\cor_{\{t\gg0\}}$, 71
  \item $\ctens$, 63, 68
  \item $\ctens[\beta\A]$, 77
  \item $\dim M$, 31
  \item $\drE_X$, 80
  \item $\drE_{\twX}$, 88
  \item $\dr_X$, 13
  \item $\drt_X$, 38
  \item $\epb{\tilde f}$, 69
  \item $\epb{f}$, 8, 20, 58
  \item $\epsilon_M$, 73
  \item $\epsilon_\fM$, 98
  \item $\etens$, 8
  \item $\fhom$, 67
  \item $\fihom$, 67
  \item $\ihom$, 17
  \item $\indc_{t^* = 0} $, 65
  \item $\indc_{t^*\geq 0}$, 65
  \item $\indc_{t^*\leq 0} $, 65
  \item $\indcc$, 15
  \item $\iota_M$, 18
  \item $\md[\D^\rop_X]$, 10
  \item $\md[\D_X]$, 10
  \item $\md[\cor]$, 16
  \item $\md[\cor_M]$, 7
  \item $\mdc[\D_X]$, 10
  \item $\mdcp[\cor_M]$, 17
  \item $\mdf[\cor]$, 16
  \item $\mdrh[\D_X]$, 44
  \item $\mu(t_1,t_2)$, 62
  \item $\ol\R$, 62
  \item $\omega_M$, 8
  \item $\opb{\tilde f}$, 69
  \item $\opb{f}$, 8, 20, 58
  \item $\ori_M$, 31
  \item $\rA$, 31
  \item $\rc[\cor_M]$, 25
  \item $\rcc[\cor_M]$, 25
  \item $\reeim{\tilde f}$, 69
  \item $\reeim{f}$, 20, 58
  \item $\reim{f}$, 8
  \item $\rho_M$, 25
  \item $\rhom$, 8
  \item $\rihom$, 20, 58
  \item $\roim{\tilde f}$, 69
  \item $\roim{f}$, 8, 20, 58
  \item $\shb_M$, 29
  \item $\shb_{\Delta_X}$, 91
  \item $\shc^\wedge$, 15
  \item $\she^\varphi_{U\vert X}$, 77
  \item $\shm^\tA$, 48
  \item $\shv_M$, 31
  \item $\sinddlim$, 15
  \item $\solE_X$, 80
  \item $\solE_{\twX}$, 88
  \item $\sol_X$, 13
  \item $\solt_X$, 38
  \item $\tens$, 8, 17, 20, 58
  \item $\thom$, 31
  \item $\twX$, 47
  \item $\twX^\tot$, 47
  \item $\twX^{>0}$, 47
  \item $\varpi$, 47
  \item $a_M\cl M\to\rmpt$, 8
  \item $d_X$, 10
  \item $e_M$, 73
  \item $e_\fM$, 98

\end{theindex}

\newpage
\markboth{Index}{Index}
\renewcommand{\indexname}{Index}
\addcontentsline{toc}{section}{\numberline{}Index}
\begin{theindex}

  \item base change formula
    \subitem for D-modules, \hyperpage{14}
    \subitem for indsheaves, \hyperpage{23}
    \subitem for sheaves, \hyperpage{9}
  \item bordered spaces, \hyperpage{56}
    \subitem proper morphism of, \hyperpage{60}

  \indexspace

  \item D-module, \hyperpage{10}
    \subitem good, \hyperpage{11}
    \subitem holonomic, \hyperpage{10}
    \subitem quasi-good, \hyperpage{10}
    \subitem regular holonomic, \hyperpage{44}
  \item de Rham functor, \hyperpage{13}
    \subitem enhanced, \hyperpage{80}
    \subitem tempered, \hyperpage{38}
  \item duality
    \subitem for D-modules, \hyperpage{11}
    \subitem for enhanced indsheaves, \hyperpage{73}
  \item dualizing complex, \hyperpage{8}

  \indexspace

  \item enhanced
    \subitem Fourier-Sato transform, \hyperpage{98--100}
    \subitem indsheaves, \hyperpage{65}
    \subitem tempered holomorphic functions, \hyperpage{78}
  \item exponential D-modules, \hyperpage{77}
  \item external tensor product
    \subitem for D-modules, \hyperpage{12}
    \subitem for enhanced indsheaves, \hyperpage{69}
    \subitem for sheaves, \hyperpage{8}

  \indexspace

  \item Fourier-Sato transform
    \subitem enhanced, \hyperpage{98--100}

  \indexspace

  \item good
    \subitem D-module, \hyperpage{11}
    \subitem topological space, \hyperpage{7}
  \item Grothendieck's six operations
    \subitem for enhanced indsheaves, \hyperpage{68}
    \subitem for indsheaves, \hyperpage{19}
    \subitem for sheaves, \hyperpage{8}

  \indexspace

  \item holonomic D-module, \hyperpage{10}

  \indexspace

  \item ind-object, \hyperpage{15}
  \item indsheaves, \hyperpage{17}
    \subitem enhanced, \hyperpage{65}
    \subitem of tempered $\mathrm{C}^\infty$-functions, \hyperpage{30}
    \subitem of tempered distributions, \hyperpage{30}
    \subitem of tempered holomorphic functions, \hyperpage{34}
    \subitem of Whitney functions, \hyperpage{30}
    \subitem of Whitney holomorphic functions, \hyperpage{34}
  \item integral transforms, \hyperpage{97}
    \subitem regular, \hyperpage{54}

  \indexspace

  \item Laplace
    \subitem kernel, \hyperpage{101}
    \subitem transform, \hyperpage{100}

  \indexspace

  \item non-characteristic, \hyperpage{13}
  \item normal form, \hyperpage{85}
    \subitem quasi-, \hyperpage{85}
    \subitem regular, \hyperpage{44}

  \indexspace

  \item polynomial growth, \hyperpage{29}
  \item projection formula
    \subitem for D-modules, \hyperpage{13}
    \subitem for indsheaves, \hyperpage{23}
    \subitem for sheaves, \hyperpage{9}
  \item proper morphism
    \subitem of bordered spaces, \hyperpage{60}

  \indexspace

  \item quasi-good D-module, \hyperpage{10}
  \item quasi-normal form, \hyperpage{85}

  \indexspace

  \item $\R$-constructible
    \subitem enhanced indsheaves, \hyperpage{73}
    \subitem sheaves, \hyperpage{9}
  \item ramification, \hyperpage{85}
  \item real blow up, \hyperpage{46}
  \item regular holonomic D-module, \hyperpage{44}
  \item Riemann-Hilbert correspondence
    \subitem enhanced, \hyperpage{96}
    \subitem generalized, \hyperpage{94}
    \subitem generalized regular, \hyperpage{50}
    \subitem regular, \hyperpage{53}

  \indexspace

  \item solution functor, \hyperpage{13}
    \subitem enhanced, \hyperpage{80}
    \subitem tempered, \hyperpage{38}
  \item stable object, \hyperpage{71}
  \item Stokes matrices, \hyperpage{83}
  \item subanalytic
    \subitem indsheaf, \hyperpage{25}
    \subitem site, \hyperpage{24}
    \subitem space, \hyperpage{24}
    \subitem subset, \hyperpage{9}

  \indexspace

  \item tempered
    \subitem $\mathrm{C}^\infty$-functions, \hyperpage{29}
    \subitem distributions, \hyperpage{29}
    \subitem holomorphic functions, \hyperpage{34}
  \item tempered Grauert theorem, \hyperpage{40}
  \item transfer bimodule, \hyperpage{12}
  \item transversal Cartesian diagram, \hyperpage{14}

  \indexspace

  \item Weyl algebra, \hyperpage{101}
  \item Whitney functions, \hyperpage{30}

\end{theindex}


\markboth{}{}
\noindent
\scalebox{.85}{\parbox[t]{60ex}
{\noindent
Masaki Kashiwara\\
Research Institute for Mathematical Sciences, Kyoto University \\
Kyoto, 606--8502, Japan\\
e-mail: masaki@kurims.kyoto-u.ac.jp

\medskip\noindent
Pierre Schapira\\
Institut de Math{\'e}matiques,
Universit{\'e} Pierre et Marie Curie\\
and Mathematics Research Unit,
University of Luxembourg\\
4 Place Jussieu, 7505 Paris\\
e-mail: schapira@math.jussieu.fr\\
http://www.math.jussieu.fr/\textasciitilde schapira/
}}


\end{document}